\documentclass[11pt,letterpaper]{amsart}
\usepackage[foot]{amsaddr}

\usepackage{mathtools}
\usepackage{amsmath}
\usepackage[T1]{fontenc}
\usepackage[latin9]{inputenc}
\usepackage{geometry}
\usepackage{wrapfig}
\usepackage{multicol}
\usepackage{graphicx}
\usepackage{soul}
\usepackage{xcolor}
\usepackage{amssymb}
\usepackage{placeins}
\usepackage{bbm}
\usepackage{cite}
\usepackage{caption}
\usepackage{enumerate}
\usepackage{afterpage}
\usepackage{enumitem}
\usepackage{bmpsize}
\usepackage{hyperref}
\usepackage{tabu}
\usepackage{enumitem}
\numberwithin{equation}{section}
\usepackage{stmaryrd}
\usepackage{tikz}
\usetikzlibrary{matrix,graphs,arrows,positioning,calc,decorations.markings,shapes.symbols}
\usepackage{ifthen}
\usetikzlibrary{math}
\usetikzlibrary{matrix,graphs,arrows,positioning,calc,decorations.markings,shapes.symbols}

\makeatletter
\renewcommand{\email}[2][]{%
  \ifx\emails\@empty\relax\else{\g@addto@macro\emails{,\space}}\fi%
  \@ifnotempty{#1}{\g@addto@macro\emails{\textrm{(#1)}\space}}%
  \g@addto@macro\emails{#2}%
}
\makeatother

\newtheorem{theorem}{Theorem}[section]
\newtheorem{lemma}[theorem]{Lemma}
\newtheorem{proposition}[theorem]{Proposition}

{ \theoremstyle{definition}
\newtheorem{definition}[theorem]{Definition}}
{ \theoremstyle{remark}
\newtheorem{remark}[theorem]{Remark}}

\newcommand{\weyl}{W^\circ}
\newcommand{\weylc}{\overline{W}}

\newcommand{\im}{\mathsf{i}}
\newcommand{\Real}{\mathrm{Re}\hspace{0.5mm}}

\newcommand{\cev}[1]{\reflectbox{\ensuremath{\vec{\reflectbox{\ensuremath{#1}}}}}}

\newcommand{\hk}{p}
\newcommand{\kcr}{K^{\mathrm{hs}; \varpi}}
\newcommand{\kgse}{K^{\mathrm{GSE}}}
\newcommand{\icr}{I^{\mathrm{hs}; \varpi}}
\newcommand{\rcr}{R^{\mathrm{hs}; \varpi}}
\newcommand{\hsa}{\mathcal{A}^{\mathrm{hs}; \varpi}}
\newcommand{\hsai}{\mathcal{A}^{\mathrm{hs}; \infty}}
\newcommand{\hsail}{\mathcal{L}^{\mathrm{hs}; \infty}}

\newcommand{\hsm}{M^{\mathsf{S};\mathrm{hs};\varpi}}
\newcommand{\hsmi}{M^{\mathsf{S}; \mathrm{hs}; \infty}}
\newcommand{\hsk}{K^{\varpi}}
\newcommand{\hsi}{I^{\varpi}}
\newcommand{\hsr}{R^{\varpi}}

\newcommand{\hski}{K^{\mathrm{hs};\infty}}
\newcommand{\hsii}{I^{\mathrm{hs};\infty}}
\newcommand{\hsri}{R^{\mathrm{hs};\infty}}

\newcommand{\hsl}{\mathcal{L}^{\mathrm{hs}; \varpi}}
\newcommand{\hsli}{\mathcal{L}^{\mathrm{hs}; \infty}}

\newcommand{\ap}{\mathsf{a}}

\newcommand{\pfbm}{\mathbb{P}_{\operatorname{free}}}
\newcommand{\efbm}{\mathbb{E}_{\operatorname{free}}}
\newcommand{\pabm}{\mathbb{P}_{\operatorname{avoid}}}
\newcommand{\eabm}{\mathbb{E}_{\operatorname{avoid}}}

\usepackage{ifthen}

\title{The pinned half-space Airy line ensemble}
\date{\today}
\author{Evgeni Dimitrov} 
\author{Christian Serio}
\author{Zongrui Yang} 

\begin{document}

\begin{abstract}
Half-space models in the Kardar-Parisi-Zhang (KPZ) universality class exhibit rich boundary phenomena that alter the asymptotic behavior familiar from their full-space counterparts. A distinguishing feature of these systems is the presence of a boundary parameter that governs a transition between subcritical, critical, and supercritical regimes, characterized by different scaling exponents and fluctuation statistics. 

In this paper we construct the {\em pinned half-space Airy line ensemble} $\hsai$ on $[0,\infty)$ -- a natural half-space analogue of the Airy line ensemble -- expected to arise as the universal scaling limit of supercritical half-space KPZ models. The ensemble $\hsai$ is obtained as the weak limit of the critical half-space Airy line ensembles $\hsa$ introduced in \cite{DY25} as the boundary parameter $\varpi$ tends to infinity.

We show that $\hsai$ has a Pfaffian point process structure with an explicit correlation kernel and that, after a parabolic shift, it satisfies a one-sided Brownian Gibbs property described by pairwise pinned Brownian motions. Far from the origin, $\hsai$ converges to the standard Airy line ensemble, while at the origin its distribution coincides with that of the ordered eigenvalues (with doubled multiplicity) of the stochastic Airy operator with $\beta = 4$.
\end{abstract}


\maketitle

\tableofcontents

%
%
\section{Introduction and main results}\label{Section1}
Over the past two decades, there has been substantial interest in the asymptotic analysis of {\em half-space} models within the {\em Kardar-Parisi-Zhang (KPZ)} universality class. The earliest studies in this direction are due to Baik and Rains, who analyzed the asymptotics of the longest increasing subsequence of random involutions and of symmetrized last passage percolation (LPP) with geometric weights \cite{BR01a, BR01b, BR01c}. Other half-space models that have been extensively investigated include the polynuclear growth model \cite{SI04}, Schur processes \cite{BR05, BBNV}, LPP with exponential weights \cite{BBCS}, the facilitated (totally) asymmetric simple exclusion process, or (T)ASEP \cite{BBCS2}, the KPZ equation studied via ASEP \cite{CS18, Par19} and through directed polymers \cite{BC23,DasSerio25,Wu20}, the stochastic six-vertex model \cite{BBCW18}, Macdonald processes \cite{BBC20}, and the log-gamma polymer \cite{IMS22, BCD24}.

The local behavior of the half-space KPZ models resembles that of their full-space counterparts, except that the origin -- which should be interpreted as the left boundary of the interval $[0, \infty)$ on which the models are defined -- introduces a boundary effect that nontrivially influences their asymptotics. For instance, in interacting particle systems, the origin acts as a reservoir that stochastically injects or absorbs particles, while in LPP models, the weights at the origin differ from those elsewhere. A central goal in this area is to understand how such boundary effects manifest across various settings, and, in some cases, to show that the asymptotic behavior remains consistent across models. In other words, one aims to demonstrate that these systems exhibit universal behavior characteristic of a half-space KPZ universality class.

To formally define this class, one must identify suitable analogues of the universal scaling limits known from the full-space KPZ setting. A significant step in this direction was taken recently in \cite{DY25}, where the authors introduced a one-parameter family of {\em half-space Airy line ensembles} $\hsa = \{\hsa_i\}_{i \geq 1}$, which serve as (critical) half-space counterparts of the standard Airy line ensemble $\mathcal{A} = \{\mathcal{A}_{i}\}_{i \geq 1}$ on $\mathbb{R}$, see Figure \ref{Fig.HSA}. We briefly elaborate on the term ``critical''. 

In many half-space models, the strength of the boundary effect is controlled by a parameter $\alpha$, and there exists a critical value $\alpha_c$ such that:
\begin{itemize}
\item When $\alpha < \alpha_c$ (the {\em subcritical regime}), the model exhibits a fluctuation scaling exponent of $1/2$ and Gaussian fluctuations.
\item When $\alpha  > \alpha_c$ (the {\em supercritical regime}), the fluctuation and transversal scaling exponents are $1/3$ and $2/3$, respectively.
\item When $\alpha$ is tuned as $\alpha = \alpha_c + \varpi N^{-1/3}$ with the system size $N$, the model again has $1/3$ fluctuation and $2/3$ transversal scaling exponents, but the limiting distribution now depends on $\varpi$ and differs from the supercritical regime.
\end{itemize}
We note that in the literature, the terminology for the subcritical and supercritical regimes is sometimes interchanged. 
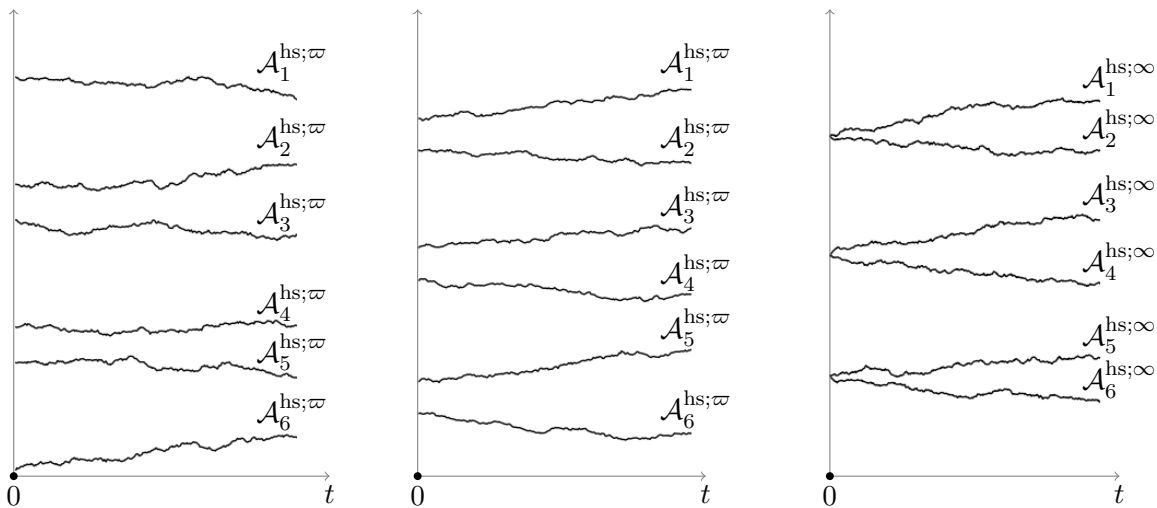
\begin{figure}[ht]
    \begin{center}
      \begin{tikzpicture}
        \node[anchor=south west, inner sep=0] (img1) at (0,0) 
            {\includegraphics[width=0.9\textwidth]{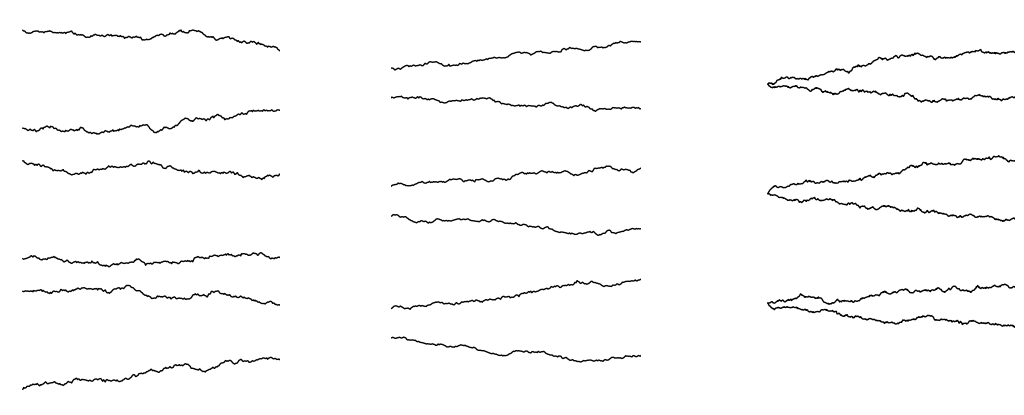}};

        \draw[->][gray] (0.3,0.29) -- (4.5,0.29);
        \draw[->][gray] (0.3,0.29) -- (0.3,6.5);
        \fill (0.3, 0.29) circle (1.5pt) node[below=0pt] {$0$};

        \draw (4.5, 0.05) node{$t$};
        \draw (4, 5.8) node{$\hsa_1$};
        \draw (4, 4.8) node{$\hsa_2$};
        \draw (4, 3.8) node{$\hsa_3$};
        \draw (4, 2.6) node{$\hsa_4$};
        \draw (4, 1.95) node{$\hsa_5$};
        \draw (4, 1.15) node{$\hsa_6$};

        \draw[->][gray] (5.67,0.29) -- (9.5,0.29);
        \draw[->][gray] (5.67,0.29) -- (5.67,6.5);
        \fill (5.67, 0.29) circle (1.5pt) node[below=0pt] {$0$};

        \draw (9.5, 0.05) node{$t$};
        \draw (9.37, 5.8) node{$\hsa_1$};
        \draw (9.37, 4.8) node{$\hsa_2$};
        \draw (9.37, 3.9) node{$\hsa_3$};
        \draw (9.37, 3) node{$\hsa_4$};
        \draw (9.37, 2.27) node{$\hsa_5$};
        \draw (9.37, 1.18) node{$\hsa_6$};

        \draw[->][gray] (11.15,0.29) -- (15,0.29);
        \draw[->][gray] (11.15,0.29) -- (11.15,6.5);
        \fill (11.15, 0.29) circle (1.5pt) node[below=0pt] {$0$};

        \draw (15, 0.05) node{$t$};
        \draw (15, 5.6) node{$\hsai_1$};
        \draw (15, 4.92) node{$\hsai_2$};
        \draw (15, 4.05) node{$\hsai_3$};
        \draw (15, 3.15) node{$\hsai_4$};
        \draw (15, 2.17) node{$\hsai_5$};
        \draw (15, 1.58) node{$\hsai_6$};

\end{tikzpicture}

    \end{center}
    \caption{The figure depicts $\hsa$ near the origin for $\varpi \in (-\infty, 0)$ (left), $\varpi \in (0, \infty)$ (middle) and $\varpi = \infty$ (right). At the origin, the curves $\hsa_{2i-1}$ and $\hsa_{2i}$ repel more for $\varpi \in (-\infty,0)$, repel less for $\varpi \in (0, \infty)$, and collide for $\varpi = \infty$.}
    \label{Fig.HSA}
\end{figure}

The line ensembles $\hsa$ constructed in \cite{DY25} arise as weak limits of certain discrete path ensembles associated with Pfaffian Schur processes --- or, equivalently, with symmetrized geometric LPP --- under critical parameter scaling. For background on Pfaffian Schur processes, we refer the reader to \cite{BR05,BBNV}, and to \cite{DY25b} for a concise explanation of their connection to LPP.\\

The {\bf main goal of the present paper is to construct, in a rigorous way, the {\em pinned half-space Airy line ensemble}} $\hsai = \{\hsai_i\}_{i \geq 1}$, which is the half-space counterpart of the Airy line ensemble in the supercritical regime described above. We obtain $\hsai$ as the weak $\varpi \rightarrow \infty$ limit of the ensembles $\hsa$ from \cite{DY25}, and we expect it to be the canonical universal scaling limit of supercritical models in the half-space KPZ universality class. For example, this model should arise as the supercritical large-time limit of the half-space KPZ line ensembles, see \cite[Section 1.1.1]{DasSerio25}, and it will be shown to be the scaling limit of supercritical geometric LPP in the forthcoming work \cite{DDY26}. 

Apart from constructing $\hsai$, we provide a description of its finite-dimensional distributions away from the origin, where the ensemble has a natural structure of a Pfaffian point process with an explicit correlation kernel. We show that (after a parabolic shift) the ensemble $\hsai$ satisfies a one-sided Brownian Gibbs property, described by ordered pairwise pinned Brownian motions, which was recently identified in \cite{DasSerio25}. As $T \rightarrow \infty$ (i.e. as one moves away from the boundary), the ensemble $\hsai(\cdot + T)$ behaves like a full-space model and converges to the standard Airy line ensemble $\mathcal{A}$. Lastly, we show that the distribution of $\hsai$ at the origin is described by the ordered eigenvalues (with doubled multiplicity) of the stochastic Airy operator $\mathcal{H}_{\beta}$ from \cite{RRV11} with $\beta =4$.

%
%
\subsection{The full-space and critical half-space Airy line ensembles}\label{Section1.1} 

The \emph{(full-space) Airy line ensemble} $\mathcal{A} = \{\mathcal{A}_{i}\}_{i \geq 1}$ is a sequence of strictly ordered, continuous, real-valued functions on $\mathbb{R}$. It arises as an edge-scaling limit in various probabilistic, combinatorial, and statistical mechanics models, including Wigner matrices~\cite{Sod15}, lozenge tilings~\cite{AH21}, and systems of non-intersecting Brownian bridges (also known as \emph{Brownian watermelons})~\cite{CorHamA}. It also emerges in a range of integrable models of non-intersecting paths and last passage percolation~\cite{DNV19}. The class of models related to $\mathcal{A}$ expands further if one includes those that converge to its various projections, such as $(\mathcal{A}_{1}(t): t \in \mathbb{R})$---the \emph{Airy process}, $(\mathcal{A}_{i}(t_0): i \geq 1)$ for fixed $t_0 \in \mathbb{R}$---the \emph{Airy point process}, and $\mathcal{A}_{1}(t_0)$ for fixed $t_0$---the \emph{Tracy--Widom distribution}. Because of its universality as a scaling limit and its fundamental role in the definition of the \emph{Airy sheet}~\cite{DOV22}, the Airy line ensemble occupies a central position within the KPZ universality class~\cite{CU2}. 

One of the salient features of the Airy line ensemble is that it has the structure of a {\em determinantal point process}. To define this precisely, we first introduce some notation.

\begin{definition}\label{Def.S1Contours} For a fixed $z \in \mathbb{C}$ and $\varphi \in (0, \pi)$, we denote by $\mathcal{C}_{z}^{\varphi}=\{z+|s|e^{\mathrm{sgn}(s)\im\varphi}: s\in \mathbb{R}\}$ the infinite contour oriented from $z+\infty e^{-\im\varphi}$ to $z+\infty e^{\im\varphi}$. 
\end{definition}

\vspace{-1.5mm}
\begin{definition}\label{Def.Measures} For a finite set $\mathsf{S} = \{s_1, \dots, s_m\} \subset \mathbb{R}$, we let $\mu_{\mathsf{S}}$ denote the counting measure on $\mathbb{R}$, defined by $\mu_{\mathsf{S}}(A) = |A \cap \mathsf{S}|$. We also let $\mathrm{Leb}$ be the usual Lebesgue measure on $\mathbb{R}$ and $\mu_{\mathsf{S}} \times \mathrm{Leb}$ the product measure on $\mathbb{R}^2$.
\end{definition}

\vspace{-1.5mm}
\begin{definition}\label{Def.Ordered} Let $\mathcal{L} = \{\mathcal{L}_{i}\}_{i \geq 1}$ be a line ensemble on an interval $\Lambda \subseteq \mathbb{R}$. We say that $\mathcal{L}$ is {\em non-intersecting} if almost surely $\mathcal{L}_i(t) > \mathcal{L}_j(t)$ for all $t \in \Lambda$ and $1 \leq i < j$. We say that $\mathcal{L}$ is {\em ordered} if almost surely $\mathcal{L}_i(t) \geq \mathcal{L}_j(t)$ for all $t \in \Lambda$ and $1 \leq i < j$.   
\end{definition}

\vspace{-1.5mm}
\begin{definition}\label{Def.AiryLE} The (full-space) Airy line ensemble $\mathcal{A} = \{\mathcal{A}_i\}_{i \geq 1}$ is a sequence of real-valued, random continuous functions on $\mathbb{R}$. It is uniquely characterized by the following properties. The ensemble is non-intersecting in the sense of Definition \ref{Def.Ordered}. In addition, for any finite set $\mathsf{S} = \{s_1, \dots, s_m\} \subset \mathbb{R}$ with $s_1 < \cdots < s_m$, the random measure on $\mathbb{R}^2$ given by
\begin{equation}\label{Eq.RMS1}
M^{\mathsf{S}; \mathcal{A}}(A) = \sum_{i \geq 1} \sum_{j = 1}^m {\bf 1} \left\{\left(s_j, \mathcal{A}_{i}(s_j) \right) \in A \right\},
\end{equation}
is a determinantal point process with reference measure $\mu_{\mathsf{S}} \times \mathrm{Leb}$ as in Definition \ref{Def.Measures}, and with correlation kernel given by the {\em extended Airy kernel}, defined for $x_1, x_2 \in \mathbb{R}$ and $t_1, t_2 \in \mathsf{S}$ by 
\begin{equation}\label{Eq.S1AiryKer}
\begin{split}
&K^{\mathrm{Airy}}(s,x; t,y) =  R(s,x;t,y) + \frac{1}{(2\pi \im)^2} \int_{\mathcal{C}_{\alpha}^{\pi/3}} d z \int_{\mathcal{C}_{\beta}^{\pi/3}} dw \frac{e^{z^3/3 +w^3/3 -xz - yw}}{z + s + w - t}, \mbox{ where }\\
&R(s,x;t,y) = - \frac{{\bf 1}\{s < t\} }{\sqrt{4\pi (t-s)}} \cdot \exp \left(\frac{- (s-t)^4 + 6 (x+ y)(s-t)^2 + 3 (x-y)^2}{12 (s-t)} \right).
\end{split}
\end{equation}
In (\ref{Eq.S1AiryKer}), $\alpha, \beta \in \mathbb{R}$ are arbitrary subject to $\alpha + s + \beta - t > 0$, and $\mathcal{C}_{z}^{\varphi}$ are as in Definition \ref{Def.S1Contours}. 
\end{definition}
\begin{remark}\label{Rem.AiryLE1}
There are various formulas for the extended Airy kernel, and the one in (\ref{Eq.S1AiryKer}) comes from \cite[Proposition 4.7 and (11)]{BK08} under the change of variables $u \rightarrow z + s$ and $w \rightarrow -w + t$. 
\end{remark}
\begin{remark}\label{Rem.AiryLE2} The introduction of the extended Airy kernel is often attributed to \cite{Spohn}, where it arises in the context of the polynuclear growth model, although it appeared earlier in \cite{FNH99} and \cite{Mac94}. The {\em existence} of a line ensemble satisfying the conditions of Definition \ref{Def.AiryLE} was established in \cite{CorHamA}, while the fact that these conditions {\em uniquely} determine the law of $\mathcal{A}$ is well known and follows, for example, from \cite[Proposition 2.13(3)]{dimitrov2024airy}, \cite[Corollary 2.20]{dimitrov2024airy} and \cite[Lemma 3.1]{DimMat}.
\end{remark}

The Airy line ensemble $\mathcal{A}$ was constructed in \cite{CorHamA} as the edge limit of Brownian watermelons by utilizing their locally avoiding Brownian bridge structure. This local description of the Brownian watermelons is now referred to as the {\em Brownian Gibbs property}, and it is enjoyed by the {\em parabolic Airy line ensemble} $\mathcal{L}^{\mathrm{Airy}}$ whose relationship to $\mathcal{A}$ is explicitly given by
\begin{equation}\label{PALE}
\sqrt{2} \cdot \mathcal{L}^{\mathrm{Airy}}_i(t) =  \mathcal{A}_i(t) - t^2 \mbox{ for $i \geq 1$ and $t \in \mathbb{R}$}.
\end{equation}
We note that the Brownian Gibbs property is not merely a cosmetic feature of the parabolic Airy line ensemble, but a fundamentally defining one. Specifically, \cite{AH23} characterizes $\mathcal{L}^{\mathrm{Airy}}$ as essentially the unique line ensemble that both possesses the Brownian Gibbs property and has a top curve that is globally parabolic.\\

In \cite{DY25} the first and third authors introduced a one-parameter family of {\em (critical) half-space Airy line ensembles} $\hsa = \{\hsa_i\}_{i \geq 1}$, indexed by $\varpi \in \mathbb{R}$. As mentioned earlier, these ensembles were obtained as weak limits of certain discrete path ensembles associated with Pfaffian Schur processes, and to formally introduce them we require some more notation. 
\begin{definition}\label{Def.kcr}
Fix $\varpi \in \mathbb{R}$, $m \in \mathbb{N}$ and $\mathsf{S} = \{s_1, \dots, s_m\}$, where $0 \leq s_1 < s_2 < \cdots < s_m$, and set $\ap_i = |\varpi| + 3i$ for $i = 1,2,3$. For $s,t \in \mathsf{S}$ and $x, y \in \mathbb{R}$, we define the kernel
\begin{equation}\label{Eq.S1DefKcross}
\begin{split}
&\kcr(s,x; t,y) = \begin{bmatrix}
    \kcr_{11}(s,x;t,y) & \kcr_{12}(s,x;t,y)\\
    \kcr_{21}(s,x;t,y) & \kcr_{22}(s,x;t,y) 
\end{bmatrix} \\
&= \begin{bmatrix}
    \icr_{11}(s,x;t,y) & \icr_{12}(s,x;t,y) + \rcr_{12}(s,x;t,y) \\
    -\icr_{12}(t,y;s,x) - \rcr_{12}(t,y;s,x) & \icr_{22}(s,x;t,y) + \rcr_{22}(s,x;t,y) 
\end{bmatrix} ,
\end{split}
\end{equation}
where the kernels $\icr_{ij}, \rcr_{ij}$ are defined as follows. We have
\begin{equation}\label{Eq.S1DefIcross}
\begin{split}
\icr_{11}(s,x;t,y) = &\frac{1}{(2\pi \im)^2} \int_{\mathcal{C}_{\ap_1 - s}^{\pi/3}}dz \int_{\mathcal{C}_{\ap_3 - t}^{\pi/3}} dw \frac{z + s - w - t}{(z + s + w + t)(z+s)(w+ t)} \\
& \times (z + s+ \varpi )(w + t + \varpi) \cdot e^{z^3/3 + w^3/3 - xz - y w}, \\
\icr_{12}(s,x;t,y) = &\frac{1}{(2\pi \im)^2} \int_{\mathcal{C}_{\ap_3 - s}^{\pi/3}}dz \int_{\mathcal{C}_{\ap_1 - t}^{2\pi/3}} dw \frac{z + s + w + t}{2(z+ s)(z + s - w -t)} \\
& \times \frac{z + \varpi + s}{w + \varpi + t} \cdot e^{z^3/3 - w^3/3 - xz + y w},\\
\icr_{22}(s,x;t,y) = &\frac{1}{(2\pi \im)^2} \int_{\mathcal{C}_{-\ap_2 - s}^{2\pi/3}}dz \int_{\mathcal{C}_{-\ap_2 - t}^{2\pi/3}} dw \frac{z + s - w - t}{4 (z + s + w + t)} \\
& \times \frac{1}{(z + s  + \varpi )(w + t + \varpi)} \cdot e^{-z^3/3 - w^3/3 + xz + y w},
\end{split}
\end{equation}
with contours as in Definition \ref{Def.S1Contours}. We also have $\rcr_{12} = R$ as in (\ref{Eq.S1AiryKer}), $\rcr_{22}(s,x; t,y) = - \rcr_{22}(t,y; s,x) $ and when $x- s^2 > y -t^2$ we have
\begin{equation}\label{Eq.S1DefRcross2}
\begin{split}
\rcr_{22}(s,x;t,y) = & \frac{1}{2\pi \im} \int_{\mathcal{C}_{\ap_1 }^{2\pi/3}}dz \frac{e^{(-z + s)^3/3 + (\varpi + t)^3/3 - x (-z+s) - y (\varpi + t)}}{4(z - \varpi)}\\
&- \frac{1}{2\pi \im}\int_{\mathcal{C}_{-\ap_2 }^{2\pi/3}}dw \frac{e^{(-w + t)^3/3 + (\varpi + s)^3/3 - y (-w+t) - x (\varpi + s)}}{4(w - \varpi)}\\
& + \frac{{\bf 1}\{s + t > 0\} }{2\pi \im} \int_{\mathcal{C}^{2\pi/3}_{-\ap_{2}}} dw \frac{we^{ (-w + t)^3/3 + (w + s)^3/3   - y (-w + t) - x (w+s)}   }{2(w- \varpi)(w + \varpi) }.
\end{split}
\end{equation}
\end{definition}
\begin{remark}\label{S1Correction} The formula for $\kcr$ in Definition \ref{Def.kcr} was originally obtained in \cite[Section 5.1]{BBCS2} and \cite[Section 2.5]{BBCS}, where it is denoted by $K^{\mathrm{cross}}$. The formulas in \cite{BBCS,BBCS2} have a few typos, which have since been corrected in the arXiv versions of these papers -- see \cite[Section 2.5]{BBCSArxiv} and \cite[Section 5.1]{BBCS2Arxiv}. We mention that our formulas for $\icr_{11}, \icr_{12}$ and $\rcr_{12}$ completely agree (after a simple change of variables) with $I_{11}^{\mathrm{cross}}, I_{12}^{\mathrm{cross}}$ and $R_{12}^{\mathrm{cross}}$ in \cite[Section 2.5]{BBCSArxiv}. As the authors chose slightly different contours in the definitions of $I^{\mathrm{cross}}_{22}$ and $R^{\mathrm{cross}}_{22}$ than ours, these terms do not match; however, their sum (which is precisely $K^{\mathrm{cross}}_{22}$) is the same as $\kcr_{22}$. In other words, $\kcr$ matches $K^{\mathrm{cross}}$ in \cite[Section 2.5]{BBCSArxiv}. We also mention that in \cite[Equation (4.10)]{BBNV} the authors introduce a kernel $K^v$, which also matches $\kcr$ after several (tedious) changes of variables, including setting $\varpi = 2 v$, $s = u_a$, $t = u_b$, $x = \xi - u_a^2$, $y = \xi' - u_b^2$, $z = Z -s$, $w = W-t$, and conjugating the kernel.
\end{remark}
\begin{remark}\label{Rem.kcr} We mention that due to the cubic terms in the exponentials, there is considerable freedom in deforming the contours that appear in the kernels $\icr_{ij}, \rcr_{ij}$ without affecting the values of the integrals. In addition, for some values of $\varpi,s,t$ one could cross some simple poles and obtain substantial cancellations. Our contour choice in Definition \ref{Def.kcr} is made after \cite{DY25}, and one advantage is that the formulas hold for all values of $\varpi \in \mathbb{R}$ and $s,t \in [0, \infty)$. Later, when we investigate the $\varpi \rightarrow \infty$ limit, we will find (arguably) simpler formulas that are more suitable for asymptotic analysis --- see Definition \ref{Def.NewKernel} and Lemma \ref{Lem.KernelMatch}.  
\end{remark}

With the above notation in place, we can state the main result from \cite{DY25}.
\begin{proposition}\label{Prop.HSAiry} Fix $\varpi \in \mathbb{R}$. There exists a line ensemble $\hsa = \{\hsa_i\}_{i \geq 1}$ on $[0, \infty)$ that satisfies the following properties. The ensemble is non-intersecting in the sense of Definition \ref{Def.Ordered}. In addition, for any finite set $\mathsf{S} = \{s_1, \dots, s_m\} \subset [0,\infty)$ with $s_1 < s_2 < \cdots < s_m$, the random measure on $\mathbb{R}^2$ given by
\begin{equation}\label{eq:HSA point process}
\hsm(\omega, A) = \sum_{i \geq 1} \sum_{j = 1}^m {\bf 1}\{(s_j, \hsa_i(s_j,\omega) ) \in A\}
\end{equation}
is a Pfaffian point process with correlation kernel $\kcr$ as in Definition \ref{Def.kcr} and reference measure $\mu_{\mathsf{S}} \times \mathrm{Leb}$ as in Definition \ref{Def.Measures}. 

Moreover, if $\hsl = \{\hsl_i\}_{i \geq 1}$ is the line ensemble defined through 
\begin{equation}\label{eq:Parabolic HSA}
\sqrt{2} \cdot \hsl_i(t) + t^2 = \hsa_i(t) \mbox{ for } i \geq 1, t \in [0,\infty),
\end{equation}
then $\hsl$ satisfies the half-space Brownian Gibbs property from Definition \ref{Def.BGP} with parameters $\mu_i = (-1)^{i} \sqrt{2} \cdot \varpi$ (see also Remark \ref{Rem.GibbsPropertyHSA} and Figure \ref{Fig.HSA}).
\end{proposition}
\begin{remark}\label{Rem.UniquenessHSA}
The conditions in Proposition \ref{Prop.HSAiry} uniquely determine the law of $\hsa$ in view of \cite[Proposition 5.8(3)]{DY25}, \cite[Corollary 2.20]{dimitrov2024airy} and \cite[Lemma 3.1]{DimMat}.
\end{remark}
\begin{remark}\label{Rem.GibbsPropertyHSA} Informally, a line ensemble $\{\mathcal{L}_i\}_{i\geq 1}$ satisfies the half-space Brownian Gibbs property from Definition \ref{Def.BGP} with parameters $\{\mu_i\}_{i \geq 1}$ if the following holds. For each $m \in \mathbb{N}$ and $T > 0$, the law of the curves $\{\mathcal{L}_{i}\}_{i = 1}^{m}$ on the interval $[0,T]$ conditioned on $\{\mathcal{L}_{i}(T)\}_{i = 1}^{m}$ and $\{\mathcal{L}_{m+1}(t): t \in [0,T]\}$ is that of $m$ independent reverse Brownian motions $\{\cev{B}_{i} \}_{i=1}^{m}$ with drifts $\{\mu_{i}\}_{i = 1}^{m}$, started from $\cev{B}_i(T) = \mathcal{L}_i(T)$, that are conditioned to not intersect:
$$\cev{B}_1(t) > \cev{B}_2(t) > \cdots > \cev{B}_{m}(t) > \mathcal{L}_{m+1}(t)  \mbox{ for } t\in [0,T].$$
Here, we say that $\cev{B}$ is a {\em reverse Brownian motion with drift $\mu$ from $\cev{B}(T) = y$} if
$$\cev{B}(t) = y + W_{T-t} + \mu (T-t) \mbox{ for } 0 \leq t \leq T,$$
where $(W_t:t\geq 0)$ is a standard Brownian motion, started from zero.
\end{remark}

%
%
\subsection{Main results}\label{Section1.2} In this section we continue with the notation from Section \ref{Section1.1}, and present the main results of the paper. Our first main result shows that the critical half-space Airy line ensembles $\hsa$ converge weakly as $\varpi \rightarrow \infty$.
\begin{theorem}\label{Thm.Convergence} Let $\hsa = \{\hsa_i\}_{i \geq 1}$ be as in Proposition \ref{Prop.HSAiry}. Then, the line ensembles $\hsa$ converge weakly to an ordered line ensemble $\hsai$ as $\varpi \rightarrow \infty$. 
\end{theorem}
\begin{remark}\label{Rem.Convergence} The convergence in Theorem \ref{Thm.Convergence} is of random elements in $C(\mathbb{N} \times [0,\infty))$, where the latter space is endowed with the topology of uniform convergence over compact sets.
\end{remark}
\begin{remark}\label{Rem.Convergence2} We mention that while $\hsa$ is non-intersecting for each $\varpi \in \mathbb{R}$, the limiting ensemble $\hsai$, henceforth referred to as the {\em pinned half-space Airy line ensemble}, is only ordered, cf. Definition \ref{Def.Ordered}. As alluded to in Figure \ref{Fig.HSA}, and rigorously established in Theorem \ref{Thm.GibbsProperty} (see also Theorem \ref{Thm.OriginDescription}), the ensemble $\hsai$ is non-intersecting on $(0,\infty)$ but at the origin satisfies 
\begin{equation}\label{Eq.OriginCollision}
\hsai_1(0) = \hsai_2(0) > \hsai_3(0) = \hsai_4(0) > \hsai_5(0) = \hsai_6(0) > \cdots.
\end{equation}
\end{remark}

In the remainder of this section we summarize various properties of the pinned half-space Airy line ensemble $\hsai$ from Theorem \ref{Thm.Convergence} in a sequence of theorems. The first of these shows that $\hsai$ has a Pfaffian point process structure with an explicit correlation kernel. 
\begin{theorem}\label{Thm.PfaffianStructure} Let $\hsai$ be as in Theorem \ref{Thm.Convergence}, and $\hsmi$ be as in (\ref{eq:HSA point process}) with $\varpi = \infty$ for $0< s_1 < s_{2} < \cdots < s_m$. Then, $\hsmi$ is a Pfaffian point process on $\mathbb{R}^2$ with reference measure $\mu_{\mathsf{S}} \times \mathrm{Leb}$ as in Definition \ref{Def.Measures} and correlation kernel
\begin{equation}\label{Eq.DefK}
\begin{split}
&\hski(s,x; t,y) = \begin{bmatrix}
    \hski_{11}(s,x;t,y) & \hski_{12}(s,x;t,y)\\
    \hski_{21}(s,x;t,y) & \hski_{22}(s,x;t,y) 
\end{bmatrix} \\
&= \begin{bmatrix}
    \hsii_{11}(s,x;t,y) & \hsii_{12}(s,x;t,y) + \hsri_{12}(s,x;t,y) \\
    -\hsii_{12}(t,y;s,x) - \hsri_{12}(t,y;s,x) & \hsii_{22}(s,x;t,y) + \hsri_{22}(s,x;t,y) 
\end{bmatrix},
\end{split}
\end{equation}
where the kernels $\hsii_{ij}, \hsri_{ij}$ are defined as follows. We have
\begin{equation}\label{Eq.DefII}
\begin{split}
\hsii_{11}(s,x;t,y) = &\frac{1}{(2\pi \im)^2} \int_{\mathcal{C}_{1+s}^{\pi/3}}dz \int_{\mathcal{C}_{1+t}^{\pi/3}} dw \frac{(z + s - w - t)H(z,x;w,y)}{4(z + s + w + t)(z+s)(w+ t)}, \\
\hsii_{12}(s,x;t,y) = &\frac{1}{(2\pi \im)^2} \int_{\mathcal{C}_{1+s}^{\pi/3}}dz \int_{\mathcal{C}_{1+t}^{\pi/3}} dw \frac{(z + s - w + t) H(z,x;w,y)}{2(z+ s)(z + s + w -t)}, \\
\hsii_{22}(s,x;t,y) = &\frac{1}{(2\pi \im)^2} \int_{\mathcal{C}_{1+ s}^{\pi/3}}dz \int_{\mathcal{C}_{1+t}^{\pi/3}} dw \frac{(z - s - w + t)H(z,x;w,y)}{z - s + w - t},
\end{split}
\end{equation}
where we have set 
\begin{equation}\label{Eq.DefH}
H(z,x;w,y) = e^{z^3/3 + w^3/3 - xz - y w},
\end{equation}
and the contours $\mathcal{C}_{z}^{\varphi}$ are as in Definition \ref{Def.S1Contours}. We also have that $\hsri_{12} = R$ is as in (\ref{Eq.S1AiryKer}) and
\begin{equation}\label{Eq.DefR22I}
\begin{split}
\hsri_{22}(s,x;t,y) = \frac{H(s,x;t,y) (y-t^2 -x + s^2) }{2\pi^{1/2} (t+s)^{3/2} } \cdot \exp \left(-\frac{(y - t^2 - x +s^2)^2}{4(t+s)}\right).
\end{split}
\end{equation}
\end{theorem}
\begin{remark}\label{Rem.PfaffianStructure0} The kernel $\hski$ appeared earlier in the context of the symplectic-unitary transition \cite{FNH99}, and half-space exponential LPP \cite[Section 2.5]{BBCS}. Specifically, our kernel formulas agree with $K^{\mathrm{SU}}$ from \cite[Section 2.5]{BBCS}, except that there is a small typo in $R_{22}^{\mathrm{SU}}$ in that paper, which should be multiplied by $4$. 
\end{remark}
\begin{remark}\label{Rem.PfaffianStructure} We mention that in Theorem \ref{Thm.PfaffianStructure}, unlike Proposition \ref{Prop.HSAiry}, it is crucial that $\mathsf{S} \subset (0,\infty)$. In other words, the Pfaffian point process structure of $\hsai$ does {\em not} extend to the origin. This is expected, as Pfaffian point processes are almost surely simple (see \cite[Proposition 5.8]{DY25}), while the point process on $\mathbb{R}$ formed by $\{\hsai_i(0)\}_{i \geq 1}$ is not in view of (\ref{Eq.OriginCollision}). 
\end{remark}
\begin{remark}\label{Rem.PfaffianStructure2} From \cite[Proposition 5.8(3)]{DY25}, \cite[Corollary 2.20]{dimitrov2024airy} and \cite[Lemma 3.1]{DimMat}, there is at most one ordered line ensemble that satisfies the conditions of Theorem \ref{Thm.PfaffianStructure}. Consequently, one can define $\hsai$ as the unique ordered ensemble satisfying these conditions, whose existence is guaranteed by Theorems \ref{Thm.Convergence} and \ref{Thm.PfaffianStructure}.
\end{remark}

The next result shows that $\hsai$ has a local description in terms of avoiding Brownian motions/bridges with special {\em pinning} at the origin.
\begin{theorem}\label{Thm.GibbsProperty} Let $\hsai$ be as in Theorem \ref{Thm.Convergence}, and define $\hsail = \{\hsail_{i}\}_{i \geq 1}$ through
\begin{equation}\label{Eq.ParabolicHSALE}
\sqrt{2} \cdot \hsail_i(t) + t^2 = \hsai_i(t) \mbox{ for } i \geq 1, t \in [0, \infty).
\end{equation}
Then the following hold.
\begin{enumerate}
\item[(a)] The restriction of $\hsail$ to $(0,\infty)$, viewed as an $\mathbb{N}$-indexed line ensemble on $(0,\infty)$, satisfies the Brownian Gibbs property in Definition \ref{Def.BGPVanilla}. In particular, on $(0,\infty)$ the line ensemble $\hsai$ is non-intersecting as in Definition \ref{Def.Ordered}.
\item[(b)] On $[0,\infty)$ the ensemble $\hsail$ satisfies the pinned half-space Brownian Gibbs property from Definition \ref{Def.PinnedBGP}.
\end{enumerate}
\end{theorem}
\begin{remark}\label{Rem.GibbsProperty} Theorem \ref{Thm.GibbsProperty}(a) states that for each $m \in \mathbb{N}$ and $b > a > 0$, the law of $\{\hsail_{i}\}_{i = 1}^{m}$ on the interval $[a,b]$, conditioned on $\{\hsail_{i}(a)\}_{i = 1}^{m}$, $\{\hsail_{i}(b)\}_{i = 1}^{m}$ and $\{\hsail_{m+1}(t): t \in [a,b]\}$, is that of $m$ independent Brownian bridges $\{{B}_{i} \}_{i=1}^{m}$ from ${B}_i(a) = \hsail_i(a)$ to ${B}_i(b) = \hsail_i(b)$ that are conditioned to not intersect:
$$B_1(t) > B_2(t) > \cdots > B_{m}(t) > \hsail_{m+1}(t)  \mbox{ for } t\in [a,b].$$

Theorem \ref{Thm.GibbsProperty}(b) states that for each $m \in \mathbb{N}$ and $T > 0$, the law of the curves $\{\mathcal{L}_{i}\}_{i = 1}^{2m}$ on the interval $[0,T]$ conditioned on $\{\mathcal{L}_{i}(T)\}_{i = 1}^{2m}$ and $\{\mathcal{L}_{2m+1}(t): t \in [0,T]\}$ is that of $2m$ independent reverse Brownian motions $\{\cev{B}_{i} \}_{i=1}^{2m}$ started from $\cev{B}_i(T) = \mathcal{L}_i(T)$ that are conditioned to stay ordered
$$\cev{B}_1(t) \geq \cev{B}_2(t) \geq \cdots \geq \cev{B}_{2m}(t) \geq \mathcal{L}_{2m+1}(t)  \mbox{ for } t\in [0,T],$$
and with pairwise pinning at the origin:
$$\cev{B}_{2i-1}(0) = \cev{B}_{2i}(0) \mbox{ for } i = 1, 2, \dots, m.$$
\end{remark}
\begin{remark}\label{Rem.GibbsProperty2} Recall from Remark \ref{Rem.GibbsPropertyHSA} that near the origin the odd curves $\hsa_{2i-1}$ have a drift $-\sqrt{2} \varpi$, while the even ones $\hsa_{2i}$ have a drift $\sqrt{2} \varpi$ (when run backwards). As $\varpi \rightarrow \infty$ this causes $\hsa_{2i-1}$ to collide with $\hsa_{2i}$ for $i \geq 1$, as the ensemble needs to remain ordered. Consequently, the pinned half-space Brownian Gibbs property of $\hsai$ arises as a natural $\varpi \rightarrow \infty$ limit of the half-space Brownian Gibbs property enjoyed by $\hsa$.
\end{remark}

The next result explains the behavior of $\hsai$ far from the origin. As one moves away from the zero boundary, its effect diminishes; $\hsai$ starts to behave like a full-space model, and one recovers the full-space Airy line ensemble in the limit. 
\begin{theorem}\label{Thm.ConvergenceToAiryLE} Let $\hsai$ be as in Theorem \ref{Thm.Convergence} and fix a sequence $t_n \in (0, \infty)$ with $t_n \uparrow \infty$. For each $n \in \mathbb{N}$ define the line ensemble $\mathcal{A}^n = \{\mathcal{A}^n_{i}\}_{i \geq 1}$ on $\mathbb{R}$ by setting for each $i \geq 1$
\begin{equation}\label{Eq.ShiftedLE}
\mathcal{A}^n_i(t) = \begin{cases} \hsai_i(t+ t_n) &\mbox{ if } t \geq - t_n, \\ \hsai_i(0) &\mbox{ if } t < -t_n. \end{cases}
\end{equation}
Then, $\mathcal{A}^n \Rightarrow \mathcal{A}$, where $\mathcal{A}$ is the full-space Airy line ensemble from Definition \ref{Def.AiryLE}.
\end{theorem}

Our last main result about $\hsai$ describes its distribution at the origin. To formulate it precisely, and also establish it later in Section \ref{Section6}, we require the following statement.
\begin{lemma}\label{Lem.SAO} Let $\{\Lambda_k\}_{k \geq 0}$ denote the ordered in ascending order eigenvalues of the stochastic Airy operator $\mathcal{H}_{\beta}$ from \cite{RRV11} with $\beta = 4$. Let $M^{\mathrm{GSE}}$ be the random measure on $\mathbb{R}$, defined through
\begin{equation}\label{Eq.DefMGSE}
M^{\mathrm{GSE}}(A) = \sum_{i = 1}^{\infty} \{ - 2^{2/3}\Lambda_{i-1} \in A\}.
\end{equation}
Then $M^{\mathrm{GSE}}$ is a Pfaffian point process on $\mathbb{R}$ with reference measure $\mathrm{Leb}$ and correlation kernel $\kgse$, given by
\begin{equation}\label{Def.Kgse}
\begin{split}
\kgse_{11}(x,y) = &\frac{1}{(2\pi \im)^2} \int_{\mathcal{C}_{1}^{\pi/3}}dz \int_{\mathcal{C}_{1}^{\pi/3}} dw \frac{(z - w )H(z,x;w,y)}{4(z + w )zw}, \\
\kgse_{12}(x,y) = - \kgse_{21}(y,x) = &\frac{1}{(2\pi \im)^2} \int_{\mathcal{C}_{1}^{\pi/3}}dz \int_{\mathcal{C}_{1}^{\pi/3}} dw \frac{(z - w ) H(z,x;w,y)}{4z(z + w)}, \\
\kgse_{22}(x,y) = &\frac{1}{(2\pi \im)^2} \int_{\mathcal{C}_{1}^{\pi/3}}dz \int_{\mathcal{C}_{1}^{\pi/3}} dw \frac{(z - w)H(z,x;w,y)}{4(z + w)},
\end{split}
\end{equation}
where $H(z,x;w,y)$ is as in (\ref{Eq.DefH}) and the contour $\mathcal{C}_1^{\pi/3}$ is as in Definition \ref{Def.S1Contours}.
\end{lemma}
\begin{remark}\label{Rem.SAO} We mention that Lemma \ref{Lem.SAO} is known to experts, and follows from \cite{RRV11} and the well-studied edge asymptotics of the eigenvalues of the Gaussian Symplectic Ensemble. As we could not find this statement in the literature, we provide a short proof of it in Appendix \ref{SectionA}. 
\end{remark}

The following theorem establishes a distributional equality between the $\hsai$ at the origin and the eigenvalues of the stochastic Airy operator with $\beta = 4$.
\begin{theorem}\label{Thm.OriginDescription} Let $\hsai$ be as in Theorem \ref{Thm.Convergence}, and $\{\Lambda_{k}\}_{k \geq 0}$ as in Lemma \ref{Lem.SAO}. Then we have the following equality in law as random vectors in $\mathbb{R}^{\infty}$
\begin{equation}\label{Eq.Equality in law}
(\hsai_1(0), \hsai_2(0), \hsai_3(0), \hsai_4(0), \dots) \overset{d}{=}  (-2^{2/3}\Lambda_0, -2^{2/3}\Lambda_0, -2^{2/3}\Lambda_1, -2^{2/3}\Lambda_1, \dots).
\end{equation}
\end{theorem}
\begin{remark}\label{Rem.OriginDescription} Lemma \ref{Lem.SAO} and Theorem \ref{Thm.OriginDescription} show that the point process formed by $\{\hsai_i(0)\}_{i \geq 1}$ is not itself Pfaffian, but rather a doubling of the Pfaffian point process $M^{\mathrm{GSE}}$. This completely specifies the law of $\{\hsai_i(0)\}_{i \geq 1}$ in view of \cite[Proposition 5.8(3)]{DY25}, \cite[Corollary 2.20]{dimitrov2024airy}.
\end{remark}
\begin{remark}\label{Rem.OriginDescription2} The convergence of the point processes formed by $\{\hsai_i(t)\}_{i \geq 1}$ to the doubled point process $2M^{\mathrm{GSE}}$ as $t \rightarrow 0+$ was predicted in \cite[Remark 5.1]{BBCS}, although the authors did not establish this statement, and to our knowledge Theorem \ref{Thm.OriginDescription} is the first rigorous verification of this fact.
\end{remark}

%
%
\subsection{Ideas behind the proofs and paper outline}\label{Section1.3} In Section \ref{Section2.1}, we derive an alternative representation of the correlation kernel $\kcr$ from Definition \ref{Def.kcr} that is more amenable to asymptotic analysis; see Definition \ref{Def.NewKernel} and Lemma \ref{Lem.KernelMatch}. We denote this new kernel by $\hsk$. In Section \ref{Section2.2}, Lemma \ref{Lem.KernelLimit} analyzes the limit of $\hsk$ as $\varpi \rightarrow \infty$ and relates it to the kernel $\hski$ appearing in Theorem \ref{Thm.PfaffianStructure}. Subsequently, in Section \ref{Section2.3}, we study the large-time limit of $\hski$ and recover the extended Airy kernel; see Lemma \ref{Lem.ConvToAiryKernel}. The results of Section \ref{Section2} rely on appropriate contour deformations in the kernel representations, combined with careful estimates along the relevant contours.

In Section \ref{Section3}, we establish finite-dimensional analogues of Theorem \ref{Thm.Convergence} (Lemma \ref{Lem.FDConvVarpi}) and Theorem \ref{Thm.ConvergenceToAiryLE} (Lemma \ref{Lem.FDConvT}). Our proofs adapt a general framework developed in \cite{DY25} and require the verification of the following properties:
\begin{enumerate}
\item Tightness from above for the top curves of the ensembles. This is proved in Section \ref{Section3.1}; see Lemma \ref{Lem.TightFromAbove}.  
\item Weak convergence of the associated point processes. This is established in Section \ref{Section3.2}; see Lemmas \ref{Lem.PointProcessConvergenceVarpi} and \ref{Lem.PointProcessConvergenceT}, using our kernel convergence results from Lemmas \ref{Lem.KernelLimit} and \ref{Lem.ConvToAiryKernel}.
\item Almost sure infinitude of atoms in the weak limits of the point processes from (2) along vertical slices. 
\end{enumerate}

For the large-time limit, the third property is immediate, since the limiting object is the Airy point process, which is well known to have infinitely many atoms almost surely. The $\varpi \rightarrow \infty$ limit is more delicate, as we do not a priori know the structure of the limiting point processes at different times $t$. The required statement appears as Lemma \ref{Lem.InfinitelyManyAtoms} and is proved in Section \ref{SectionSpecial}. At a high level, the key idea is to exploit the half-space Brownian Gibbs property (satisfied by $\hsl$; see Proposition \ref{Prop.HSAiry}) to transfer information from the behavior of the ensemble at ``infinity'', where the associated point processes coincide with the Airy point process, to the finite-time ensemble. A similar strategy was used in the construction of the half-space Airy line ensemble in \cite{DY25}; we refer the reader to Section \ref{SectionSpecial} for further details. \\

In Section \ref{Section4.1}, we formally introduce the various Brownian Gibbs properties that are satisfied by our ensembles, while Section \ref{Section4.2} establishes several technical properties of finite Brownian line ensembles. Section \ref{Section4.3} contains three key results we require for our main arguments. The first is Lemma \ref{Lem.BrownianEnsembleConv}, which shows that a finite family of (reverse) Brownian motions with alternating drifts $(-1)^i \varpi$ converges, as $\varpi\rightarrow\infty$, to a finite family of pairwise pinned (reverse) Brownian motions; see Definition \ref{Def.PinnedBM}. The proof follows that of a related result due to Das and the second author in \cite{DasSerio25}. The key idea is to consider sums and differences of consecutive pairs of curves in half-space Brownian ensembles and show that these converge instead. The advantage of this reformulation is that the sum and difference of two avoiding Brownian motions are independent, allowing each auxiliary process to be analyzed separately. 

Another central result is Lemma \ref{Lem.MOC}, which roughly states that if a sequence of avoiding Brownian line ensembles on $[0,b]$ with alternating drifts has boundary conditions and a lower-bounding curve that are positively separated in a neighborhood of $b$ and remain within a compact window, then the modulus of continuity is uniformly controlled as the drift parameter tends to infinity. This result plays a key role in upgrading the finite-dimensional convergence of Lemma \ref{Lem.FDConvVarpi} to convergence that is uniform over compact sets, as required for Theorem \ref{Thm.Convergence}. The third key result is Lemma \ref{Lem.NoLowMin}, which feeds, in a nontrivial way, into the proof of the ``infinitely many atoms'' result, Lemma \ref{Lem.InfinitelyManyAtoms}.

In Section \ref{Section5}, we prove Theorems \ref{Thm.Convergence}, \ref{Thm.PfaffianStructure}, \ref{Thm.GibbsProperty}, and \ref{Thm.ConvergenceToAiryLE}. The main technical result is Proposition \ref{Prop.Tightness}, which shows that for any sequence $\varpi_n \rightarrow \infty$, the line ensembles $\mathcal{L}^{\mathrm{hs}; \varpi_n}$ from Proposition \ref{Prop.HSAiry} form a tight sequence. Moreover, any subsequential limit of these ensembles satisfies the regular and pinned half-space Brownian Gibbs properties. 

With finite-dimensional convergence already established, proving tightness reduces to controlling the modulus of continuity of the top $2k$ curves on an interval $[0,b]$, for arbitrary $k\in\mathbb{N}$ and $b>0$. The idea here is to use Lemma \ref{Lem.MOC}, for which we must show that the curves of our ensembles are likely to separate from one another at a fixed positive time. This is achieved by combining the finite-dimensional convergence from Lemma \ref{Lem.FDConvVarpi} with some classical results from \cite{CorHamA}. The fact that all subsequential limits of $\mathcal{L}^{\mathrm{hs}; \varpi_n}$ satisfy the pinned Brownian Gibbs property follows from a standard monotone class argument, with Lemma \ref{Lem.BrownianEnsembleConv} providing the key input. 

Theorems \ref{Thm.Convergence}, \ref{Thm.PfaffianStructure}, and \ref{Thm.GibbsProperty} follow readily from Proposition \ref{Prop.Tightness}, the point process convergence of Lemma \ref{Lem.PointProcessConvergenceVarpi}, and the finite-dimensional convergence of Lemma \ref{Lem.FDConvVarpi}. Finally, Theorem \ref{Thm.ConvergenceToAiryLE} is a consequence of \cite[Proposition 3.6, Theorem 3.8]{CorHamA}, the finite-dimensional convergence of Lemma \ref{Lem.FDConvT}, and the Brownian Gibbs property of the half-space Airy line ensemble established in \cite[Lemma 2.9]{DY25}.\\

In Section \ref{Section6}, we establish Theorem \ref{Thm.OriginDescription}. Although the approach is described in greater detail in the beginning of that section, we note here that the main difficulty lies in showing that the point processes on $\mathbb{R}$ formed by $\{\hsai_i(t_N)\}_{i \geq 1}$ converge, as $t_N \downarrow 0$, to $2M^{\mathrm{GSE}}$ (as defined in Lemma \ref{Lem.SAO}). By Theorem \ref{Thm.PfaffianStructure}, these point processes are Pfaffian for each fixed $t_N > 0$; however, the limit is not itself Pfaffian but a doubled Pfaffian point process. This precludes the use of standard convergence methods based on pointwise convergence of correlation kernels. At the level of formulas, the obstruction arises from the presence of a singular kernel term that admits only a distributional, rather than functional, limit. To overcome this difficulty, we work not with the kernels themselves but with the associated joint factorial moments. These moments can be expressed as integrals involving the kernel, and the additional integration sufficiently smooths out the singular contribution, allowing the asymptotic analysis to be carried out.

Lastly, Appendix \ref{SectionA} contains the proof of Lemma \ref{Lem.SAO}. This proof combines the Pfaffian point process structure of the Gaussian Symplectic Ensemble with its convergence to the stochastic Airy operator from \cite{RRV11}. 

\subsection*{Acknowledgments}
E.D. was partially supported by Simons Foundation International through Simons Award TSM-00014004. C.S.~acknowledges support from the Northern California Chapter of the ARCS Foundation.  Z.Y. was partially supported by Ivan Corwin's NSF grants DMS:1811143, DMS:2246576, Simons Foundation Grant 929852, and the Fernholz Foundation's `Summer Minerva Fellows' program.

%
%
\section{Kernel asymptotics}\label{Section2} The goal of this section is to derive an alternative formula for the kernel $\kcr$ from Definition \ref{Def.kcr} and analyze its $\varpi \rightarrow \infty$ limit (see Lemma \ref{Lem.KernelLimit}) --- this will be used later in the proof of Theorem \ref{Thm.Convergence}. In addition, we analyze the $T\rightarrow \infty$ limit of $\hski(T + \cdot, \cdot; T + \cdot, \cdot)$ from Theorem \ref{Thm.PfaffianStructure} (see Lemma \ref{Lem.ConvToAiryKernel}) --- this will be used later in the proof of Theorem \ref{Thm.ConvergenceToAiryLE}.

%
%
\subsection{Alternative kernel formula}\label{Section2.1} 
In this section we derive an alternative formula for the kernel $\kcr$ from Definition \ref{Def.kcr}, which is suitable for asymptotic analysis. The new formula is contained in the following definition.

\begin{definition}\label{Def.NewKernel} Fix $\varpi > 1$ and $s,t > 0$. With this data we define the kernel $\hsk$ by
\begin{equation}\label{Eq.DefKVarpi}
\begin{split}
&\hsk(s,x; t,y) = \begin{bmatrix}
    \hsk_{11}(s,x;t,y) & \hsk_{12}(s,x;t,y)\\
    \hsk_{21}(s,x;t,y) & \hsk_{22}(s,x;t,y) 
\end{bmatrix} \\
&= \begin{bmatrix}
    \hsi_{11}(s,x;t,y) & \hsi_{12}(s,x;t,y) + \hsr_{12}(s,x;t,y) \\
    -\hsi_{12}(t,y;s,x) - \hsr_{12}(t,y;s,x) & \hsi_{22}(s,x;t,y) + \hsr_{22}(s,x;t,y) 
\end{bmatrix},
\end{split}
\end{equation}
where the kernels $\hsi_{ij}, \hsr_{ij}$ are defined as follows. We have
\begin{equation}\label{Eq.DefI}
\begin{split}
\hsi_{11}(s,x;t,y) = &\frac{1}{(2\pi \im)^2} \int_{\mathcal{C}_{1+s}^{\pi/3}}dz \int_{\mathcal{C}_{1+t}^{\pi/3}} dw \frac{(z + s - w - t)(\varpi + z + s )(\varpi + w + t) H(z,x;w,y)}{(z + s + w + t)(z+s)(w+ t)}, \\
\hsi_{12}(s,x;t,y) = &\frac{1}{(2\pi \im)^2} \int_{\mathcal{C}_{1+s}^{\pi/3}}dz \int_{\mathcal{C}_{1+t}^{\pi/3}} dw \frac{(z + s - w + t)(\varpi + z + s) H(z,x;w,y)}{2(z+ s)(z + s + w -t)(\varpi - w + t)}, \\
\hsi_{22}(s,x;t,y) = &\frac{1}{(2\pi \im)^2} \int_{\mathcal{C}_{1+ s}^{\pi/3}}dz \int_{\mathcal{C}_{1+t}^{\pi/3}} dw \frac{(z - s - w + t)H(z,x;w,y)}{4 (z - s + w - t)(\varpi - z + s  )(\varpi - w + t)},
\end{split}
\end{equation}
where we have set $H(z,x;w,y) = e^{z^3/3 + w^3/3 - xz - y w}$ as in (\ref{Eq.DefH}) and the contours $\mathcal{C}_{z}^{\varphi}$ are as in Definition \ref{Def.S1Contours}. We also have $\hsr_{12} = R$ as in (\ref{Eq.S1AiryKer}) and
\begin{equation}\label{Eq.DefR22}
\begin{split}
\hsr_{22}(s,x;t,y) = \frac{e^{t^3/3 + s^3/3 -yt-xs} }{2\pi \im} \int_{\mathcal{C}^{\pi/2}_{0}} dw \frac{we^{ w^2(t+s) -w (t^2-s^2) + yw - xw }   }{2(w- \varpi)(w + \varpi) }.
\end{split}
\end{equation}
\end{definition}

\begin{lemma}\label{Lem.KernelMatch} Assume the same notation as in Definitions \ref{Def.kcr} and \ref{Def.NewKernel} for fixed $\varpi > 1$, and $s, t > 0$. Then, for Lebesgue a.e. $(x,y) \in \mathbb{R}^2$
\begin{equation}\label{Eq.KernelMatch}
\kcr_{ij}(s,x;t,y) = \hsk_{ij}(s,x;t,y) \mbox{ for } i,j = 1,2.
\end{equation}
\end{lemma}
\begin{proof} Notice that we can deform the contours $\mathcal{C}_{\ap_1 - s}^{\pi/3}$ and $\mathcal{C}_{\ap_3 - t}^{\pi/3}$ in the definition of $\icr_{11}$ to $\mathcal{C}_{1+s}^{\pi/3}$ and $\mathcal{C}_{1+t}^{\pi/3}$, respectively, without crossing any poles of the integrand, and hence without affecting the value of the integral by Cauchy's theorem. We mention that the decay necessary to deform the contours near infinity comes from the cubic terms in the exponential. When we perform the deformation we arrive at $\hsi_{11}$, which proves (\ref{Eq.KernelMatch}) when $i = j = 1$. 

Since $\kcr_{12}(s,x;t,y) = - \kcr_{21}(t,y; s,x)$ and $\rcr_{12} = \hsr_{12}$ by definition, we see that to prove (\ref{Eq.KernelMatch}) when $i \neq j$, it suffices to show that 
\begin{equation}\label{Eq.MatchI12}
\icr_{12}(s,x;t,y) = \hsi_{12}(s,x;t,y).
\end{equation}
As before, we deform the contours $\mathcal{C}_{\ap_3 - s}^{\pi/3}$ and $\mathcal{C}_{\ap_1 - t}^{2\pi/3}$ in the definition of $\icr_{12}$ to $\mathcal{C}_{1+s}^{\pi/3}$ and $\mathcal{C}_{-1-t}^{2\pi/3}$, respectively, without crossing any poles. By Cauchy's theorem, the value of the integral stays the same. Applying subsequently the change of variables $w \rightarrow -w$ leads us to (\ref{Eq.MatchI12}).\\

In the remainder of the proof we verify (\ref{Eq.KernelMatch}) for $i = j = 2$. Directly from the definitions, one observes that 
$$K(s,x;t,y) = - K(t,y;s,x) \mbox{ for } K = \icr_{22}, \rcr_{22}, \hsi_{22}, \hsr_{22}.$$
Consequently, it suffices to prove (\ref{Eq.KernelMatch}) when $x-s^2 > y - t^2$, which we assume in the sequel.

We start by deforming the contours $\mathcal{C}_{-\ap_2 - s}^{2\pi/3}$ and $\mathcal{C}_{-\ap_2 - t}^{2\pi/3}$ in the definition of $\icr_{22}$ to $\mathcal{C}_{-1 - s}^{2\pi/3}$ and $\mathcal{C}_{-1 - t}^{2\pi/3}$, respectively. In the process of deformation, we cross the simple poles at $z =  -\varpi-s$ and $w = -\varpi - t $, and so by the residue theorem
$$\icr_{22}(s,x;t,y) = \frac{1}{(2\pi \im)^2} \int_{\mathcal{C}_{-1 - s}^{2\pi/3}}dz \int_{\mathcal{C}_{-1 - t}^{2\pi/3}} dw \frac{(z + s - w - t)H(-z,x; -w, y)}{4 (z + s + w + t)(z + s  + \varpi )(w + t + \varpi)}  + A + B,$$
where 
$$A = - \frac{1}{2\pi \im} \int_{\mathcal{C}_{-1 - s}^{2\pi/3}}dz \frac{1 }{4 (z + s - \varpi)} \cdot e^{-z^3/3 + (\varpi + t)^3/3 + xz - y (\varpi + t)},$$
$$B = \frac{1}{2\pi \im} \int_{\mathcal{C}_{-1 - t}^{2\pi/3}} dw \frac{1}{4 ( w + t - \varpi )}  \cdot e^{(\varpi + s)^3/3 - w^3/3 - x(\varpi +s) + y w}.$$
In addition, we deform the contours $\mathcal{C}_{\ap_1 }^{2\pi/3}$ and $\mathcal{C}_{-\ap_2}^{2\pi/3}$ in the definition of $\rcr_{22}$ to $\mathcal{C}_{-1 }^{2\pi/3}$. In the process of deformation we cross the simple pole at $z = \varpi$ on the first line of (\ref{Eq.S1DefRcross2}), we do not cross any poles for the second line, and cross the simple pole at $w = -\varpi$ on the third line. By the residue theorem we conclude
\begin{equation*}
\begin{split}
&\rcr_{22}(s,x;t,y) = \frac{1}{2\pi \im} \int_{\mathcal{C}^{2\pi/3}_{-1}} dw \frac{we^{ (-w + t)^3/3 + (w + s)^3/3   - y (-w + t) - x (w+s)}   }{2(w- \varpi)(w + \varpi) } + C+ D+ E+F,
\end{split}
\end{equation*}
where
\begin{equation*}
\begin{split} 
&C =  \frac{1}{2\pi \im} \int_{\mathcal{C}_{-1 }^{2\pi/3}}dz \frac{e^{(-z + s)^3/3 + (\varpi + t)^3/3 - x (-z+s) - y (\varpi + t)}}{4(z - \varpi)} \\
&D = - \frac{1}{2\pi \im}\int_{\mathcal{C}_{-1 }^{2\pi/3}}dw \frac{e^{(-w + t)^3/3 + (\varpi + s)^3/3 - y (-w+t) - x (\varpi + s)}}{4(w - \varpi)}, \\
&E =  \frac{e^{(-\varpi + s)^3/3 + (\varpi + t)^3/3 - x (-\varpi+s) - y (\varpi + t)}}{4},\\
&F =  -  \frac{e^{ (\varpi + t)^3/3 + (-\varpi + s)^3/3   - y (\varpi + t) - x (-\varpi+s)}   }{4 }.
\end{split}
\end{equation*}
Using straightforward changes of variables we get
$$A + C = 0 , \hspace{2mm}  B+ D = 0 , \hspace{2mm} E + F = 0.$$ 
Combining the last several equations, we obtain
\begin{equation}\label{Eq.K22New}
\begin{split} 
\kcr_{22}(s,x;t,y) =  &\frac{1}{(2\pi \im)^2} \int_{\mathcal{C}_{-1 - s}^{2\pi/3}}dz \int_{\mathcal{C}_{-1 - t}^{2\pi/3}} dw \frac{(z + s - w - t)H(-z,x; -w, y)}{4 (z + s + w + t)(z + s  + \varpi )(w + t + \varpi)}  \\
& + \frac{1}{2\pi \im} \int_{\mathcal{C}^{2\pi/3}_{-1}} dw \frac{we^{ (-w + t)^3/3 + (w + s)^3/3   - y (-w + t) - x (w+s)}   }{2(w- \varpi)(w + \varpi) } .
\end{split}
\end{equation}
The first line in (\ref{Eq.K22New}) agrees with $\hsi_{22}(s,x;t,y)$ upon changing variables $z \rightarrow -z$ and $w \rightarrow -w$. In addition, the second line in (\ref{Eq.K22New}) agrees with $\hsr_{22}(s,x;t,y)$ once we simplify the polynomial in the exponential and deform $\mathcal{C}^{2\pi/3}_{-1}$ to $\mathcal{C}^{\pi/2}_{0}$ without crossing any poles. Consequently, (\ref{Eq.K22New}) verifies (\ref{Eq.KernelMatch}) for $i = j =2$, which completes the proof of the lemma.
\end{proof}

%
%
\subsection{Kernel asymptotics for $\varpi \rightarrow \infty$}\label{Section2.2} The goal of this section is to obtain the $\varpi \rightarrow \infty$ limit of the kernel $\hsk$ from Definition \ref{Def.NewKernel}. The precise statement is contained in the following lemma.

\begin{lemma}\label{Lem.KernelLimit} Fix $A,s,t > 0$, and let $\hski$, $\hsk$ be as in Theorem \ref{Thm.PfaffianStructure} and Definition \ref{Def.NewKernel}, respectively. Then,
\begin{equation}\label{Eq.KernelLimit}
\begin{split}
&\lim_{\varpi \rightarrow \infty} \frac{1}{4\varpi^2} \hsk_{11}(s,x;t,y) = \hski_{11}(s,x;t,y), \hspace{2mm} \lim_{\varpi \rightarrow \infty}\hsk_{12}(s,x;t,y) = \hski_{12}(s,x;t,y), \\
&\lim_{\varpi \rightarrow \infty}\hsk_{21}(s,x;t,y) = \hski_{21}(s,x;t,y), \hspace{2mm} \lim_{\varpi \rightarrow \infty} 4\varpi^2 \hsk_{22}(s,x;t,y) = \hski_{22}(s,x;t,y),
\end{split}
\end{equation}
where the convergence is uniform over $x, y \in [-A,A]$.  
\end{lemma}
\begin{proof} Throughout the proof $C_1$ is a large constant, which depends on $s,t,A$, and whose value changes from line to line.

By analyzing the real part of $z^3/3 + w^3/3 -xz - yw$, we can find $C_1 > 0$, such that for $x,y \in [-A,A]$, $\varpi \geq 2$, $z \in \mathcal{C}_{1+s}^{\pi/3}$, and $w \in \mathcal{C}_{1+t}^{\pi/3}$, we have
\begin{equation}\label{Eq.HBound}
|H(z,x;w,y)| \leq \exp(-|z|^3/3 - |w|^3/3 + C_1|z|^2 + C_1|w|^2 + C_1).
\end{equation}
In addition, using that the contours $\mathcal{C}^{\pi/3}_{1 +s}$ and $\mathcal{C}^{\pi/3}_{1+t}$ are well-separated from the corresponding poles, we can find $C_1 > 0$, such that for $\varpi \geq 2$, $z \in \mathcal{C}_{1+s}^{\pi/3}$, and $w \in \mathcal{C}_{1+t}^{\pi/3}$, we have
\begin{equation}\label{Eq.BoundRatFunsI11}
\begin{split}
&\left|\frac{(z + s - w - t)(\varpi + z + s )(\varpi + w + t)}{4\varpi^2(z + s + w + t)(z+s)(w+ t)}\right| \leq C_1 (1+|z|+|w|)^3,
\end{split}
\end{equation} 
\begin{equation}\label{Eq.BoundRatFunsI12}
\begin{split}
&\left|\frac{(z + s - w + t)(\varpi + z + s)}{2(z+ s)(z + s + w -t)(\varpi - w + t)}\right| \leq C_1 (1+|z|+|w|)^2,
\end{split}
\end{equation} 
\begin{equation}\label{Eq.BoundRatFunsI22}
\begin{split}
&\left|\frac{\varpi^2(z - s - w + t)}{(z - s + w - t)(\varpi - z + s  )(\varpi - w + t)}\right| \leq C_1 (1+|z|+|w|).
\end{split}
\end{equation} 

From Definition \ref{Def.NewKernel}
\[
\frac{1}{4\varpi^2}\hsi_{11}(s,x;t,y) = \frac{1}{(2\pi \im)^2} \int_{\mathcal{C}_{1+s}^{\pi/3}}dz \int_{\mathcal{C}_{1+t}^{\pi/3}} dw \frac{(z + s - w - t)(\varpi + z + s )(\varpi + w + t) H(z,x;w,y)}{4\varpi^2(z + s + w + t)(z+s)(w+ t)},
\]
and the integrand converges pointwise to that of $\hsii_{11}(s,x;t,y)$ from (\ref{Eq.DefII}). By the dominated convergence theorem, we conclude
$$\lim_{\varpi \rightarrow \infty} \frac{1}{4\varpi^2}\hsi_{11}(s,x;t,y)  = \hsii_{11}(s,x;t,y).$$
We mention that the application of the dominated convergence theorem is justified by the bounds in (\ref{Eq.HBound}) and (\ref{Eq.BoundRatFunsI11}). Applying the dominated convergence theorem twice more, using (\ref{Eq.HBound}), (\ref{Eq.BoundRatFunsI12}), and (\ref{Eq.BoundRatFunsI22}), we conclude
$$\lim_{\varpi \rightarrow \infty} \hsi_{12}(s,x;t,y)  = \hsii_{12}(s,x;t,y) \mbox{ and } \lim_{\varpi \rightarrow \infty} 4\varpi^2 \hsi_{22}(s,x;t,y)  = \hsii_{22}(s,x;t,y).$$

From the last two displayed equations, the fact that $\hsr_{12}=R=\hsri_{12}$ as in (\ref{Eq.S1AiryKer}), and the definition of $\hsr_{22}$, we see that to prove (\ref{Eq.KernelLimit}) it remains to show that 
\begin{equation}\label{Eq.ConvR22Varpi}
\lim_{\varpi \rightarrow \infty} \frac{H(s,x;t,y)}{ \pi \im} \int_{\mathcal{C}^{\pi/2}_{0}} dw \frac{w\varpi^2e^{ w^2(t+s) -w (t^2-s^2) + yw - xw }   }{ (w- \varpi)(w + \varpi) } = \hsri_{22}(s,x;t,y).
\end{equation}

By analyzing the real part of the exponent, we can find $C_1 > 0$, such that for $\varpi \geq 2$ and $w \in \mathcal{C}_0^{\pi/2}$, we have
$$\left|\frac{w\varpi^2e^{ w^2(t+s) -w (t^2-s^2) + yw - xw }   }{ (w- \varpi)(w + \varpi) } \right| \leq (1+|w|)\cdot \exp(C_1 + C_1|w| - (s+t)|w|^2).$$
By the dominated convergence theorem, we conclude
\begin{equation*}
\begin{split}
&\lim_{\varpi \rightarrow \infty} \frac{H(s,x;t,y)}{ \pi \im} \int_{\mathcal{C}^{\pi/2}_{0}} dw \frac{w\varpi^2e^{ w^2(t+s) -w (t^2-s^2) + yw - xw }   }{ (w- \varpi)(w + \varpi) } \\
&= -\frac{H(s,x;t,y)}{ \pi \im} \int_{\mathcal{C}^{\pi/2}_{0}} dw \cdot we^{ w^2(t+s) -w (t^2-s^2) + yw - xw }. 
\end{split}
\end{equation*}
Consequently, to establish (\ref{Eq.ConvR22Varpi}), it remains to show that 
\begin{equation}\label{Eq.AltFormulaR22}
\begin{split}
 \hsri_{22}(s,x;t,y) = -\frac{H(s,x;t,y)}{ \pi \im} \int_{\mathcal{C}^{\pi/2}_{0}} dw \cdot we^{ w^2(t+s) -w (t^2-s^2) + yw - xw }. 
\end{split}
\end{equation}

If we change variables $w = \frac{\im \sqrt{\pi}}{\sqrt{t+s}} \cdot u$, we get
\begin{equation*}
\begin{split}
&\int_{\mathcal{C}^{\pi/2}_{0}} we^{ w^2(t+s) -w (t^2-s^2) + yw - xw } dw = \frac{- \pi}{t+s}\int_{\mathbb{R}} u \exp\left( -\pi u^2 - 2\pi \im u \left[ \frac{t^2-s^2}{2\sqrt{\pi(t+s)}} + \frac{x-y}{2\sqrt{\pi(t+s)}} \right]  \right) \\
& = \frac{1}{2 \im (t+s)} \cdot \frac{d}{d\xi} F(\xi)\vert_{\xi = \xi_0}, \mbox{ where } F(\xi) = e^{-\pi \xi^2} \mbox{ and } \xi_0 = \frac{t^2-s^2 + x - y}{2\sqrt{\pi(t+s)}}. 
\end{split}
\end{equation*}
In going from the first to the second line we used basic properties of the Fourier transform, see Proposition 1.2(v) and Theorem 1.4 in \cite[Chapter 5]{SteinFourier}.

The last displayed equation and the definition of $\hsri_{22}$ in (\ref{Eq.DefR22I}) give (\ref{Eq.AltFormulaR22}).
\end{proof}

%
%
\subsection{Kernel asymptotics for $T \rightarrow \infty$}\label{Section2.3} The goal of this section is to obtain the $T \rightarrow \infty$ limit of the kernel $\hski(T + \cdot, \cdot; T + \cdot, \cdot)$ from Theorem \ref{Thm.PfaffianStructure}. The precise statement is contained in the following lemma.

\begin{lemma}\label{Lem.ConvToAiryKernel} Fix $A > 0$, $s, t \in \mathbb{R}$, and let $\hski$ be as in Theorem \ref{Thm.PfaffianStructure}. Then, the following limits all hold uniformly over $x,y \in [-A,A]$
\begin{equation}\label{Eq.LimToAiry11}
\begin{split}
\lim_{T\rightarrow\infty} \hski_{11} ( T + s ,x; T + t, y)= 0,
\end{split}
\end{equation}
\begin{equation}\label{Eq.LimToAiry12}
\lim_{T \rightarrow \infty} \hski_{12}(T+ s,x;T+ t,y) =K^{\mathrm{Airy}}(s,x; t,y),
\end{equation}
\begin{equation}\label{Eq.LimToAiry22}
\lim_{T \rightarrow \infty} \hski_{22}(T+s,x;T+t,y)  = 0,
\end{equation}
where we recall that $K^{\mathrm{Airy}}$ is as in (\ref{Eq.S1AiryKer}).
\end{lemma}
\begin{proof} Throughout the proof $C_2$ is a large constant, which depends on $s,t,A$, and whose value changes from line to line.

We first prove \eqref{Eq.LimToAiry11}. From (\ref{Eq.DefK}), we have $\hski_{11} =\hsii_{11}$, and so from (\ref{Eq.DefII}), we have
\begin{equation*}
    \begin{split}
         \hski_{11} ( T + s ,x; T + t, y) =\frac{1}{(2\pi \im)^2} \int_{\mathcal{C}_{1+s}^{\pi/3}}dz \int_{\mathcal{C}_{1 + t}^{\pi/3}} dw \frac{(z + s - w - t)H(z,x;w,y)}{4(z + s + w + t+2T)(z+s+T)(w+ t+T)},
    \end{split}
\end{equation*} 
where we deformed the contours $\mathcal{C}_{1+s+T}^{\pi/3}$, $\mathcal{C}_{1+t+T}^{\pi/3}$ to $\mathcal{C}_{1+s}^{\pi/3}$, $\mathcal{C}_{1+t}^{\pi/3}$, respectively, without crossing any poles of the integrand, and hence without affecting the value of the integral by Cauchy's theorem.

Notice that for $z \in \mathcal{C}_{1+s}^{\pi/3}$, $w \in \mathcal{C}_{1+t}^{\pi/3}$, we have
$$\left|\frac{(z + s - w - t)}{4(z + s + w + t+2T)(z+s+T)(w+ t+T)}\right| \leq \frac{|z| + |w| + |s| + |t|}{4T^3}.$$
The latter and (\ref{Eq.HBound}) imply for some $C_2 > 0$ 
\begin{equation*}
\begin{split}
&\left| \hski_{11} ( T + s ,x; T + t, y) \right| \leq \frac{1}{4T^3} \cdot \int_{\mathcal{C}_{1+s}^{\pi/3}}|dz| \int_{\mathcal{C}_{1 + t}^{\pi/3}} |dw| (|z| + |w| + |s| + |t|) \\
& \times e^{-|z|^3/3 - |w|^3/3 + C_1|z|^2 + C_1|w|^2 + C_1} \leq C_2 \cdot T^{-3},
\end{split}
\end{equation*}
which implies \eqref{Eq.LimToAiry11}. We mention that in the last integral $|dz|$ and $|dw|$ denote integration with respect to arc-length.\\

We next prove \eqref{Eq.LimToAiry12}. In view of the definition of $K^{\mathrm{Airy}}$ and the equality $\hsri_{12}(T+s,x;T+t,y)=R(s,x;t,y)$, see (\ref{Eq.S1AiryKer}), we only need to prove   that uniformly for $x, y \in [-A,A]$,  
\begin{equation}\label{Eq.LimToI12Only} 
\lim_{T \rightarrow \infty} \hsii_{12}(T+ s,x;T+ t,y)=\frac{1}{(2\pi \im)^2} \int_{\mathcal{C}_{\alpha}^{\pi/3}} d z \int_{\mathcal{C}_{\beta}^{\pi/3}} dw \frac{H(z,x;w,y)}{z + s + w - t},
\end{equation}
where $\alpha + s + \beta - t > 0$. For concreteness, put $\alpha = 1- s$ and $\beta = t + 1$.

In (\ref{Eq.DefII}) we deform $\mathcal{C}_{1+s+T}^{\pi/3}$, $\mathcal{C}_{1+t+T}^{\pi/3}$ to $\mathcal{C}_{\alpha}^{\pi/3}$, $\mathcal{C}_{\beta}^{\pi/3}$, respectively, without crossing any poles, to get
 \begin{equation*}
\begin{split}
\hsii_{12}(T+ s,x;T+ t,y)=\frac{1}{(2\pi \im)^2} \int_{\mathcal{C}_{\alpha}^{\pi/3}} d z \int_{\mathcal{C}_{\beta}^{\pi/3}} dw \frac{(z + s - w + t+2T) H(z,x;w,y)}{2(z+ s+T)(z + s + w -t)}.
\end{split}
\end{equation*} 

Note that the above integrand converges pointwise to the one on the right side of (\ref{Eq.LimToI12Only}). In addition, using (\ref{Eq.HBound}), we conclude for some $C_2 > 0$ and all $z \in \mathcal{C}_{\alpha}^{\pi/3}, w \in  \mathcal{C}_{\beta}^{\pi/3}$
$$\left|\frac{(z + s - w + t+2T) H(z,x;w,y)}{2(z+ s+T)(z + s + w -t)}\right| \leq C_2 (1 + |z| + |w|) \cdot e^{-|z|^3/3 - |w|^3/3 + C_1|z|^2 + C_1|w|^2 + C_1}.$$
By the dominated convergence theorem, we conclude (\ref{Eq.LimToI12Only}).\\

We finally prove \eqref{Eq.LimToAiry22}. We begin by proving an alternative formula for $\hski_{22}$, see (\ref{Eq.K22VertCont}), and then take the $T \rightarrow \infty$ limit in this new formula. 

In what follows we assume $T$ is large so that $T > -s$, $T > -t$, $T > 2- s$, $2T > 2- s - t$. In (\ref{Eq.DefII}) we deform $\mathcal{C}_{1+s+T}^{\pi/3}$, $\mathcal{C}_{1+t+T}^{\pi/3}$ to $\mathcal{C}_{1 + s + T}^{\pi/2}$, $\mathcal{C}_{1 + t + T}^{\pi/2}$, respectively, without crossing any poles, to get
\begin{equation}\label{Eq.I22VertCont}
\hsii_{22}(s + T ,x;t + T,y) = \frac{1}{(2\pi \im)^2} \int_{\mathcal{C}_{1+ s + T}^{\pi/2}}dz \int_{\mathcal{C}_{1+t + T}^{\pi/2}} dw \frac{(z - s - w + t)H(z,x;w,y)}{z - s + w - t -2T}.
\end{equation}
We mention that the deformation near infinity is justified, since for $\zeta = c + d e^{\im \phi}$ with $c > 0$, $d \geq 0$ and $\phi \in [\pi/3, \pi/2] \cup [-\pi/2, -\pi/3]$, we have 
\begin{equation}\label{Eq.CubicBoundVert}
\begin{split}
 &\left| e^{\zeta^3/3}\right| = e^{\Real[\zeta^3/3]} = e^{c^3/3 + [d^3/3]\cos(\phi) [\cos^2(\phi) - 3 \sin^2(\phi)] + cd^2(\cos^2(\phi) - \sin^2(\phi)) + c^2d \cos(\phi)} \\
 & \leq e^{c^3/3 - cd^2/2 + c^2d}.
\end{split}
\end{equation}

We now proceed to deform the $z$-contour $\mathcal{C}_{1+ s + T}^{\pi/2}$ in (\ref{Eq.I22VertCont}) to $\mathcal{C}_{1}^{\pi/2}$. In the process of deformation we cross the simple pole at $z = 2T + s + t - w$ (here we used $T > 2-s$). Afterwards we deform the $w$-contour $\mathcal{C}_{1+ t + T}^{\pi/2}$ to $\mathcal{C}_{1}^{\pi/2}$ without crossing any poles (here we used $2T > 2 - s - t$). By the residue theorem, we conclude 
\begin{equation}\label{Eq.I22VertCont2}
\begin{split}
&\hsii_{22}(s + T ,x;t + T,y) = \frac{1}{(2\pi \im)^2} \int_{\mathcal{C}_{1}^{\pi/2}}dz \int_{\mathcal{C}_{1}^{\pi/2}} dw \frac{(z - s - w + t)H(z,x;w,y)}{z - s + w - t -2T} \\
& + \frac{1}{2\pi \im} \int_{\mathcal{C}_{1 + t + T}^{\pi/2}} dw (2T - 2w + 2t)H(s + t + 2T -w,x;w,y).
\end{split}
\end{equation}

Changing variables $w \rightarrow -w$ in (\ref{Eq.AltFormulaR22}) and deforming the contour $\mathcal{C}^{\pi/2}_{0}$ to $\mathcal{C}^{\pi/2}_{1}$, we obtain 
\begin{equation}\label{Eq.R22VertCont2}
\begin{split}
 &\hsri_{22}(s + T,x;t + T,y) = \frac{H(s + T,x;t + T,y)}{ \pi \im} \\
 & \times \int_{\mathcal{C}^{\pi/2}_{1}} dw \cdot we^{ w^2(t+s + 2T) + w (t^2 - s^2 + 2Tt - 2Ts) - yw + xw }. 
\end{split}
\end{equation}

We now observe (upon changing variables $w \rightarrow w + T + t$) that the second line in (\ref{Eq.I22VertCont2}) is precisely $-\hsri_{22}(s + T,x;t + T,y)$. Consequently, if we add (\ref{Eq.I22VertCont2}) and (\ref{Eq.R22VertCont2}), we conclude 
\begin{equation}\label{Eq.K22VertCont}
\begin{split}
&\hski_{22}(s + T ,x;t + T,y) = \frac{1}{(2\pi \im)^2} \int_{\mathcal{C}_{1}^{\pi/2}}dz \int_{\mathcal{C}_{1}^{\pi/2}} dw \frac{(z - s - w + t)H(z,x;w,y)}{z - s + w - t -2T}.
\end{split}
\end{equation}
This is our desired formula, and below we take $T \rightarrow \infty$ in it.\\

From (\ref{Eq.K22VertCont}) and (\ref{Eq.CubicBoundVert}) with $c = 1$, we have for some $C_2 > 0$ and all large $T$
\begin{equation*}
\begin{split}
&\left|\hski_{22}(s + T ,x;t + T,y) \right| \leq  \frac{1}{T}\int_{\mathcal{C}_{1}^{\pi/2}}|dz| \int_{\mathcal{C}_{1}^{\pi/2}} |dw| (|z| + |w| + |s| + |t|) \\
& \times e^{C_2-|z-1|^2/2 - |w-1|^2/2 + C_2|z| + C_2|w|},
\end{split}
\end{equation*}
which implies (\ref{Eq.LimToAiry22}).
\end{proof}

%
%
\section{Finite-dimensional convergence}\label{Section3} In this section we establish the finite-dimensional analogues of Theorem \ref{Thm.Convergence} (see Lemma \ref{Lem.FDConvVarpi}) and Theorem \ref{Thm.ConvergenceToAiryLE} (see Lemma \ref{Lem.FDConvT}).

%
%
\subsection{Tightness from above}\label{Section3.1} The goal of this section is to show that for each $t > 0$ the variables $\{\hsa_{1}(t): \varpi \geq 2\}$ are tight from above. The precise statement is contained in the following lemma.

\begin{lemma}\label{Lem.TightFromAbove} Let $\hsa$ be as in Proposition \ref{Prop.HSAiry}. Then we have
\begin{equation}\label{Eq.UpperTailFixedT}
\lim_{a \rightarrow \infty} \sup_{t > 0} \sup_{\varpi \geq 2} \mathbb{P}\left( \hsa_1(t) \geq a \right) = 0.
\end{equation}
\end{lemma}
\begin{proof} 
For $t > 0$ and $\varpi \in \mathbb{R}$, define the random measure on $\mathbb{R}$ by
\begin{equation}\label{Eq.RMSlice}
M^{t; \mathrm{hs};\varpi}(A) = \sum_{ i \geq 1} {\bf 1}\{\hsa_i(t) \in A\}.
\end{equation}
By Proposition \ref{Prop.HSAiry}, Lemma \ref{Lem.KernelMatch} and \cite[Lemma 5.13]{DY25}, for all $\varpi > 1$, $M^{t; \mathrm{hs};\varpi}$ is a Pfaffian point process on $\mathbb{R}$ with reference measure $\mathrm{Leb}$ and correlation kernel $K^{\varpi} (t,x;t,y)$.
Therefore
\begin{equation}\label{eq.in proof the upper tail number of atoms}
\sum_{i\geq1}\mathbb{P}\left( \hsa_i(t) \geq a \right)=\mathbb{E}\left[\sum_{i\geq1} \mathbf{1} \{\hsa_i(t)\in[a,\infty) \} \right]=\mathbb{E}\left[M^{t; \mathrm{hs};\varpi}([a,\infty))\right].
\end{equation}
From \cite[(2.13)]{dimitrov2024airy} and \cite[(5.12)]{DY25}, we obtain
\begin{equation*}
    \begin{split}
        \mathbb{E}\left[M^{t; \mathrm{hs};\varpi}([a,\infty))\right]&=\int_a^{\infty} K_{12}^{\varpi} (t,x;t,x)dx\\
&=\frac{1}{(2\pi \im)^2}\int_a^{\infty}dx \int_{\mathcal{C}_{1+t}^{\pi/3}}dz \int_{\mathcal{C}_{1+t}^{\pi/3}} dw \frac{(z  - w+ 2t  )(\varpi + z + t) H(z,x;w,x)}{2(z+ t)(z  + w  )(\varpi - w + t)}\\
&=\frac{1}{(2\pi \im)^2}\int_a^{\infty}dx \int_{\mathcal{C}_{1 }^{\pi/3}}dz \int_{\mathcal{C}_{1 }^{\pi/3}} dw \frac{(z  - w+ 2t  )(\varpi + z + t) H(z,x;w,x)}{2(z+ t)(z  + w  )(\varpi - w + t)}\\
&=\frac{1}{(2\pi \im)^2}  \int_{\mathcal{C}_{1 }^{\pi/3}}dz \int_{\mathcal{C}_{1 }^{\pi/3}} dw \frac{(z  - w+ 2t  )(\varpi + z + t)\cdot e^{z^3/3+w^3/3}  }{2(z+ t)(z  + w  )(\varpi - w + t)}\cdot \int_a^{\infty}e^{-(z+w)x}dx\\
&=\frac{1}{(2\pi \im)^2}  \int_{\mathcal{C}_{1 }^{\pi/3}}dz \int_{\mathcal{C}_{1 }^{\pi/3}} dw \frac{(z  - w+ 2t  )(\varpi + z + t)\cdot e^{z^3/3+w^3/3}  }{2(z+ t)(z  + w  )^2(\varpi - w + t)}\cdot e^{-(z+w)a},
    \end{split}
\end{equation*}
where in the second step we used the definition of $K_{12}^{\varpi}$ and notice that $R_{12}^{\varpi}=R=0$ at $(t,x;t,x)$; in the third step we deformed the contour $\mathcal{C}_{1+t}^{\pi/3}$ for both $z$ and $w$ to $\mathcal{C}_{1 }^{\pi/3}$ without crossing any poles, since $\varpi>1$; in the fourth step we swapped the order of the integrals, using Fubini's theorem, justified by the cubic decay in the term $ e^{z^3/3+w^3/3} $ and the fact that $|e^{-(z+w)x}|\leq e^{-2x}$ for $x>0$ and $z,w\in\mathcal{C}_{1 }^{\pi/3}$; in the last step we evaluated the integral with respect to $x$. 

We next bound the absolute value of the integrand. For any $z,w\in\mathcal{C}_{1 }^{\pi/3}$, $t>0$ and $\varpi\geq2$, 
\begin{align*}
    &\left|\frac{z  - w+ 2t }{2(z+ t)}\right|\leq1+\frac{|w+z|}{|2(z+t)|}\leq1+\frac{1}{2}|w+z| \leq 1 + 2|z| + 2|w|,\\
    &\left|\frac{\varpi + z + t}{\varpi - w + t}\right|\leq1+\frac{|z+w|}{|\varpi - w + t|}\leq1+\frac{2}{\sqrt{3}}|z+w| \leq 1 + 2|z| + 2|w|.
\end{align*}
In the first line we used $|2(z+t)|\geq2$, and the second we used $|\varpi - w + t|\geq \sqrt{3}/2$. We also notice that $|e^{-(z+w)a}|\leq e^{-2a}$ for $a>0$. Therefore, for $a>0$ we have
\[
\sup_{t > 0} \sup_{\varpi \geq 2}\mathbb{E}\left[M^{t; \mathrm{hs};\varpi}([a,\infty))\right]\leq\frac{e^{-2a}}{4\pi^2}  \int_{\mathcal{C}_{1 }^{\pi/3}}|dz| \int_{\mathcal{C}_{1 }^{\pi/3}} |dw|\left(1+2|z| + 2|w|\right)^2\cdot\left|\frac{e^{z^3/3+w^3/3}}{(z+w)^2}\right| ,
\]
which converges to $0$ as $a\rightarrow\infty$, since the integral is finite. Using \eqref{eq.in proof the upper tail number of atoms}, we conclude the proof. 
\end{proof}

%
%
\subsection{Point process convergence}\label{Section3.2} In this section we show that the point processes $\hsm$ from Proposition \ref{Prop.HSAiry} converge weakly to a Pfaffian point process $M^{\mathsf{S};\infty}$ as $\varpi \rightarrow \infty$. Subsequently, we show that, when appropriately shifted, $M^{T + \mathsf{S};\infty}$ converges to the extended Airy point process $M^{\mathsf{S}; \mathcal{A}}$ from (\ref{Eq.RMS1}) as $T \rightarrow \infty$. Our proofs rely on combining our kernel convergence results from Sections \ref{Section2.2} and \ref{Section2.3} with a few statements from \cite{DY25}.

\begin{lemma}\label{Lem.PointProcessConvergenceVarpi} Assume the same notation as in Proposition \ref{Prop.HSAiry} and Theorem \ref{Thm.PfaffianStructure}. Then, the following statements hold.
\begin{enumerate}
\item[(a)] Fix a finite set $\mathsf{S} = \{s_1, \dots, s_m\} \subset (0, \infty)$ with $s_1 < s_2 < \cdots < s_m$. Then $\hsm \Rightarrow M^{\mathsf{S}; \infty}$ as $\varpi \rightarrow \infty$, where $M^{\mathsf{S};\infty}$ is a Pfaffian point process on $\mathbb{R}^2$ with reference measure $\mu_{\mathsf{S}} \times \mathrm{Leb}$ and correlation kernel $\hski$.
\item[(b)] For a fixed $t > 0$ and $\varpi \in \mathbb{R}$, let $M^{t; \mathrm{hs};\varpi}$ be as in (\ref{Eq.RMSlice}). Then $M^{t; \mathrm{hs};\varpi} \Rightarrow M^{t; \infty}$ as $\varpi \rightarrow \infty$, where $M^{t;\infty}$ is a Pfaffian point process on $\mathbb{R}$ with reference measure $ \mathrm{Leb}$ and correlation kernel $K^{t;\infty}(x,y) = \hski(t,x; t,y)$.
\end{enumerate}
\end{lemma}
\begin{proof} From Proposition \ref{Prop.HSAiry}, Lemma \ref{Lem.KernelMatch} and \cite[Proposition 5.8(4)]{DY25} with $f(x) = 1/(2 \varpi)$ we have for all $\varpi > 1$ that $\hsm$ is a Pfaffian point process on $\mathbb{R}^2$ with reference measure $\mu_{\mathsf{S}}\times \mathrm{Leb}$ and correlation kernel 
$$\begin{bmatrix} \frac{1}{4\varpi^2} K^{\varpi}_{11}(s,x;t,y) & K^{\varpi}_{12}(s,x;t,y) \\ K_{21}^{\varpi}(s,x;t,y) & 4 \varpi^2 K_{22}^{\varpi}(s,x;t,y) \end{bmatrix}.$$
Part (a) now follows from \cite[Proposition 5.10]{DY25}, whose conditions are satisfied in view of Lemma \ref{Lem.KernelLimit} and the continuity of $\hski_{ij}(s,x;t,y)$ over $(x,y) \in \mathbb{R}^2$ for fixed $s,t > 0$.\\

From Proposition \ref{Prop.HSAiry}, Lemma \ref{Lem.KernelMatch}, \cite[Lemma 5.13]{DY25}, and \cite[Proposition 5.8(4)]{DY25} with $f(x) = 1/(2 \varpi)$ we have for all $\varpi > 1$ that $M^{t; \mathrm{hs};\varpi}$ is a Pfaffian point process on $\mathbb{R}$ with reference measure $\mathrm{Leb}$ and correlation kernel 
$$\begin{bmatrix} \frac{1}{4\varpi^2} K^{\varpi}_{11}(t,x;t,y) & K^{\varpi}_{12}(t,x;t,y) \\ K_{21}^{\varpi}(t,x;t,y) & 4 \varpi^2 K_{22}^{\varpi}(t,x;t,y) \end{bmatrix}.$$
Part (b) now follows from Lemma \ref{Lem.KernelLimit} and \cite[Proposition 5.10]{DY25}.
\end{proof}

\begin{lemma}\label{Lem.PointProcessConvergenceT} Assume the same notation as in Definition \ref{Def.AiryLE} and Lemma \ref{Lem.PointProcessConvergenceVarpi}. Then, the following statements hold.
\begin{enumerate}
\item[(a)] Fix a finite set $\mathsf{S} = \{s_1, \dots, s_m\} \subset \mathbb{R}$ with $s_1 < s_2 < \cdots < s_m$, and set $T + \mathsf{S} := \{T+s_1, T + s_2,  \dots, T+ s_m\}$. If $\phi_T(s,x) = (s-T, x)$, then $M^{T+ \mathsf{S};\infty}\phi_T^{-1} \Rightarrow M^{\mathsf{S}; \mathcal{A}}$ as $T \rightarrow \infty$.
\item[(b)] For a fixed $t_0 \in \mathbb{R}$ define the random measure on $\mathbb{R}$ through
\begin{equation}
M^{t_0; \mathcal{A}}(A) = \sum_{ i \geq 1} {\bf 1}\{\mathcal{A}_i(t_0) \in A\}.
\end{equation}
Then $M^{T + t_0;\infty} \Rightarrow M^{t_0; \mathcal{A}}$ as $T \rightarrow \infty$.
\end{enumerate}
\end{lemma}
\begin{proof} From Lemma \ref{Lem.PointProcessConvergenceVarpi}(a) and \cite[Proposition 5.8(5)]{DY25}, we have for $T > -s_1$ that $M^{T+ \mathsf{S};\infty}\phi_T^{-1}$ is a Pfaffian point process on $\mathbb{R}^2$ with reference measure $\mu_{\mathsf{S}} \times \mathrm{Leb}$ and correlation kernel
$$ \begin{bmatrix} \hski_{11}(T+s, x; T+t, y) & \hski_{12}(T+s, x; T+t, y) \\ \hski_{21}(T+s, x; T+t, y) & \hski_{22}(T+s, x; T+t, y) \end{bmatrix}.$$
From Lemma \ref{Lem.ConvToAiryKernel} and \cite[Proposition 5.10]{DY25}, we conclude that $M^{T+ \mathsf{S};\infty}\phi_T^{-1}$ converges to a Pfaffian point process on $\mathbb{R}^2$ with reference measure $\mu_{\mathsf{S}} \times \mathrm{Leb}$ and correlation kernel 
$$\begin{bmatrix} 0 & K^{\mathrm{Airy}}(s,x;t,y) \\ -K^{\mathrm{Airy}}(t,y;s,x) & 0 \end{bmatrix}.$$
From \cite[Lemma 5.9]{DY25} and \cite[Proposition 2.13(3)]{dimitrov2024airy} we conclude that this limiting point process has the same law as $M^{\mathsf{S}; \mathcal{A}}$. This proves part (a).\\

From Lemma \ref{Lem.PointProcessConvergenceVarpi}(b), \cite[Lemma 5.13]{DY25} and \cite[Proposition 5.8(5)]{DY25}, we have for $T > -t_0$ that $M^{T + t_0;\infty}$ is a Pfaffian point process on $\mathbb{R}$ with reference measure $\mathrm{Leb}$ and correlation kernel
$$ \begin{bmatrix} \hski_{11}(T+t_0, x; T+t_0, y) & \hski_{12}(T+t_0, x; T+t_0, y) \\ \hski_{21}(T+t_0, x; T+t_0, y) & \hski_{22}(T+t_0, x; T+t_0, y) \end{bmatrix}.$$
From Lemma \ref{Lem.ConvToAiryKernel} and \cite[Proposition 5.10]{DY25}, we conclude that $M^{T+ t_0;\infty}$ converges to a Pfaffian point process on $\mathbb{R}$ with reference measure $\mathrm{Leb}$ and correlation kernel 
$$\begin{bmatrix} 0 & K^{\mathrm{Airy}}(t_0,x;t_0,y) \\ -K^{\mathrm{Airy}}(t_0,y;t_0,x) & 0 \end{bmatrix}.$$
From \cite[Lemma 5.9]{DY25} and \cite[Proposition 2.13(3)]{dimitrov2024airy}, we conclude that this limiting point process has the same law as $M^{t_0; \mathcal{A}}$. This proves part (b).
\end{proof}

%
%
\subsection{Finite-dimensional convergence of $\hsa$}\label{Section3.3} The goal of this section is to prove the finite-dimensional version of Theorem \ref{Thm.Convergence} --- this is Lemma \ref{Lem.FDConvVarpi} below. In order to establish that result we require the following auxiliary lemma whose proof is postponed until Section \ref{SectionSpecial}.
\begin{lemma}\label{Lem.InfinitelyManyAtoms} For a fixed $t > 0$ let $M^{t;\infty}$ be as in Lemma \ref{Lem.PointProcessConvergenceVarpi}(b). Then $\mathbb{P}\left(M^{t;\infty}(\mathbb{R}) = \infty \right) = 1$.
\end{lemma}

\begin{lemma}\label{Lem.FDConvVarpi} Assume the same notation as in Proposition \ref{Prop.HSAiry}. Fix a finite set $\mathsf{S} = \{s_1, \dots, s_m\} \subset (0, \infty)$ with $s_1 < s_2 < \cdots < s_m$. Then, for some $X^{\mathsf{S}} = \left( X^{j,\mathsf{S}}_i: i \geq 1, j = 1, \dots, m\right) \in \mathbb{R}^{\infty}$
$$\left( \hsa_i(s_j): i \geq 1, j = 1, \dots, m\right) \overset{f.d.}{\rightarrow} \left( X^{j,\mathsf{S}}_i: i \geq 1, j = 1, \dots, m\right)$$
 as $\varpi \rightarrow \infty$. In addition, we have that the vector $X^{\mathsf{S}}$ satisfies almost surely
\begin{equation}\label{Eq.OrdX}
X^{j, \mathsf{S}}_1 > X^{j, \mathsf{S}}_2 > \cdots \mbox{ for } j = 1, \dots, m,
\end{equation}
and if we define the random measure $M^{\mathsf{S}; X}$ on $\mathbb{R}^2$ through 
\begin{equation}\label{Eq.MultiSliceFormedX}
M^{\mathsf{S};X}(A) = \sum_{i \geq 1} \sum_{j = 1}^m {\bf 1}\{(s_j, X^{j,\mathsf{S}}_i) \in A \},
\end{equation}
then $M^{\mathsf{S};X}$ is a point process that has the same distribution as $M^{\mathsf{S}; \infty}$ from Lemma \ref{Lem.PointProcessConvergenceVarpi}(a).
\end{lemma}
\begin{remark}\label{Rem.FDConvVarpi} From \cite[Corollary 2.20]{dimitrov2024airy} we have that (\ref{Eq.OrdX}) and $M^{\mathsf{S};X} \overset{d}{=} M^{\mathsf{S}; \infty}$ uniquely specify the distribution of $X^{\mathsf{S}}$.
\end{remark}
\begin{proof} We only need to prove the result for each sequence $\varpi_k\in\mathbb{R}$, $k=1,2,\dots$ satisfying $ \varpi_k\rightarrow\infty$ as $k\rightarrow\infty$, since the law of the limit vector $X^{\mathsf{S}}$ is the same for each such sequence by Remark \ref{Rem.FDConvVarpi}. 

Suppose that for each $i \geq 1$ and $j = 1, \dots, m$, we know that 
\begin{equation}\label{Eq.AISTIGHT}
\{\mathcal{A}_i^{\mathrm{hs};\varpi_k}(s_j)\}_{k \geq 1} \mbox{ is tight.}
\end{equation}
From Lemma \ref{Lem.PointProcessConvergenceVarpi}(a) we know that $M^{\mathsf{S}; \mathrm{hs};\varpi_k}\Rightarrow M^{\mathsf{S}; \infty}$ as $k\rightarrow\infty$. The latter and (\ref{Eq.AISTIGHT}) verify the conditions of \cite[Proposition 2.19]{dimitrov2024airy}, which implies the statement of the lemma, with the exception that it only ensures
$$X^{j, \mathsf{S}}_1 \geq X^{j, \mathsf{S}}_2 \geq \cdots \mbox{ for } j = 1, \dots, m.$$
Let us briefly explain how to get (\ref{Eq.OrdX}) from here. Our work so far shows $M^{\mathsf{S};X} \overset{d}{=} M^{\mathsf{S}; \infty}$, which by Lemma \ref{Lem.PointProcessConvergenceVarpi}(b) is a Pfaffian point process on $\mathbb{R}^2$. By \cite[Proposition 5.8(1)]{DY25}, we conclude $M^{\mathsf{S}; X}$ is a simple point process, and so almost surely $\{X^{j,\mathsf{S}}_{i}\}_{i \geq 1}$ are all distinct, implying (\ref{Eq.OrdX}).\\

In the remainder of the proof, we establish (\ref{Eq.AISTIGHT}). Fix $j \in \{1, \dots, m\}$. In order to show (\ref{Eq.AISTIGHT}), we seek to apply \cite[Proposition 2.21]{dimitrov2024airy} to the sequence $(\mathcal{A}_i^{\mathrm{hs};\varpi_k}(s_j):i\geq1)$, $k=1,2,\dots$ of random vectors in $\mathbb{R}^{\infty}$.

By Lemma \ref{Lem.PointProcessConvergenceVarpi}(b), the measures $M^{s_j; \mathrm{hs};\varpi_k}$ converge weakly to $M^{s_j;\infty}$ as $k\rightarrow\infty$ , which verifies condition (1) in \cite[Proposition 2.21]{dimitrov2024airy}. In addition, they satisfy condition (2) by Lemma \ref{Lem.InfinitelyManyAtoms} and condition (3) by Lemma \ref{Lem.TightFromAbove}. From \cite[Proposition 2.21]{dimitrov2024airy} we  conclude (\ref{Eq.AISTIGHT}).
\end{proof}

%
%
\subsection{Finite-dimensional convergence of $X^{\mathsf{S}}$}\label{Section3.4} The goal of this section is to prove the finite-dimensional version of Theorem \ref{Thm.ConvergenceToAiryLE}. The precise statement is given in the following lemma.

\begin{lemma}\label{Lem.FDConvT} Fix a finite set $\mathsf{S} = \{s_1, \dots, s_m\} \subset \mathbb{R}$ with $s_1 < s_2 < \cdots < s_m$ and set $T + \mathsf{S} := \{T+ s_1, T+s_2, \dots, T + s_m\}$. For $T > -s_1$ let $X^{T+ \mathsf{S}}$ be as in Lemma \ref{Lem.FDConvVarpi}. Then, $X^{T+ \mathsf{S}} \overset{f.d.}{\rightarrow} \left( \mathcal{A}_i(s_j): i \geq 1, j = 1, \dots, m \right)$ as $T \rightarrow \infty$, where $\mathcal{A}$ is the Airy line ensemble from Definition \ref{Def.AiryLE}.
\end{lemma}

\begin{proof} The proof is similar to the proof of Lemma \ref{Lem.FDConvVarpi} above. As before, it suffices to prove the result for each sequence $T_k> -s_1$, $k=1,2,\dots$ satisfying $ T_k\rightarrow\infty$ as $k\rightarrow\infty$. 

Suppose that for each $i \geq 1$ and $j = 1, \dots, m$, we know that 
\begin{equation}\label{Eq.XISTIGHT}
\{X_i^{j;T_k + \mathsf{S}}\}_{k \geq 1} \mbox{ is tight.}
\end{equation}

From Lemma \ref{Lem.FDConvVarpi}, we know that for $T_k > -s_1$ the measures
$$\tilde{M}^{\mathsf{S};k}(A) = \sum_{i \geq 1} \sum_{j = 1}^m{\bf 1}\{(s_j, X^{j,T_k+\mathsf{S}}_i) \in A\}$$
have the same law as $M^{T_k + \mathsf{S};\infty} \phi_{T_k}^{-1}$, where $\phi_{T}(s,x) = (s-T, x)$. From Lemma \ref{Lem.PointProcessConvergenceT}(a), we conclude that $\tilde{M}^{\mathsf{S};k} \Rightarrow  M^{\mathsf{S}; \mathcal{A}}$ as $k \rightarrow \infty$. The latter and (\ref{Eq.XISTIGHT}) verify the conditions of \cite[Proposition 2.19]{dimitrov2024airy}, which implies the statement of the lemma. Thus, we only need to establish (\ref{Eq.XISTIGHT}). \\

Fix $j \in \{1, \dots, m\}$. In order to show (\ref{Eq.XISTIGHT}), we seek to apply \cite[Proposition 2.21]{dimitrov2024airy} to the sequence $(X^{j,T_k+\mathsf{S}}_i:i\geq1)$, $k=1,2,\dots$ of random vectors in $\mathbb{R}^{\infty}$. If we show that all three conditions are satisfied, that proposition would imply (\ref{Eq.XISTIGHT}).

We define the random measures on $\mathbb{R}$
$$\tilde{M}^{s_j;k}(A) = \sum_{i \geq 1} {\bf 1}\{X^{j,T_k+\mathsf{S}}_i \in A\},$$
and note that $\tilde{M}^{s_j;k} \overset{d}{=} M^{T_k + s_j;\infty}$, where the latter is as in Lemma \ref{Lem.PointProcessConvergenceT}(b). One way to see the latter is to use that both measures are Pfaffian point processes on $\mathbb{R}$ with the same reference measure and correlation kernel (in view of Lemma \ref{Lem.PointProcessConvergenceVarpi}(b), Lemma \ref{Lem.FDConvVarpi} and \cite[Lemma 5.13]{DY25}), and so they have the same law by \cite[Proposition 5.8(3)]{DY25}.

From the above observation, and Lemma \ref{Lem.PointProcessConvergenceT}(b), we conclude $\tilde{M}^{s_j;k} \Rightarrow M^{s_j; \mathcal{A}}$, which verifies condition (1) in \cite[Proposition 2.21]{dimitrov2024airy}. Condition (2) follows from the well-known fact that the Airy point process has infinitely many atoms almost surely, see e.g. \cite[Equation (7.11)]{dimitrov2024airy}. 

By Lemma \ref{Lem.FDConvVarpi}, $ \hsa_1(T_k+s_j) \Rightarrow X_1^{j,T_k+\mathsf{S}} $  as $\varpi\rightarrow\infty$. By the Portmanteau Theorem, we conclude for each $a > 0$
$$\mathbb{P}\left(X_1^{j,T_k+\mathsf{S}} > a\right) \leq \liminf_{\varpi \rightarrow \infty} \mathbb{P}\left(\hsa_1(T_k+s_j) > a\right) \leq \sup_{t > 0} \sup_{\varpi \geq 2} \mathbb{P}\left(\hsa_1(t) \geq a\right).$$
Taking $\limsup_{a \rightarrow \infty}$ on both sides and applying Lemma \ref{Lem.TightFromAbove}, we conclude 
$$\limsup_{a \rightarrow \infty} \mathbb{P}\left(X_1^{j,T_k+\mathsf{S}} > a\right) = 0,$$
verifying the last condition (3) in \cite[Proposition 2.21]{dimitrov2024airy}. 
\end{proof}

%
%
\section{Gibbsian line ensembles}\label{Section4} In Proposition \ref{Prop.HSAiry} we mentioned that $\hsl$ satisfies the half-space Brownian Gibbs property, which provides a local description of this ensemble in terms of avoiding reverse Brownian motions with alternating drifts, see Definition \ref{Def.BGP} for a precise statement. The goal of this section is to leverage this property in order to establish three key statements. 

The first is Lemma \ref{Lem.MOC}, which provides estimates for the modulus of continuity of several avoiding reverse Brownian motions. This will be used to establish tightness in the proof of Theorem \ref{Thm.Convergence}, and improve the finite-dimensional convergence result from Lemma \ref{Lem.FDConvVarpi} to a functional limit statement.  

The second is Lemma \ref{Lem.NoLowMin}, which will allow us to transfer certain information from $\hsl(T)$ back to $\hsl(t)$ for each $t \in (0, T)$. The latter will feed (in a non-trivial way) into the proof of the ``infinitely many atoms'' lemma, Lemma \ref{Lem.InfinitelyManyAtoms}, by effectively allowing us to apply the same transfer mechanism from $\hsl(\infty)$ to $\hsl(t)$ for each $t \in (0, \infty)$. 

The third is Lemma \ref{Lem.BrownianEnsembleConv}, which essentially states that the half-space Brownian Gibbs property becomes the pinned half-space Brownian Gibbs property in the $\varpi \rightarrow \infty$ limit. This will be used in the proof of Theorem \ref{Thm.GibbsProperty}.

%
%
\subsection{Definitions and notation for line ensembles}\label{Section4.1}
In this section we recall some basic definitions and notation regarding line ensembles, mostly following \cite[Section 2]{DY25}.

 Given two integers $a \leq b$, we let $\llbracket a, b \rrbracket$ denote the set $\{a, a+1, \dots, b\}$. We also set $\llbracket a,b \rrbracket = \emptyset$ when $a > b$, $\llbracket a, \infty \rrbracket = \{a, a+1, a+2 , \dots \}$, $\llbracket - \infty, b\rrbracket = \{b, b-1, b-2, \dots\}$ and $\llbracket - \infty, \infty \rrbracket = \mathbb{Z}$. Given an interval $\Lambda \subseteq \mathbb{R}$, we endow it with the subspace topology of the usual topology on $\mathbb{R}$. We let $(C(\Lambda), \mathcal{C})$ denote the space of continuous functions $f: \Lambda \rightarrow \mathbb{R}$ with the topology of uniform convergence over compact sets, see \cite[Chapter 7, Section 46]{Munkres}, and Borel $\sigma$-algebra $\mathcal{C}$. Given a set $\Sigma \subseteq \mathbb{Z}$, we endow it with the discrete topology and denote by $\Sigma \times \Lambda$ the set of all pairs $(i,x)$ with $i \in \Sigma$ and $x \in \Lambda$ with the product topology. We also denote by $\left(C (\Sigma \times \Lambda), \mathcal{C}_{\Sigma}\right)$ the space of real-valued continuous functions on $\Sigma \times \Lambda$ with the topology of uniform convergence over compact sets and Borel $\sigma$-algebra $\mathcal{C}_{\Sigma}$. We typically take $\Sigma = \llbracket 1, N \rrbracket$ with $N \in \mathbb{N} \cup \{\infty\}$. The following defines the notion of a line ensemble.
\begin{definition}\label{Def.LineEnsembles}
Let $\Sigma \subseteq \mathbb{Z}$ and $\Lambda \subseteq \mathbb{R}$ be an interval. A {\em $\Sigma$-indexed line ensemble $\mathcal{L}$} is a random variable defined on a probability space $(\Omega, \mathcal{F}, \mathbb{P})$ that takes values in $\left(C (\Sigma \times \Lambda), \mathcal{C}_{\Sigma}\right)$. Intuitively, $\mathcal{L}$ is a collection of random continuous curves (sometimes referred to as {\em lines}), indexed by $\Sigma$, each of which maps $\Lambda$ into $\mathbb{R}$. We will often slightly abuse notation and write $\mathcal{L}: \Sigma \times \Lambda \rightarrow \mathbb{R}$, even though it is not $\mathcal{L}$ which is such a function, but $\mathcal{L}(\omega)$ for every $\omega \in \Omega$. For $i \in \Sigma$ we write $\mathcal{L}_i(\omega) = (\mathcal{L}(\omega))(i, \cdot)$ for the curve of index $i$ and note that the latter is a map $\mathcal{L}_i: \Omega \rightarrow C(\Lambda)$, which is $\mathcal{F}/\mathcal{C}$ measurable. If $a,b \in \Lambda$ satisfy $a \leq b$, we let $\mathcal{L}_i[a,b]$ denote the restriction of $\mathcal{L}_i$ to $[a,b]$. We also recall that the notions of {\em ordered} and {\em non-intersecting} line ensembles were introduced in Definition \ref{Def.Ordered}.
\end{definition}
\begin{remark}\label{Rem.Polish} As shown in \cite[Lemma 2.2]{DEA21}, we have that $C(\Sigma \times \Lambda)$ is a Polish space, and so a line ensemble $\mathcal{L}$ is just a random element in $C(\Sigma \times \Lambda)$ in the sense of \cite[Section 3]{Billing}.
\end{remark}

We let $W_t$ denote a standard one-dimensional Brownian motion, and if $y, \mu \in \mathbb{R}$, we define the {\em Brownian motion with drift $\mu$ from $B(0) = y$} by
\begin{equation}\label{eq:DefBrownianMotionDrift}
B(t) = y + W_t + \mu t.
\end{equation}
If $b \in (0,\infty)$ we also define the {\em reverse Brownian motion with drift $\mu$ from $\cev{B}(b) = y$} by
\begin{equation}\label{eq:RevDefBrownianMotionDrift}
\cev{B}(t) = B(b-t) \mbox{ for } 0 \leq t \leq b.
\end{equation}

For $k\in\mathbb{N}$ we denote by $\weyl_k$ and $\weylc_k$ the open and closed Weyl chambers in $\mathbb{R}^{k}$, i.e.
\begin{equation}\label{Eq.DefWeyl}
\weyl_k=\{\vec{x}\in\mathbb{R}^k: x_1>x_2>\cdots>x_k\}, \hspace{2mm} \weylc_k=\{\vec{x}\in\mathbb{R}^k: x_1\geq x_2 \geq \cdots \geq x_k\}.
\end{equation}
We next define the $g$-avoiding reverse Brownian line ensembles.
\begin{definition}\label{Def.AvoidBLE} 
Suppose $\vec{y},\vec{\mu} \in \mathbb{R}^k$, $b > 0$ and $g:[0,b]\rightarrow[-\infty,\infty)$ is a continuous function (i.e., either $g \in C([0,b])$  or $g= -\infty$ everywhere) such that $g(b) < y_k$. We denote the law of $k$ independent reverse Brownian motions $\{\cev{B}_i\}_{i = 1}^k$ on $[0,b]$, such that $\cev{B}_i$ has drift $\mu_i$ and $\cev{B}_i(b) = y_i$ by $\pfbm^{b, \vec{y}, \vec{\mu}}$, and write $\efbm^{b, \vec{y}, \vec{\mu}}$ for the expectation with respect to this measure. If $\vec{y} \in \weyl_k$, we also let $\pabm^{b, \vec{y}, \vec{\mu}, g}$ be the law $\pfbm^{b, \vec{y}, \vec{\mu}}$, conditioned on the event 
$$E_{\operatorname{avoid}}=\left\{\cev{B}_1(r)>\cev{B}_2(r)>\cdots>\cev{B}_k(r)>g(r)\mbox{ for all } r\in[0,b] \right\},$$
and write $\eabm^{b, \vec{y}, \vec{\mu}, g}$ for the expectation with respect to this measure. When $g=-\infty$ we omit the last superscript.
\end{definition}

\begin{remark}\label{Rem.WellDAvoid} As mentioned in \cite[Remark 2.6]{DY25}, we have that $\pfbm^{b, \vec{y}, \vec{\mu}}(E_{\operatorname{avoid}}) > 0$, and hence $\pabm^{b, \vec{y}, \vec{\mu}, g}$ is well-defined, see \cite[Definition 2.4]{DimMat} for a similar argument.
\end{remark}

We next define the non-intersecting Brownian line ensembles with pairwise pinning at the boundary. This property was introduced by Das and the second author in \cite{DasSerio25}; the following is essentially a restatement of the definition therein with notation closer to that in the present paper. We will first need the definition of a 3D Bessel bridge.

\begin{definition}\label{Def.Bessel}
Suppose $b,y>0$. Let $\{\vec{W}_t\}_{t\ge 0}$ be a standard three-dimensional Brownian motion. The \textit{3D Bessel bridge} on $[0,b]$ to $y$ is the stochastic process $\{V_t = \lVert \vec{W}_t\rVert_2 : t\in[0,b]\}$ conditioned on $V_b = y$. (For a formal definition of this conditioning, see \cite[Section XI.3]{RY}.) 
\end{definition}

\begin{remark}\label{Rem.Bessel}
The 3D Bessel bridge $\{V_t\}_{t\in[0,b]}$ to $y$ is uniquely characterized by its finite-dimensional distributions, which can be found in \cite[p. 464]{RY}, and we recall below. For $t>0$, $x\in\mathbb{R}$, define the standard heat kernel $\hk_t(x,y) = (2\pi t)^{-1/2}e^{-(x-y)^2/2t}$. Fix $0<t_1<t_2<\cdots<t_k < t_{k+1} = b$, $y_1,\dots,y_k>0$, and $y_{k+1}=y$. Then the joint density $f_{({t_1},\dots,{t_k})}$ of $(V_{t_1},\dots,V_{t_k})$ is given by
\begin{equation}\label{Eq.BesselDensity}
f_{({t_1},\dots,{t_k})}(y_1,\dots,y_k) = \frac{b}{t_1} \cdot \frac{y_1}{y} \cdot \frac{\hk_{t_1}(0,y_1)}{\hk_b(0,y)}\prod_{i=1}^k \left[\hk_{t_{i+1}-t_i}(y_i,y_{i+1}) - \hk_{t_{i+1}-t_i}(y_i, -y_{i+1})\right].
\end{equation}
\end{remark}

We next define the $g$-avoiding pinned reverse Brownian line ensembles.
\begin{definition}\label{Def.PinnedBM}
Suppose $\vec{y}\in \weyl_{2k}$, $b>0$, and $g:[0,b]\to[-\infty,\infty)$ is a continuous function with $g(b)<y_{2k}$. Let $\{U_i\}_{i=1}^k$ be independent reverse Brownian motions (with no drift) on $[0,b]$ from $U_i(b) = 2^{-1/2}(y_{2i-1}+y_{2i})$, and let $\{V_i\}_{i=1}^k$ be independent 3D Bessel bridges on $[0,b]$ to $V_i(b) = 2^{-1/2}(y_{2i-1}-y_{2i})$ as in Definition \ref{Def.Bessel}. For $1\leq i\leq k$, define $B_{2i-1} = 2^{-1/2}(U_i+V_i)$ and $B_{2i} = 2^{-1/2}(U_i-V_i)$. Then we let $\mathbb{P}_{\mathrm{pin}}^{b,\vec{y},g}$ denote the law of $\{B_i\}_{i=1}^{2k}$, conditioned on the event
\[
E_{\mathrm{avoid}}^{\mathrm{pin}} = \left\{B_{2i}(r) > B_{2i+1}(r) \mbox{ for all } i\in\llbracket 1,k\rrbracket, \, r\in(0,b)\right\},
\]
where $B_{2k+1} = g$. Expectation with respect to this law is denoted $\mathbb{E}_{\mathrm{pin}}^{b,\vec{y},g}$. When $g=-\infty$ we omit the last superscript.
\end{definition}

\begin{remark}\label{Rem.WellDPinAvoid}
Let us briefly explain why we can condition on $E_{\mathrm{avoid}}^{\mathrm{pin}}$, and hence why $\mathbb{P}_{\mathrm{pin}}^{b,\vec{y},g}$ is well-defined. Firstly, the set of functions satisfying the inequalities in $E_{\mathrm{avoid}}^{\mathrm{pin}}$ is an open set in $C(\llbracket 1, 2k \rrbracket \times [0,b])$, and hence $E_{\mathrm{avoid}}^{\mathrm{pin}}$ is measurable. Thus, it suffices to show that $\mathbb{P}\left(E_{\mathrm{avoid}}^{\mathrm{pin}}\right) > 0$.

When $k = 1$ and $g = -\infty$, the conditions in $E_{\mathrm{avoid}}^{\mathrm{pin}}$ are vacuously satisfied and so $\mathbb{P}(E_{\mathrm{avoid}}^{\mathrm{pin}}) = 1$. Thus, $\mathbb{P}_{\mathrm{pin}}^{b,\vec{y}, -\infty}$ is well-defined for $k = 1$. For general $k \geq 1$ and $g$, we can find $\varepsilon > 0$ and continuous functions $h^{\pm}_i$ for $1 \leq i \leq k$, such that:
\begin{itemize}
\item $\varepsilon < (1/4) \min_{i \in \llbracket 1, 2k \rrbracket} (y_i - y_{i+1})$, where $y_{2k+1} = g(b)$; 
\item $h_i^{-}(t) + \varepsilon < h^+_i(t) - \varepsilon $ for $i \in \llbracket 1, k \rrbracket$ and $t \in [0,b]$;
\item $h_i^-(b) + \varepsilon \leq y_{2i}$ and $h_{i}^+(b) - \varepsilon \geq y_{2i-1} $ for $i \in \llbracket 1, k \rrbracket$;
\item $h_i^-(t)  > h_{i+1}^+(t)$ for $i \in \llbracket 1, k \rrbracket$ and $t \in [0,b]$, where $h_{k+1}^+(t) = g(t)$.
\end{itemize}
In particular, we see
$$E_{\mathrm{avoid}}^{\mathrm{pin}} \supseteq \cap_{i = 1}^k \{h_i^-(t) \leq B_{2i}(t) \leq B_{2i-1}(t) \leq h_i^+(t)  \mbox{ for all } t\in [0,b] \},$$
and so by independence
$$\mathbb{P}\left(E_{\mathrm{avoid}}^{\mathrm{pin}}\right) \geq \prod_{i = 1}^k \mathbb{P}(h_i^-(t)  \leq B_{2i}(t) \leq B_{2i-1}(t) \leq h_i^+(t) \mbox{ for all }t\in [0,b]) > 0,$$
where in the last inequality we used Lemma \ref{Lem.StayInCorridor} with $f(t) = h_i^+(t) - \varepsilon$ and $g(t) = h_i^-(t) + \varepsilon$.
\end{remark}

We next introduce the two main definitions of the section -- the half-space Brownian Gibbs property from \cite[Definition 2.7]{DY25} and the pinned half-space Brownian Gibbs property.
\begin{definition}\label{Def.BGP}
Fix a set $\Sigma = \llbracket 1, N \rrbracket$ with $N \in \mathbb{N} \cup \{\infty\}$, an interval $\Lambda= [0,T]$ or $\Lambda = [0,T)$ with $T \in (0, \infty]$, and $\mu_i \in \mathbb{R}$ for $i \in \llbracket 1, N-1 \rrbracket$. For $k\in\llbracket 1,N-1\rrbracket$ and $b \in \Lambda \cap (0,\infty)$ we write $\vec{\mu}_k = (\mu_1, \dots, \mu_k)$, $D_{\llbracket 1,k\rrbracket,b} = \llbracket 1,k\rrbracket \times [0,b)$ and $D_{\llbracket 1,k\rrbracket,b}^c = (\Sigma \times \Lambda) \setminus D_{\llbracket 1,k\rrbracket,b}$. 
 
A $\Sigma$-indexed line ensemble $\mathcal{L}$ on $\Lambda$ satisfies the {\em half-space Brownian Gibbs property with parameters $\{\mu_i\}_{i \in \llbracket 1, N- 1 \rrbracket}$} if it is non-intersecting and for any $b \in \Lambda \cap (0,\infty)$, any $k\in\llbracket 1,N-1\rrbracket$, and any bounded Borel-measurable function $F: C(\llbracket 1,k\rrbracket \times [0,b]) \rightarrow \mathbb{R}$
\begin{equation}\label{Eq.HSBGP}
\mathbb{E} \left[ F\left(\mathcal{L}|_{\llbracket 1,k\rrbracket \times [0,b]} \right) \,\big|\, \mathcal{F}_{\operatorname{ext}} (\llbracket 1,k\rrbracket \times [0,b))  \right] =\eabm^{b, \vec{y}, \vec{\mu}_k,g} \bigl[ F( \mathcal{Q} ) \bigr],\quad \mathbb{P}\text{-almost surely,} 
\end{equation}
where $\mathcal{L}|_{\llbracket 1,k\rrbracket \times [0,b]} $ is the restriction of $\mathcal{L}$ to $\llbracket 1, k \rrbracket \times [0,b]$, $g = \mathcal{L}_{k + 1}[0,b]$, $\vec y=(\mathcal{L}_1(b), \dots, \mathcal{L}_k(b))$, and
\[\mathcal{F}_{\operatorname{ext}} (\llbracket 1,k\rrbracket \times [0,b)) := \sigma \left ( \mathcal{L}_i(s): (i,s) \in D_{\llbracket 1,k\rrbracket, b}^c \right).
\] 
On the right side of (\ref{Eq.HSBGP}), we have that $\mathcal{Q}$ has law $\pabm^{b, \vec{y}, \vec{\mu}_k,g}$.
\end{definition}

\begin{definition}\label{Def.PinnedBGP} Fix a set $\Sigma = \llbracket 1, N \rrbracket$ with $N \in \mathbb{N} \cup \{\infty\}$, and an interval $\Lambda= [0,T]$ or $\Lambda = [0,T)$ with $T \in (0, \infty]$. A $\Sigma$-indexed line ensemble $\mathcal{L}$ on $\Lambda$ satisfies the {\em pinned half-space Brownian Gibbs property} if its restriction to $\Lambda\cap(0,\infty)$ is non-intersecting and for any $b \in \Lambda \cap (0,\infty)$, any $k\in\llbracket 1,\lfloor(N-1)/2\rfloor\rrbracket$, and any bounded Borel-measurable function $F: C(\llbracket 1,2k\rrbracket \times [0,b]) \rightarrow \mathbb{R}$,
\begin{equation}\label{Eq.HSPinnedBGP}
\mathbb{E} \left[ F\left(\mathcal{L}|_{\llbracket 1,2k\rrbracket \times [0,b]} \right) \,\big|\, \mathcal{F}_{\operatorname{ext}} (\llbracket 1,2k\rrbracket \times [0,b))  \right] =\mathbb{E}_{\mathrm{pin}}^{b,\vec{y}, g} \bigl[ F( \mathcal{Q} ) \bigr],\quad \mathbb{P}\text{-almost surely,} 
\end{equation}
where $g = \mathcal{L}_{2k + 1}[0,b]$, $\vec y=(\mathcal{L}_1(b), \dots, \mathcal{L}_{2k}(b))\in\mathbb{R}^{2k}$ and $\mathcal{F}_{\operatorname{ext}}$ is as above.
\end{definition}

We end this section by recalling the Brownian Gibbs property from \cite[Definition 2.2]{CorHamA}, see also \cite[Definition 2.8]{DEA21}. To state it, we require the notion of an $(f,g)$-avoiding Brownian bridge ensemble from \cite[Definition 2.7]{DEA21}.
\begin{definition}\label{Def.fgAvoidingBE}
Fix $k \in \mathbb{N}$, $\vec{x}, \vec{y} \in \weyl_k$, $a,b \in \mathbb{R}$ with $a < b$, and two continuous functions $f: [a,b] \rightarrow (-\infty, \infty]$ and $g: [a,b] \rightarrow [-\infty,\infty)$. As before, the latter means that either $f \in C([a,b])$ or $f \equiv \infty$, and similarly for $g$. In addition, we assume that $f(t) > g(t)$ for $t \in[a,b]$, $f(a) > x_1$, $f(b) > y_1$, $g(a) < x_k$, $g(b) < y_k$. With this data we let $\mathbb{P}_{\mathrm{avoid};\mathrm{Br}}^{a,b,\vec{x},\vec{y}, f,g}$ denote the law of $k$ independent Brownian bridges $\{B_i:[a,b] \rightarrow \mathbb{R}\}_{i = 1}^k$ from $B_i(a) = x_i$ to $B_i(b) = y_i$, conditioned on the event 
$$E = \{f(t) > B_1(t) > B_2(t) > \cdots > B_k(t) > g(t) \mbox{ for all } t\in [a,b]\}.$$
We note that the law $\mathbb{P}_{\mathrm{avoid};\mathrm{Br}}^{a,b,\vec{x},\vec{y}, f,g}$ on $C(\llbracket 1, k \rrbracket \times [a,b])$ is well-defined; see \cite[Definition 2.7]{DEA21}. The expectation with respect to $\mathbb{P}_{\mathrm{avoid};\mathrm{Br}}^{a,b,\vec{x},\vec{y}, f,g}$ is denoted by $\mathbb{E}_{\mathrm{avoid};\mathrm{Br}}^{a,b,\vec{x},\vec{y}, f,g}$.
\end{definition}

\begin{definition}\label{Def.BGPVanilla}
An $\mathbb{N}$-indexed line ensemble $\mathcal{L} = \{\mathcal{L}_i\}_{i \geq 1}$ on an interval $\Lambda \subseteq \mathbb{R}$ is said to have the {\em Brownian Gibbs property}, if it is non-intersecting, and the following holds for all $[a,b] \subseteq \Lambda$ and $1 \leq k_1 \leq k_2$. If we set $K = \llbracket k_1, k_2\rrbracket$, then for any bounded Borel-measurable function $F: C(K \times [a,b]) \rightarrow \mathbb{R}$, we have $\mathbb{P}$-almost surely
\begin{equation}\label{BGPTower}
\mathbb{E} \left[ F\left(\mathcal{L}|_{K \times [a,b]} \right)  {\big \vert} \mathcal{F}_{\mathrm{ext}} (K \times (a,b))  \right] =\mathbb{E}_{\mathrm{avoid}; \mathrm{Br}}^{a,b, \vec{x}, \vec{y}, f, g} \bigl[ F(\tilde{\mathcal{Q}}) \bigr].
\end{equation}
On the left side of (\ref{BGPTower}), we have that
$$\mathcal{F}_{\mathrm{ext}} (K \times (a,b)) = \sigma \left \{ \mathcal{L}_i(s): (i,s) \in (\mathbb{N} \times \Lambda) \setminus (K \times (a,b)) \right\},$$
and $ \mathcal{L}|_{K \times [a,b]}$ is the restriction of $\mathcal{L}$ to the set $K \times [a,b]$. On the right side of (\ref{BGPTower}), we have $\vec{x} = (\mathcal{L}_{k_1}(a), \dots, \mathcal{L}_{k_2}(a))$, $\vec{y} = (\mathcal{L}_{k_1}(b), \dots, \mathcal{L}_{k_2}(b))$, $f = \mathcal{L}_{k_1 - 1}[a,b]$ with the convention that $f = \infty$ if $k_1 = 1$, and $g = \mathcal{L}_{k_2 +1}[a,b]$. In addition, $\mathcal{Q} = \{\mathcal{Q}_i\}_{i = 1}^{k_2 - k_1 + 1}$ has law $\mathbb{P}_{\mathrm{avoid};\mathrm{Br}}^{a,b, \vec{x}, \vec{y}, f, g}$ as in Definition \ref{Def.fgAvoidingBE}, and $\tilde{\mathcal{Q}} = \{\tilde{\mathcal{Q}}_{i}\}_{i = k_1}^{k_2}$ satisfies $\tilde{\mathcal{Q}}_i = \mathcal{Q}_{i - k_1 + 1}$.
\end{definition}

%
%
\subsection{Properties of half-space Brownian line ensembles}\label{Section4.2} The goal of this section is to establish several basic properties about ensembles with laws $\pabm^{b, \vec{y}, \vec{\mu}, g}$ as in Definition \ref{Def.AvoidBLE}, $\mathbb{P}_{\mathrm{pin}}^{b,\vec{y},g}$ as in Definition \ref{Def.PinnedBM}, and $\mathbb{P}_{\mathrm{avoid};\mathrm{Br}}^{a,b, \vec{x}, \vec{y}, f, g}$ as in Definition \ref{Def.fgAvoidingBE}. 

The following statement shows that, for a fixed drift vector $\vec{\mu}$, the ensembles $\pabm^{b, \vec{y}, \vec{\mu}, g}$ can be coupled monotonically with respect to their boundary data $\vec{y}$ and $g$.
\begin{lemma}\label{Lem.MonotoneCoupling} Fix $k \in \mathbb{N}$, $b > 0$, $\vec{\mu} \in \mathbb{R}^k$, and two continuous functions $g^{\mathrm{b}}, g^{\mathrm{t}}: [0, b] \rightarrow [-\infty, \infty)$ satisfying $g^{\mathrm{b}}(t) \leq g^{\mathrm{t}}(t)$ for $t \in [0,b]$. Let $\vec{y}\,^{\mathrm{b}},\vec{y}\,^{\mathrm{t}} \in \weyl_k$ satisfy $y_i^{\mathrm{b}}\leq y_i^{\mathrm{t}}$ for $i\in\llbracket1,k\rrbracket$, $g^{\mathrm{b}}(b) < y_k^{\mathrm{b}}$, and $g^{\mathrm{t}}(b)< y_k^{\mathrm{t}}$. Then there exists a probability space $(\Omega, \mathcal{F}, \mathbb{P})$, supporting two $\llbracket 1, k \rrbracket$-indexed line ensembles $\mathcal{L}^{\mathrm{b}}$ and $\mathcal{L}^{\mathrm{t}}$ on $[0,b]$, such that:
\begin{itemize}
\item the laws of $\mathcal{L}^{\mathrm{b}}$ and $\mathcal{L}^{\mathrm{t}}$ under $\mathbb{P}$ are given by $\pabm^{b, \vec{y}\,^{\mathrm{b}}, \vec{\mu}, g^{\mathrm{b}}}$ and $\pabm^{b, \vec{y}\,^{\mathrm{t}}, \vec{\mu}, g^{\mathrm{t}}}$, respectively;
\item $\mathbb{P}$-a.s. $\mathcal{L}_i^{\mathrm{b}}(t) \leq \mathcal{L}^{\mathrm{t}}_i(t)$ for $i \in \llbracket 1, k \rrbracket$, $t \in [0,b]$. 
\end{itemize}
\end{lemma}
\begin{proof} 
A discrete analogue of this monotone coupling lemma was established in \cite[Lemma 2.20]{DY25}. As we show below, the coupling in the present setup arises as a diffusive scaling limit, using \cite[Lemma 2.22]{DY25}, of the coupling in \cite[Lemma 2.20]{DY25}. We recall the notation $\mathbb{P}_{\mathsf{Inter},\mathrm{Geom}}^{B,\vec{Y},\vec{q},G}$ from \cite[Definition 2.12]{DY25}, for the law of the reverse $G$-interlacing geometric line ensemble on the interval $\llbracket 0,B\rrbracket$ with jump parameters $\vec{q}$ and exit data $\vec{Y}$.\\

\noindent\textbf{Step 1.} In this step, we construct the desired probability space $(\Omega,\mathcal{F},\mathbb{P})$ in the statement of the lemma, modulo a few claims. For $n\in\mathbb{N}$, let $B_n = \lceil bn\rceil$, and define functions $g_n^{\mathrm{b}},g_n^{\mathrm{t}} : [0,B_n/n] \to [-\infty,\infty)$ and $G_n^{\mathrm{b}},G_n^{\mathrm{t}} : [0,B_n]\to [-\infty,\infty)$ via
\begin{align*}
g_n^\star(t) &= \inf_{y\in[0,b]} \left\{ g^\star(y) + \sqrt{n/2}\,|t-y| \right\}, \quad G_n^\star(s) = \sqrt{2n}\cdot g_n^\star(s/n) + s, \qquad \star\in\{\mathrm{b},\mathrm{t}\}.
\end{align*}
Additionally, define $\vec{Y}^{n,\mathrm{b}},\vec{Y}^{n,\mathrm{t}} \in \mathbb{Z}^k$ and $\vec{q}^n\in(0,1)^k$ via 
\begin{align*}
    Y_i^{n,\star} = \lceil\sqrt{2n}\cdot y_i^\star + B_n\rceil, \quad q_i^n = \frac12 - \frac{1}{2\sqrt{2}}\mu_i n^{-1/2}, \qquad i\in\llbracket 1,k\rrbracket,\,\star\in\{\mathrm{b},\mathrm{t}\}.
\end{align*}
We claim for $\star\in\{\mathrm{b},\mathrm{t}\}$ and $g^{\star} \neq -\infty$ that:
\begin{enumerate}
\item[1.] $g_n^\star$ is Lipschitz continuous with constant $\sqrt{n/2}$;
\item[2.] $g_n^\star \to g^\star \mbox{ uniformly on }[0,b]$;
\item[3.] $\lim_{n\to\infty}|g_n^\star(B_n/n)-g^\star(b)| = 0$;
\item[4.] the functions $G_n^\star$ are increasing.
\end{enumerate}
We now construct $(\Omega,\mathcal{F},\mathbb{P})$ assuming these claims.\\

Since $G_n^{\mathrm{b}}, G_n^{\mathrm{t}}$ are increasing, $G_n^{\mathrm{b}} \leq G_n^{\mathrm{t}}$, $Y_i^{n,\mathrm{b}} \le Y_i^{n,\mathrm{t}}$ for $i\in\llbracket 1,k\rrbracket$, and $G_n^\star(B_n) \le Y_k^{n,\star}$ for $\star\in\{\mathrm{b},\mathrm{t}\}$ and all large $n$, we can apply \cite[Lemma 2.20]{DY25}. The latter ensures the existence of probability spaces $(\Omega_n,\mathcal{F}_n,\mathbb{P}_n)$ supporting discrete line ensembles $\mathfrak{Q}^{n,\mathrm{b}} = \{Q_i^{n,\mathrm{b}}\}_{i = 1}^k$, $\mathfrak{Q}^{n,\mathrm{t}} = \{Q_i^{n,\mathrm{t}}\}_{i = 1}^k$ that satisfy the following properties. Under $\mathbb{P}_n$ the ensemble $\mathfrak{Q}^{n,\star}$ is distributed according to $\mathbb{P}_{\mathsf{Inter},\mathrm{Geom}}^{B_n,\vec{Y}^{n,\star},\vec{q}^n,G_n^\star}$, and $\mathbb{P}_n$-a.s., $Q_i^{n,\mathrm{b}}(r) \le Q_i^{n,\mathrm{t}}(r)$ for $i\in\llbracket 1,k\rrbracket, r\in\llbracket 0,B_n\rrbracket$. Define the rescaled line ensembles $\mathcal{Q}^{n,\star}$ on $[0,b]$ via
\[
\mathcal{Q}_i^{n,\star}(t) = \frac{1}{\sqrt{2n}} (Q^{n,\star}_i(tn)-tn), \qquad i\in\llbracket 1,k\rrbracket, \, t\in[0,b].
\]
Then $\mathbb{P}_n$-a.s. we have $\mathcal{Q}_i^{n,\mathrm{b}}(t) \le \mathcal{Q}_i^{n,\mathrm{t}}(t)$ for all $i\in\llbracket 1,k\rrbracket$, $t\in[0,b]$. In addition, our construction above ensures that the conditions of \cite[Lemma 2.22]{DY25} are met with $p=1/2$, which proves that $\mathcal{Q}^{n,\star}$ converges in law to $\pabm^{b,\vec{y}^\star,\vec{\mu},g^\star}$ as $n\to\infty$.

Let $\Omega_0 = C(\llbracket 1,k\rrbracket\times[0,b])$, and let $\pi_n$ denote the joint law of $(\mathcal{Q}^{n,\mathrm{b}},\mathcal{Q}^{n,\mathrm{t}})$ under $\mathbb{P}_n$, and $\mu_n^\star$ the marginal law of $\mathcal{Q}^{n,\star}$. As $\mu_n^{\mathrm{b}}, \mu_n^{\mathrm{t}}$ are tight, so is $\pi_n$. Thus we can find a subsequence $\pi_{n_j}$ converging to some probability measure $\mathbb{P}$ on $\Omega := \Omega_0 \times \Omega_0$ with the product topology and corresponding $\sigma$-algebra $\mathcal{F} := \mathcal{C}_{\llbracket 1, k \rrbracket}\otimes \mathcal{C}_{\llbracket 1, k \rrbracket}$. Note the set $A = \{(\mathcal{L}^\mathrm{b},\mathcal{L}^\mathrm{t}) : \mathcal{L}_i^\mathrm{b}(t) \le \mathcal{L}_i^\mathrm{t}(t) \mbox{ for all } i\in\llbracket 1,k\rrbracket, t\in[0,b]\} \subset \Omega$ is closed, and $\pi_n(A) = 1$ for all $n$. By the Portmanteau theorem, $\mathbb{P}(A) \ge \limsup \pi_{n_j}(A) = 1$. The probability space $(\Omega,\mathcal{F},\mathbb{P})$ therefore satisfies the conditions in the statement of the lemma, once we set $(\mathcal{L}^{\mathrm{b}}, \mathcal{L}^{\mathrm{t}}) := (\omega_{\mathrm{b}}, \omega_{\mathrm{t}})$ for $(\omega_{\mathrm{b}}, \omega_{\mathrm{t}}) \in \Omega_0 \times \Omega_0 = \Omega$.\\

\noindent\textbf{Step 2.} In this step we prove the claims 1--4 from Step 1. Fix $s_1,s_2\in[0,B_n]$ with $s_1<s_2$. By the triangle inequality, for any $y\in[0,b]$ we have
\[
g^\star(y) + \sqrt{n/2}\cdot|y-s_2/n| \ge g^\star(y) + \sqrt{n/2}\cdot|y-s_1/n| - \sqrt{n/2}(s_2/n-s_1/n)
\]
Taking infima on both sides yields
\[
g_n^\star(s_2/n) \ge g_n^\star(s_1/n) - \sqrt{n/2}\cdot(s_2/n-s_1/n).
\]
A similar argument yields 
\[
g_n^\star(s_1/n) \ge g_n^\star(s_2/n) - \sqrt{n/2}\cdot(s_2/n-s_1/n),
\]
and in combination these imply claim 1.

To see that $g_n^\star \to g^\star$ uniformly, first note that by taking $y=t$ in the infimum, we have $g_n^\star(t) \le g^\star(t)$. On the other hand, let $w^\star(\delta) = \sup_{|x-y|<\delta} |g^\star(x)-g^\star(y)|$ be the modulus of continuity of $g^\star$. Since $g^\star$ is uniformly continuous, $w^\star(\delta)\to 0$ as $\delta\downarrow 0$. Given $\varepsilon>0$, choose $\delta>0$ so that $w^\star(\delta)<\varepsilon$. Let $\mathrm{osc}(g^\star) = \sup(g^\star) - \inf(g^\star) < \infty$, and fix $N_0$ large enough so that $\mathrm{osc}(g^\star)  \le \delta\sqrt{N_0/2}$. For $n \geq N_0$, we have
\begin{equation*}
    g^\star(y) + \sqrt{n/2}|t-y| \ge \begin{cases}
        g^\star(t) - \mathrm{osc}(g^\star) + \delta\sqrt{n/2} \geq g^\star(t), & |t-y|\ge \delta,\\
        g^\star(t) - w^\star(\delta) \geq g^\star(t) - \varepsilon, & |t-y|<\delta. \end{cases}
\end{equation*}
Thus, for $n \geq N_0$, we have $g^\star(t) \geq g_n^\star(t) \ge g^\star(t) - \varepsilon$ for all $t\in[0,b]$. This proves claim 2. 

From claims 1 and 2, and the fact that $|B_n/n - b| \leq 1/n$, we have 
$$|g_n^\star(B_n/n)-g^\star(b)| \le \sqrt{\frac{n}{2}}\cdot\frac{1}{n} + |g_n^\star(b)-g^\star(b)|  \to 0,$$
establishing claim 3.

Finally, from claim 1, we have for $s_1,s_2\in[0,B_n]$ with $s_1<s_2$, that
\begin{align*}
    G_n^\star(s_2)-G_n^\star(s_1) &= \sqrt{2n}\cdot(g_n^\star(s_2/n)-g_n^\star(s_1/n)) + (s_2-s_1)\\
    &\ge -\sqrt{2n}\cdot\sqrt{n/2}\cdot(s_2/n-s_1/n) + (s_2-s_1) = 0,
\end{align*}
proving claim 4.
\end{proof}

The next lemma states that the probability of a reverse Brownian motion touching but not crossing a fixed continuous curve is $0$. Its proof follows directly from that of \cite[Corollary 2.9]{CorHamA}; see also \cite[Lemma 2.3]{DY25}.
\begin{lemma}\label{Lem.NoTouchFiniteDrift} Fix $b > 0$, $\mu, y \in \mathbb{R}$, and let $\cev{B}$ be as in (\ref{eq:RevDefBrownianMotionDrift}). If $f,g \in C([0,b])$ satisfy $g(b) < y < f(b)$, then 
\begin{equation}\label{Eq.NoTouchFiniteDrift}
\begin{split}
&\mathbb{P}\left( \cev{B}(t) \geq g(t) \mbox{ for all $t \in [0,b]$ and } \min_{s \in [0,b]} |\cev{B}(s)- g(s)| = 0   \right) = 0, \\
&\mathbb{P}\left( \cev{B}(t) \leq f(t) \mbox{ for all $t \in [0,b]$ and } \min_{s \in [0,b]} |\cev{B}(s)- f(s)| = 0   \right) = 0.
\end{split}
\end{equation}
\end{lemma}

The following lemma gives the analogous no-touching result for a pair of pinned reverse Brownian motions.
\begin{lemma}\label{Lem.NoTouchPinned} Fix $b > 0$, $\vec{y} \in \weyl_2$, and let $f,g \in C([0,b])$ be such that $f(b) > y_1 > y_2 > g(b)$. Then 
\begin{align}
&\mathbb{P}_{\mathrm{pin}}^{b, \vec{y}}\left( \mathcal{Q}_2(t) \geq g(t) \mbox{ for all $t \in [0,b]$ and } \min_{s \in [0,b]} |\mathcal{Q}_2(s)- g(s)| = 0   \right) = 0, \label{Eq.NoTouchPinned} \\
&\mathbb{P}_{\mathrm{pin}}^{b, \vec{y}}\left( \mathcal{Q}_1(t) \leq f(t) \mbox{ for all $t \in [0,b]$ and } \min_{s \in [0,b]} |\mathcal{Q}_1(s)- f(s)| = 0   \right) = 0. \label{Eq.NoTouchPinned2}
\end{align}
\end{lemma}
\begin{proof}
We prove this as a consequence of Lemma \ref{Lem.NoTouchFiniteDrift}. We only give the proof of \eqref{Eq.NoTouchPinned} --- the proof of \eqref{Eq.NoTouchPinned2} is analogous. 

Recall from Definition \ref{Def.PinnedBM} that $\mathcal{Q}_2 = 2^{-1/2}(\cev{B} - V)$, where $\cev{B}$ is a reverse Brownian motion on $[0,b]$ from $2^{-1/2}(y_1+y_2)$ and $V$ is an independent 3D Bessel bridge on $[0,b]$ to $2^{-1/2}(y_1-y_2)$. In particular, $V \in C([0,b])$ a.s. The event in \eqref{Eq.NoTouchPinned} can be written as 
$$E := \left\{\cev{B}(t) \ge h(t) \mbox{ for all }t\in[0,b] \mbox{ and } \min_{s\in[0,b]} |\cev{B}(s) - h(s)| =0 \right\},$$
where $h(t) := 2^{1/2}g(t) + V(t)$. Conditioned on $V$, $h \in C([0,b])$ a.s., and by assumption $h(b) = 2^{1/2}g(b) + V(b) < 2^{1/2}y_2 + 2^{-1/2}(y_1 - y_2) = 2^{-1/2}(y_1 + y_2)$. It follows from Lemma \ref{Lem.NoTouchFiniteDrift} (with $\mu=0$, $y = 2^{-1/2}(y_1+y_2)$, and $g=h$), that $\mathbb{P}(E \mid V) = 0$ a.s. Integrating over $V$ implies \eqref{Eq.NoTouchPinned}.
\end{proof}

The next lemma shows that a pair of pinned reverse Brownian motions can be confined to a tubular neighborhood of arbitrary fixed continuous functions with positive probability.
\begin{lemma}\label{Lem.StayInCorridor} Fix $b > 0$,  $\vec{y} \in \weyl_2$, and $f,g \in C([0,b])$, such that $f(b) \ge y_1 > y_2 \ge g(b)$, and $f(t) > g(t)$ for $t\in[0,b]$. For any $\varepsilon > 0$, there exists $\delta > 0$, depending on $\varepsilon, b, y_1, y_2 ,f,g$, such that
\begin{align}\label{Eq.StayInCorridor}
\mathbb{P}_{\mathrm{pin}}^{b,\vec{y}}\left(g(t) - \varepsilon \leq \mathcal{Q}_2(t) \leq \mathcal{Q}_1(t) \leq f(t) + \varepsilon \mbox{ for all }t \in [0,b] \right) &\geq \delta.
\end{align}
Now fix $M>0$ and assume in addition that $f(t) = y_1+m(b-t)$ and $g(t) = y_2+m(b-t)$ for some $m\in\mathbb{R}$ with $|m|\le M$. Then the lower bound $\delta$ in \eqref{Eq.StayInCorridor} may be taken to depend only on $\varepsilon,b,M$.
\end{lemma}
\begin{proof}
By Definition \ref{Def.PinnedBM}, $\mathcal{Q}_1 = 2^{-1/2}(U+V)$ and $\mathcal{Q}_2=2^{-1/2}(U-V)$, where $U$ is a reverse Brownian motion on $[0,b]$ from $U(b) = 2^{-1/2}(y_1+y_2)$ and $V$ is an independent 3D Bessel bridge on $[0,b]$ to $V(b) = 2^{-1/2}(y_1-y_2)$. 

From our assumptions that $f(b) \ge y_1 > y_2 \ge g(b)$, and $f(t) > g(t)$ for $t\in[0,b]$, we can find $\tilde{f}, \tilde{g} \in C([0,b])$, such that $\tilde{f}(b) = y_1$, $\tilde{g}(b) = y_2$, and $f(t) \geq \tilde{f}(t) > \tilde{g}(t) \geq g(t)$ for $t \in [0,b]$. Combining the latter inequalities with the triangle inequality and the independence of $U$ and $V$, we see that the probability in \eqref{Eq.StayInCorridor} is bounded below by
\begin{equation}\label{Eq.UVcorr}
\begin{split}
&\mathbb{P}\left(\left|U(t) - 2^{-1/2}[\tilde{f}(t) + \tilde{g}(t)] \right| \le 2^{-1/2}\varepsilon \mbox{ for all } t\in[0,b]\right) \\
& \times \mathbb{P}\left(V(t) \le 2^{-1/2}[\tilde{f}(t)-\tilde{g}(t)] + 2^{-1/2}\varepsilon \mbox{ for all } t\in[0,b]\right).
\end{split}
\end{equation}
By Lemmas 7.8 and 7.9 in \cite{DasSerio25} (the notation $\mathbf{P}_{\infty,\infty}$ therein refers to the law of $V$), the two probabilities in (\ref{Eq.UVcorr}) are both strictly positive, implying \eqref{Eq.StayInCorridor}.\\

Now assume $f(t) = \tilde{f}(t) = y_1 + m(b-t)$, $g(t) = \tilde{g}(t) = y_2+m(b-t)$ with $|m|\le M$. The second factor in \eqref{Eq.UVcorr} is then equal to 
\[ \mathbb{P}\left(V(t) \le 2^{-1/2}(y_1-y_2) + 2^{-1/2}\varepsilon \mbox{ for all }t\in[0,b]\right), \]
which is bounded below by a positive constant $\delta_1$ depending only on $\varepsilon,b$ by \cite[Lemma 7.8]{DasSerio25}. The first factor in \eqref{Eq.UVcorr} is equal to
\begin{equation}\label{Eq.Ue}
\mathbb{P}\left(\left|U(t)-2^{-1/2}(y_1+y_2)-2^{1/2}m(b-t) \right| \le 2^{-1/2}\varepsilon \mbox{ for all } t\in[0,b] \right).
\end{equation}
We observe that \eqref{Eq.Ue} is a continuous function of $m$. Indeed, $U(t)-2^{-1/2}(y_1+y_2)-2^{1/2}m(b-t)$ is a reverse Brownian motion on $[0,b]$ from $0$ with drift $-2^{1/2}m$; denote its law by $\mathbb{P}_{-2^{1/2}m}$. For $m_0$ fixed, as $m\to m_0$, $\mathbb{P}_{-2^{1/2}m}\to\mathbb{P}_{-2^{1/2}m_0}$ weakly, and $\{x\in C[0,b] : \sup_{t\in[0,b]}|x(t)|\le 2^{-1/2}\varepsilon\}$ is a continuity set for $\mathbb{P}_{-2^{1/2}m_0}$, so the Portmanteau theorem implies the desired continuity. Therefore \eqref{Eq.Ue} attains a minimum over $m\in[-M,M]$, i.e., is bounded below by some $\delta_2 = \delta_2(\varepsilon,b,M)>0$. Taking $\delta = \delta(\varepsilon,b,M) := \delta_1\delta_2$ as the lower bound in \eqref{Eq.UVcorr} verifies the second claim in the lemma.
\end{proof}

We end this section with a lemma, which shows that a sequence of $(f,g)$-avoiding Brownian bridge ensembles converges weakly if their boundary data converge. 
\begin{lemma}\label{Lem.BridgeEnsemblesCty} Fix $k \in \mathbb{N}$, $\vec{x}, \vec{y} \in \weylc_k$, $a,b \in \mathbb{R}$ with $a < b$, and two continuous functions $f: [a,b] \rightarrow (-\infty, \infty]$ and $g: [a,b] \rightarrow [-\infty,\infty)$. In addition, we assume that $f(t) \geq g(t)$ for $t \in[a,b]$, $f(a) \geq x_1$, $f(b) \geq y_1$, $g(a) \leq x_k$, $g(b) \leq y_k$.

Suppose that $\vec{x}\,^n, \vec{y}\,^n \in \weyl_k$, $f^n: [a,b] \rightarrow (-\infty, \infty]$ and $g^n: [a,b] \rightarrow [-\infty,\infty)$ are continuous functions, such that:
\begin{enumerate}
\item $\lim_{n \rightarrow \infty} \vec{x}\,^n = \vec{x}$, $\lim_{n \rightarrow \infty} \vec{y}\,^n = \vec{y}$;
\item $\lim_{n \rightarrow \infty} f^n = f$, $ \lim_{n \rightarrow \infty} g^n = g$ in the sense that for some $\epsilon_n \rightarrow 0+$, and all $t \in [a,b]$
$$f^n(t) - \epsilon_n \leq f(t) \leq f^n(t) + \epsilon_n, \mbox{ and }g^n(t) - \epsilon_n \leq g(t) \leq g^n(t) + \epsilon_n;$$
\item $f^n(t) > g^n(t)$ for $t \in[a,b]$, $f^n(a) > x^n_1$, $f^n(b) > y^n_1$, $g^n(a) < x^n_k$, $g^n(b) < y^n_k$.
\end{enumerate} 
Then, as $n \rightarrow \infty$, the measures $\mathbb{P}_{\mathrm{avoid};\mathrm{Br}}^{a,b,\vec{x}\,^n,\vec{y}\,^n, f^n,g^n}$ converge weakly to a probability measure on $C(\llbracket 1, k \rrbracket \times [a,b])$, which we denote by $\mathbb{P}_{\mathrm{avoid};\mathrm{Br}}^{a,b,\vec{x},\vec{y}, f,g}$.
\end{lemma}
\begin{proof} As we explain below, the statement is a quick corollary of \cite[Lemmas 4.8 and 21.1]{AH23}.

From conditions (1) and (2), we can find $\delta_n \downarrow 0$, such that for all $n \geq 1$
$$u_j^n:= x_j - j \delta_n \leq x_j^n \leq x_j + \delta_n, \hspace{2mm} v_j^n:= y_j - j \delta_n \leq y_j^n \leq y_j + \delta_n \mbox{ for } j \in \llbracket 1, k \rrbracket,$$
and for all $t \in [a,b]$,
$$\tilde{g}^n(t):= g(t) - (k+1) \delta_n \leq g^n(t) \leq g(t) + \delta_n, \hspace{2mm} f(t) - \delta_n \leq f^n(t) \leq f(t) + \delta_n.$$

From \cite[Lemma 21.1]{AH23}, we know that $\mathbb{P}_{\mathrm{avoid};\mathrm{Br}}^{a,b,\vec{u}^n,\vec{v}^n, f,\tilde{g}^n}$ converges weakly as $n \rightarrow \infty$. On the other hand, by \cite[Lemma 4.8(1)]{AH23}, we can couple $\{B^n_i\}_{i = 1}^k$ with law $\mathbb{P}_{\mathrm{avoid};\mathrm{Br}}^{a,b,\vec{x}\,^n,\vec{y}\,^n, f^n,g^n}$ and $\{\tilde{B}^n_i\}_{i = 1}^k$ with law $\mathbb{P}_{\mathrm{avoid};\mathrm{Br}}^{a,b,\vec{u}^n,\vec{v}^n, f,\tilde{g}^n}$ on the same probability space so that almost surely
$$|B^n_i(t) - \tilde{B}_i^n(t)| \leq (k+1) \delta_n \mbox{ for all } (i,t) \in \llbracket 1, k \rrbracket \times [a,b].$$
By the convergence together theorem, see \cite[Theorem 3.1]{Billing}, the statement of the lemma follows.
\end{proof}

%
%
\subsection{Three key lemmas}\label{Section4.3} In this section we establish the three key lemmas we discussed in the beginning of Section \ref{Section4}.

Our first result shows that a line ensemble of avoiding reverse Brownian motions with alternating drifts converges to a pinned reverse Brownian line ensemble as the drift parameter tends to infinity.
\begin{lemma}\label{Lem.BrownianEnsembleConv} Fix $k \in \mathbb{N}$, $b > 0$, $\vec{y} \in \weyl_{2k}$ and a continuous function $g: [0,b]\rightarrow [-\infty, \infty)$ with $y_{2k} > g(b)$. Suppose that we are given sequences $\varpi_n \in \mathbb{R}$, $\vec{y}\,^{n} \in \weyl_{2k}$, and continuous functions $g_n: [0,b]\rightarrow [-\infty, \infty)$, such that 
$$\varpi_n \rightarrow \infty, \hspace{2mm} \vec{y}\,^{n} \rightarrow \vec{y}, \hspace{2mm} g_n \rightarrow g \mbox{ as $n \rightarrow \infty$ and } y^{n}_{2k} > g_n(b) \mbox{ for all $n \geq 1$}.$$
The convergence $g_n \rightarrow g$ is uniform on $[0,b]$, and when $g = -\infty$ this means that $g_n = -\infty$ for all large $n$. Finally, suppose that $\mathcal{L}^n$ has law $\pabm^{b, \vec{y}\,^n, \vec{\mu}\,^n, g_n}$ with $\mu_i^n = (-1)^{i} \varpi_n$ for $i \in \llbracket 1, 2k \rrbracket$ and $\mathcal{L}$ has law $\mathbb{P}_{\mathrm{pin}}^{b,\vec{y},g}$. Then $\mathcal{L}^n \Rightarrow \mathcal{L}$ as $n \rightarrow \infty$. 
\end{lemma}
\begin{proof} 

We will split the proof into four steps. In Steps 1 through 3, we establish the statement for the special case of two curves with no floor, and in Step 4 we extend this to the general setting.\\

\noindent\textbf{Step 1.} In this step we prove the lemma when $k=1$ and $g_n=-\infty$. Define the processes
\[
U^n = 2^{-1/2}(\mathcal{L}_1^n + \mathcal{L}_2^n), \qquad V^n = 2^{-1/2}(\mathcal{L}_1^n - \mathcal{L}_2^n).
\]
By Definition \ref{Def.AvoidBLE}, $(\mathcal{L}_1^n,\mathcal{L}_2^n)$ has the law of $\mathbb{P}_{\mathrm{free}}^{b,\vec{y},\vec{\mu}}$ (independent reverse Brownian motions), conditioned on the event $E_{\mathrm{avoid}} = \{V^n(r) > 0 \mbox{ for all }r\in[0,b]\}$. As the sum and difference of two independent standard Brownian motions are themselves independent Brownian motions with diffusion coefficient 2, we see that when $(\mathcal{L}_1^n,\mathcal{L}_2^n)$ has law $\mathbb{P}_{\mathrm{free}}^{b,\vec{y},\vec{\mu}}$, the processes $U^n$ and $V^n$ are distributed as independent reverse Brownian motions on $[0,b]$ from $w^n := 2^{-1/2}(y_1^n+y_2^n)$ and $z^n := 2^{-1/2}(y_1^n-y_2^n)$, respectively, and with drifts $0$ and $-2^{1/2}\varpi_n$, respectively. As the event $E_{\mathrm{avoid}}$ depends only on $V^n$, it is independent of $U^n$, and so we conclude that when $(\mathcal{L}_1^n,\mathcal{L}_2^n)$ has law $\mathbb{P}_{\mathrm{avoid}}^{b,\vec{y},\vec{\mu}}$, $U^n$ is still a reverse Brownian motion on $[0,b]$ from $w^n$ with drift $0$, while $V^n$ is an independent reverse Brownian motion from $z^n$ with drift $-2^{1/2}\varpi_n$, conditioned on $E_{\mathrm{avoid}}$.

As $n\to\infty$, we see that $U^n$ converges weakly to a reverse Brownian motion $U$ on $[0,b]$ from $w := 2^{-1/2}(y_1+y_2)$ with drift $0$. We claim that $V^n$ converges weakly to a 3D Bessel bridge $V$ on $[0,b]$ to $z := 2^{-1/2}(y_1-y_2)$, as in Definition \ref{Def.Bessel}. We will prove this claim in Steps 2 and 3. Assuming it, we see that $(\mathcal{L}_1^n, \mathcal{L}_2^n) \to \left(2^{-1/2}[U+V],2^{-1/2}[U-V]\right)$ as $n\to\infty$. By Definition \ref{Def.PinnedBM}, this is precisely the law $\mathbb{P}_{\mathrm{pin}}^{b,\vec{y}}$.\\

\noindent\textbf{Step 2.} In this step, we prove that the process $V^n$ from the previous step converges in the finite-dimensional sense to a 3D Bessel bridge $\{V_t\}_{t\in[0,b]}$ to $z$ as $n\to\infty$. Let us first show that $V^n(0)\Rightarrow 0$.

From \cite[Appendix 1.17]{borodin2015handbook}, we have that $V^n(0)$ has density 
$$f_n(y) \propto {\bf 1}\{y > 0\} \cdot \frac{1}{\sqrt{2\pi b}} \cdot e^{2^{1/2}\varpi_n(z^n-y) - b \varpi_n^2 } \cdot \left(\exp\left(- \frac{(z^n-y)^2}{2b}\right) - \exp\left(- \frac{(z^n+y)^2}{2b}\right) \right) .$$
Integrating the above, we get for $y > 0$ 
$$\mathbb{P}(V^n(0) > y) = \frac{\left[1 - \Phi\left(\frac{y - z^n + 2^{1/2}b \varpi_n}{\sqrt{b}} \right) \right] - e^{2^{3/2}\varpi_n z^n} \cdot \left[ 1- \Phi\left(\frac{y + z^n + 2^{1/2}b \varpi_n}{\sqrt{b}} \right) \right]}{\left[1 - \Phi\left(\frac{ - z^n + 2^{1/2}b \varpi_n}{\sqrt{b}} \right) \right] - e^{2^{3/2}\varpi_n z^n} \cdot \left[ 1- \Phi\left(\frac{ z^n + 2^{1/2}b \varpi_n}{\sqrt{b}} \right) \right]},$$
where $\Phi(x) = \int_{-\infty}^x \varphi(y)dy$ with $\varphi(y) = \frac{e^{-y^2/2}}{\sqrt{2\pi}}$. From \cite[7.1.23]{Stegun64} we have as $x \rightarrow \infty$ that $1- \Phi(x) = x^{-1} \varphi(x) + O(x^{-3} \varphi(x))$, and so if we set $\Delta_n = z^n + 2^{1/2}b \varpi_n$, we get as $n \rightarrow \infty$
$$\mathbb{P}(V^n(0) > y)  \sim \frac{\frac{1}{\sqrt{2b}\varpi_n} \cdot e^{2^{3/2}\varpi_n z^n-(\Delta_n+y)^2/2b} \cdot\left(e^{2z^ny/b} - 1 \right)}{\frac{z^n}{b^{3/2} \varpi_n^2}\cdot e^{2^{3/2}\varpi_n z^n - \Delta_n^2/2b } } \rightarrow 0.$$
As $\mathbb{P}(V^n(0) > y) \rightarrow 0$ for all $y > 0$, and $V^n(0) \geq 0$, we conclude $V^n(0) \Rightarrow 0$. \\

Now let us establish the finite-dimensional convergence of $V^n$ to the 3D Bessel bridge $V$ on the whole interval $[0,b]$. The transition densities of $V^n$ are written down explicitly in \cite[Theorem 3.1]{IO19}. For $0<t_1<\cdots<t_k<t_{k+1}=b$, $y_1,\dots,y_k>0$, and $y_{k+1}=z$, the joint density $f^n_{(t_1,\dots,t_k)}$ of $(V^n(t_1),\dots,V^n(t_k))$ is given by
\begin{align}
f^n_{(t_1,\dots,t_k)}(y_1,\dots,y_k) &= \prod_{i=1}^k \left[ \hk_{t_{i+1}-t_{i}}(y_i-y_{i+1}) - \hk_{t_{i+1}-t_{i}}(y_i+ y_{i+1}) \right] \cdot \frac{I_n(t_1,y_1)}{I_n(b,z)},\label{Eq.VJointDensity} \\
\mbox{where} \quad I_n(t,y) &= \int_0^\infty e^{-2^{1/2}\varpi_n u}(\hk_{t}(u-y)-\hk_{t}(u+y))\,du,\nonumber
\end{align}
and we have set $\hk_t(z) = (2\pi t)^{-1/2}e^{-z^2/2t}$. We change variables in the integrals to $u=\tfrac{v}{2^{1/2}\varpi_n}$ and introduce a constant factor of $2\varpi_n^2$ in both numerator and denominator to rewrite the ratio as
\[
\frac{I_n(t_1,y_1)}{I_n(b,z)} = \frac{J_n(t_1,y_1)}{J_n(b,z)}, \quad \mbox{where } J_n(t,y) = \int_0^\infty 2^{1/2}\varpi_n e^{-v}\left(\hk_{t}\left(\tfrac{v}{2^{1/2}\varpi_n}-y\right) - \hk_t\left(\tfrac{v}{2^{1/2}\varpi_n}+y\right) \right)\,dv.
\]

Set $\alpha = \frac{1}{2^{1/2}\varpi_n}$, so that $\alpha\to 0$ as $n\to\infty$. Then by L'H\^opital's rule, the pointwise limit of the integrand in $J_n(t,y)$ as $n\to\infty$ is equal to the limit as $\alpha\to 0$ of
\begin{align*}
    e^{-v}\frac{\hk_t(\alpha v - y) - \hk_t(\alpha v+y)}{\alpha} \to \frac{2y}{t} ve^{-v}\hk_t(y).
\end{align*}
Furthermore, using the inequality $|e^{-A}-e^{-B}|\le |A-B|$ for $A, B \geq 0$ (by the mean value theorem), the integrand of $J_n(t,y)$ is bounded for all $n$ by
\begin{align*}
    \frac{2}{\sqrt{2\pi t}} e^{-v}\varpi_n\left|\left(\frac{v}{2^{1/2}\varpi_n}-y\right)^2 - \left(\frac{v}{2^{1/2}\varpi_n}+y\right)^2\right| = \frac{4}{\sqrt{2\pi t}}\,yve^{-v},
\end{align*}
which is integrable in $v$. Therefore by dominated convergence, $J_n(t,y) \to \frac{2y}{t}\hk_t(y) \int_0^\infty ve^{-v}\,dv$. Returning to \eqref{Eq.VJointDensity}, and recalling the formula for the joint density $f_{(t_1,\dots,t_k)}$ of $(V(t_1),\dots,V(t_k))$ from \eqref{Eq.BesselDensity}, we find that
\[
f^n_{(t_1,\dots,t_k)}(y_1,\dots,y_k) \to f_{(t_1,\dots,t_k)}(y_1,\dots,y_k).
\]
Scheff\'e's lemma implies that $(V^n(t_1),\dots,V^n(t_k)) \to (V(t_1),\dots,V(t_k))$ in law, as desired.\\

\noindent\textbf{Step 3.} We claim that for each $\rho>0$,
\begin{equation}\label{Eq.Vmoc}
    \lim_{\delta \rightarrow 0+}\limsup_{n \rightarrow \infty}\mathbb{P}(w(V^n,\delta) \ge \rho) = 0,
\end{equation}
where for $f \in C([0,b])$, we define the usual modulus of continuity
\begin{equation}\label{Eq.MOCDef}
w(f,\delta) = \sup_{\substack{x,y \in [0,b] \\ |x-y| \leq \delta}} |f(x) - f(y)|.
\end{equation}
From \cite[Theorem 7.5]{Billing}, we see that (\ref{Eq.Vmoc}) upgrades the finite-dimensional convergence from the previous step to the topology of uniform convergence on compact sets. In the remainder of this step, we establish (\ref{Eq.Vmoc}).\\

Fix $\varepsilon > 0$. Let $B \in C([0,b])$ be a random continuous function with law $\mathbb{P}_{\mathrm{avoid};\mathrm{Br}}^{0,b,0,z, \infty,0}$ as in Lemma \ref{Lem.BridgeEnsemblesCty}. By the continuity of $B$, we can find $\delta > 0$, such that 
\begin{equation}\label{Eq.MOCLimitB}
\mathbb{P}(w(B,\delta) \geq \rho) \leq \varepsilon.
\end{equation}

Since $V^{n} \Rightarrow 0$ from Step 2, we have by Skorohod's representation theorem \cite[Theorem 6.7]{Billing}, that we may assume that $V^{n}(0)$ are defined on the same probability space and $V^{n}(0) \rightarrow 0$ a.s. From Lemma \ref{Lem.BridgeEnsemblesCty} and the Portmanteau theorem, we have 
\begin{equation}\label{Eq.Step2Limsup2}
\limsup_{n \rightarrow \infty} \mathbb{P}_{\mathrm{avoid};\mathrm{Br}}^{0,b, V^{n}(0), z^{n}, \infty, 0} (w(B, \delta) \geq \rho) \leq \mathbb{P}_{\mathrm{avoid};\mathrm{Br}}^{0,b, 0, z, \infty, 0} (w(B, \delta) \geq \rho) \leq \varepsilon,
\end{equation}
where in the last inequality we used (\ref{Eq.MOCLimitB}). 

Since $V^n$ is a reverse Brownian motion from $z^n$ with drift $-2^{1/2}\varpi_n$, conditioned on being positive, we know that the conditional law of $V^n$, given $V^n(0)$, is that of a Brownian bridge from $V^n(0)$ to $z^n$, conditioned on being positive. In other words, it is given by $\mathbb{P}_{\mathrm{avoid};\mathrm{Br}}^{0,b, V^n(0), z^n, \infty, 0}$ as in Section \ref{Section4.2}. We conclude 
\begin{equation*}
\begin{split}
&\limsup_{n \rightarrow \infty}  \mathbb{P}(w(V^n,\delta) \ge \rho) = \limsup_{n \rightarrow \infty} \mathbb{E} \left[ \mathbb{P}_{\mathrm{avoid};\mathrm{Br}}^{0,b, V^{n}(0), z^{n}, \infty, 0} (w(B, \delta) \geq \rho) \right] \\
&\leq \mathbb{E} \left[ \limsup_{n \rightarrow \infty} \mathbb{P}_{\mathrm{avoid};\mathrm{Br}}^{0,b, V^{n}(0), z^{n}, \infty, 0} (w(B, \delta) \geq \rho) \right]  \leq \varepsilon,
\end{split}
\end{equation*}
where in the first equality we used the tower property of conditional expectation. The inequality on the second line follows by the reverse Fatou's lemma, see \cite[page 10]{cairoli2011}, and the last inequality follows from (\ref{Eq.Step2Limsup2}). As $\varepsilon > 0$ was arbitrary, we see that the last inequality implies (\ref{Eq.Vmoc}).\\

\noindent\textbf{Step 4.} In this step we complete the proof of the lemma for all $k$ and $g_n$. We seek to show that for any bounded continuous function $F : C(\llbracket 1,2k\rrbracket\times[0,b])\to\mathbb{R}$
\begin{equation}\label{Eq.WeakConvToPinned}
\lim_{n\to\infty} \eabm^{b,\vec{y}\,^n,\vec{\mu}^n,g_n}[F(\mathcal{L}^n)] =  \mathbb{E}_{\mathrm{pin}}^{b,\vec{y},g}[F(\mathcal{L})].
\end{equation}

For $i\in\llbracket 1,k\rrbracket$, let $(B_{2i-1}^n,B_{2i}^n)$ be independent pairs with laws $\pabm^{b,\vec{y}^{n,i},\vec{\mu}}$, where $\vec{y}\,^{n,i} = (y_{2i-1}^n,y^n_{2i})$. Define the set $S_{\mathrm{avoid}} = \{ f_{2i}(r)>f_{2i+1}(r) \mbox{ for all }i\in\llbracket 1,k\rrbracket,\,r\in(0,b)\} \subset C(\llbracket 1,2k+1\rrbracket\times[0,b])$. Then by Definition \ref{Def.PinnedBM}, we may write
\begin{equation}\label{eq.Savoidfrac}
    \eabm^{b,\vec{y}^n,\vec{\mu}^n,g_n}[F(\mathcal{L}^n)] = \frac{\mathbb{E}[F(B^n)\mathbf{1}_{S_{\mathrm{avoid}}}(B^n,g^n)]}{\mathbb{E}[\mathbf{1}_{S_{\mathrm{avoid}}}(B^n,g^n)]},
\end{equation}
where we have written $(B^n,g^n)$ as shorthand for $(B_1^n,\dots,B_{2k}^n,g^n)$.

By Step 1, each independent pair $(B_{2i-1}^n,B_{2i}^n)$ converges weakly as $n\to\infty$ to $(B_{2i-1},B_{2i})$ with law $\mathbb{P}_{\mathrm{pin}}^{b,\vec{y}\,^i}$, where $\vec{y}\,^i = (y_{2i-1}, y_{2i})$. By Skorohod's Representation Theorem, \cite[Theorem 6.7]{Billing}, we may assume that $B^n = (B_1^n,\dots,B^n_{2k})$ and $B = (B_1,\dots,B_{2k})$ are defined on the same probability space and the convergence $B^n \rightarrow B$ (in $C(\llbracket 1, 2k \rrbracket \times [0,b])$) is almost sure.

Using that $B^n \rightarrow B$ and $g^n \rightarrow g$, we have
$$\mathbf{1}_{S_{\mathrm{avoid}}}(B^n,g^n) \rightarrow 1 \mbox{ a.s. on the event $E_1 = \{(B,g) \in S_{\mathrm{avoid}} \}$},$$
$$\mathbf{1}_{S_{\mathrm{avoid}}}(B^n,g^n) \rightarrow 0 \mbox{ a.s. on the event $E_2= \left\{(B,g) \in \overline{S_{\mathrm{avoid}}}^c \right\}$}.$$
In addition, by Lemma \ref{Lem.NoTouchPinned}, we have $\mathbb{P}(E_1 \cup E_2) = 1$. From the bounded convergence theorem, 
\begin{equation}\label{Eq.SavoidNumDen}
\begin{split}
&\lim_{n \rightarrow \infty}\mathbb{E}[F(B^n)\mathbf{1}_{S_{\mathrm{avoid}}}(B^n,g^n)] = \mathbb{E}[F(B)\mathbf{1}_{S_{\mathrm{avoid}}}(B,g)] \\
&\lim_{n \rightarrow \infty}\mathbb{E}[\mathbf{1}_{S_{\mathrm{avoid}}}(B^n,g^n)] = \mathbb{E}[\mathbf{1}_{S_{\mathrm{avoid}}}(B,g)] > 0,
\end{split}
\end{equation}
where in the last inequality we used $\mathbb{P}\left(E_{\mathrm{avoid}}^{\mathrm{pin}}\right) > 0$ as in Remark \ref{Rem.WellDAvoid}, which holds by Lemma \ref{Lem.StayInCorridor}.

Combining (\ref{eq.Savoidfrac}) and (\ref{Eq.SavoidNumDen}), we conclude
$$ \lim_{n \rightarrow \infty}\eabm^{b,\vec{y}\,^n,\vec{\mu}^n,g_n}[F(\mathcal{L}^n)] = \lim_{n \rightarrow \infty} \frac{\mathbb{E}[F(B^n)\mathbf{1}_{S_{\mathrm{avoid}}}(B^n,g^n)]}{\mathbb{E}[\mathbf{1}_{S_{\mathrm{avoid}}}(B^n,g^n)]} = \frac{\mathbb{E}[F(B)\mathbf{1}_{S_{\mathrm{avoid}}}(B,g)]}{\mathbb{E}[\mathbf{1}_{S_{\mathrm{avoid}}}(B,g)]},$$
which implies (\ref{Eq.WeakConvToPinned}) once we note that by Definition \ref{Def.PinnedBM} the right side is $\mathbb{E}_{\mathrm{pin}}^{b,\vec{y},g}[F(\mathcal{L})]$.
\end{proof}

The following result shows that if a $g$-avoiding reverse Brownian line ensemble with $2k$ curves is not too low at time $b > 0$, then it is not too low for each time $t \in [0,b]$.
\begin{lemma}\label{Lem.NoLowMin} Fix $k \in \mathbb{N}$, $b > 0$, and $\varepsilon \in (0,1)$. We can find $M > 0$ and $\varpi_0\in\mathbb{R}$, depending on $k, b, \varepsilon$, such that for each $\vec{\mu} \in \mathbb{R}^{2k}$, $\vec{y} \in \weyl_{2k}$, and continuous $g: [0, b] \rightarrow [-\infty, \infty)$ that satisfy
$$y_{2k} > g(b), \hspace{2mm} \mu_i = (-1)^i \varpi \mbox{ for some $\varpi \geq \varpi_0$ and $i \in \llbracket 1, 2k \rrbracket$},$$
we have
\begin{equation}\label{Eq.NoLowMin}
\pabm^{b, \vec{y}, \vec{\mu}, g} \left( \mathcal{Q}_{2k}(t) \geq y_{2k} - M \mbox{ for all } t\in[0,b] \right) \geq 1 - \varepsilon.
\end{equation}
\end{lemma}
\begin{proof} We choose $C>1$ large enough such that 
\begin{equation}\label{eq.bebsrtwrtberefervervrvrvrgewrt}
\mathbb{P}_{\operatorname{pin}}^{b,(1,0)}\left(-C<\mathcal{Q}_2(t)<\mathcal{Q}_1(t)<C \mbox{ for all } t\in[0,b] \right)\geq1- \varepsilon/(2k).
\end{equation}
In view of Lemma \ref{Lem.BrownianEnsembleConv}, applied with $k=1$, $\vec{y}=(1,0)\in\weyl_2$, and $g_n=g=-\infty$, we know that, as $\varpi\rightarrow\infty$, the probability distribution $\mathbb{P}_{\operatorname{avoid}}^{b,(1,0),(-\varpi,\varpi)}$ converges weakly to $\mathbb{P}_{\operatorname{pin}}^{b,(1,0)}$. Using \eqref{eq.bebsrtwrtberefervervrvrvrgewrt}, we can choose $\varpi_0\in\mathbb{R}$ large enough such that for all $\varpi\geq\varpi_0$, we have
\begin{equation}\label{eq.bebsrtwrtbevrvrvrgewrt}
\mathbb{P}_{\operatorname{avoid}}^{b,(1,0),(-\varpi,\varpi)}\left(-C<\mathcal{Q}_2(t)<\mathcal{Q}_1(t)<C \mbox{ for all } t\in[0,b] \right)\geq1- \varepsilon/k.
\end{equation}
We proceed to prove the lemma with the above choice of $\varpi_0$ and $M=(2k+1)C$.

By translating the ensemble, we may assume that $y_{2k}=0$.   We define $\vec{y}\,^{\mathrm{new}}=(y^{\mathrm{new}}_1,\dots,y^{\mathrm{new}}_{2k})\in\weyl_{2k}$ by $y^{\mathrm{new}}_{2i-1}=-2iC+1$ and $y^{\mathrm{new}}_{2i}=-2iC$ for $i\in\llbracket1,k\rrbracket$. Since $C>1$, we have $0>-2C+1=y^{\mathrm{new}}_1>y^{\mathrm{new}}_2>\dots>y^{\mathrm{new}}_{2k}$. Using the monotone coupling Lemma \ref{Lem.MonotoneCoupling}, with $g^{\mathrm{t}} = g$, $g^{\mathrm{b}} = -\infty$, $\vec{y}\,^{\mathrm{t}} = \vec{y}$, and $\vec{y}\,^{\mathrm{b}} = \vec{y}\,^{\mathrm{new}}$, the left side of \eqref{Eq.NoLowMin} is lower bounded by the probability of the same event under $\pabm^{b, \vec{y}\,^{\mathrm{new}}, \vec{\mu}}$.

From Definition \ref{Def.PinnedBM}, the probability measure $\pabm^{b, \vec{y}\,^{\mathrm{new}}, \vec{\mu}}$ can be regarded as the product measure $\otimes_{i=1}^k \mathbb{P}_{\mathrm{avoid}}^{b, (y^{\mathrm{new}}_{2i-1},y^{\mathrm{new}}_{2i}),(-\varpi,\varpi)}$, conditioned on $E_{\mathrm{avoid}}$. Therefore, we have
\begin{equation}\label{eq.bebsrtwrtbwrt}
\begin{split}
&\mbox{[left side of \eqref{Eq.NoLowMin}]}\geq \\
&\otimes_{i=1}^k\mathbb{P}_{\mathrm{avoid}}^{b, (y^{\mathrm{new}}_{2i-1},y^{\mathrm{new}}_{2i}),(-\varpi,\varpi)} \left(E_{\operatorname{avoid}}\cap\{\mathcal{Q}_{2k}(t) \geq  - (2k+1)C \mbox{ for all } t\in[0,b]\}\right).
\end{split}
\end{equation}

We next consider the events 
\[
E_i:=\left\{-(2i+1)C<\mathcal{Q}_{2i}(t)<\mathcal{Q}_{2i-1}(t)<-(2i-1)C \mbox{ for all } t\in[0,b]\right\},\quad i\in\llbracket1,k\rrbracket,
\]
and let $E=E_1\cap\dots\cap E_k$. Note that on the event $E$ we have $\mathcal{Q}_{2k}(t)>-(2k+1)C$ and $\mathcal{Q}_1(t)>\dots>\mathcal{Q}_{2k}(t)$ for all $t\in[0,b]$. The latter and \eqref{eq.bebsrtwrtbwrt} give
\begin{multline*}
 \mbox{[left side of \eqref{Eq.NoLowMin}]}\geq\\\prod_{i=1}^k\mathbb{P}_{\mathrm{avoid}}^{b, (y^{\mathrm{new}}_{2i-1},y^{\mathrm{new}}_{2i}),(-\varpi,\varpi)} \left(-(2i+1)C<\mathcal{Q}_{2i}(t)<\mathcal{Q}_{2i-1}(t)<-(2i-1)C \mbox{ for all } t\in[0,b]\right).
\end{multline*}
For $i\in\llbracket1,k\rrbracket$, to lower bound the $i$-th term in the product above, we translate the ensemble upwards by $2iC$, use the definitions $y^{\mathrm{new}}_{2i-1}=-2iC+1$ and $y^{\mathrm{new}}_{2i}=-2iC$, and use \eqref{eq.bebsrtwrtbevrvrvrgewrt}. Hence the $i$-th term can be lower bounded by $1- \varepsilon/k$. The right side above is lower bounded by $(1- \varepsilon/k)^k\geq1-\varepsilon$. 
\end{proof}

The next lemma establishes control on the modulus of continuity of a $g$-avoiding reverse Brownian line ensemble uniformly in the drift parameters, assuming the boundary conditions are positively separated and lie within a compact window.

\begin{lemma}\label{Lem.MOC}
    Fix $k\in\mathbb{N}$, $b,M,M^{\mathrm{bot}},\delta^{\mathrm{sep}}>0$, and $\Delta^{\mathrm{sep}} \in (0, b]$. Fix $\varepsilon,\eta>0$, and suppose $\vec{\mu}\in\mathbb{R}^{2k}$, $\vec{y}\in \weyl_{2k}$, and $g : [0,b]\to[-\infty,\infty)$ is continuous, such that $\mu_i = (-1)^i\varpi$ for some $\varpi\in\mathbb{R}$, $|y_i|\le M$ for $i \in \llbracket 1, 2k\rrbracket$, $y_i - y_{i+1} \ge \delta^{\mathrm{sep}}$ for $i \in \llbracket 1,2k-1\rrbracket$, $\sup_{t\in[0,b]} g(t) \le M^{\mathrm{bot}}$, and $y_{2k} - \sup_{t\in[b-\Delta^{\mathrm{sep}},b]} g(t) \ge \delta^{\mathrm{sep}}$. There exist $\delta>0$ and $\varpi_0\in\mathbb{R}$, depending on $\varepsilon,\eta,k,b,M,M^{\mathrm{bot}},\delta^{\mathrm{sep}},\Delta^{\mathrm{sep}}$, such that for $\varpi\ge \varpi_0$,
    \begin{equation}\label{Eq.MOC}
        \mathbb{P}^{b,\vec{y},\vec{\mu},g}_{\mathrm{avoid}}\left( \max_{i \in \llbracket 1, 2k \rrbracket} w(\mathcal{Q}_i,\delta) > \eta\right) < \varepsilon,
    \end{equation}
where we recall that $w(f,\delta)$ is the modulus of continuity from (\ref{Eq.MOCDef}).
\end{lemma}

\begin{proof} We claim that there exist $\varpi_1\in\mathbb{R}$ and $c>0$, depending on $k$, $b$, $M$, $M^{\mathrm{bot}}$, $\delta^{\mathrm{sep}}$, and $\Delta^{\mathrm{sep}}$, such that for all $\varpi\ge\varpi_1$,
\begin{equation}\label{Eq.Floor>c}
\mathbb{P}_{\mathrm{avoid}}^{b,\vec{y},\vec{\mu}}(\mathcal{Q}_{2k}(t) > g(t) \mbox{ for all }t\in[0,b]) > c.
\end{equation}
We further claim that there exist $\varpi_2 \in \mathbb{R}$ and $\delta > 0$, depending on the same set of parameters and $\varepsilon, \eta, c$, such that for $\varpi \geq \varpi_2$
\begin{equation}\label{Eq.ModCtyGood}
\pabm^{b,\vec{y},\vec{\mu}}\left( \max_{i \in \llbracket 1, 2k \rrbracket} w(\mathcal{Q}_i,\delta)> \eta \right) < c \varepsilon.
\end{equation}
Assuming (\ref{Eq.Floor>c}) and (\ref{Eq.ModCtyGood}), we get for $\varpi \geq \varpi_0 := \max(\varpi_1, \varpi_2)$
\begin{align*}
\mathbb{P}^{b,\vec{y},\vec{\mu},g}_{\mathrm{avoid}}\left( \max_{i \in \llbracket 1, 2k \rrbracket} w(\mathcal{Q}_i,\delta) > \eta\right) &= \pabm^{b,\vec{y},\vec{\mu}}\left( \max_{i \in \llbracket 1, 2k \rrbracket} w(\mathcal{Q}_i,\delta) > \eta \big\vert \mathcal{Q}_{2k}(t) > g(t) \mbox{ for all }t\in[0,b]\right)\\
&\le c^{-1} \pabm^{b,\vec{y},\vec{\mu}}\left( \max_{i \in \llbracket 1, 2k \rrbracket} w(\mathcal{Q}_i,\delta) >\eta\right) < c^{-1} \cdot c\varepsilon = \varepsilon,
\end{align*}
which implies the statement of the lemma. \\

{\bf \raggedleft Proof of (\ref{Eq.Floor>c}).} For $i\in\llbracket 1,2k\rrbracket$, we define the deterministic linear functions $\ell_i : [0,b] \to \mathbb{R}$ by
\begin{align*}
\ell_i(t) &= y_i + \frac{M^{\mathrm{bot}}+ M + 1}{\Delta^{\mathrm{sep}}} \cdot (b-t).
\end{align*}
By our assumptions, we have $g(t) <  \ell_{2k}(t) < \ell_{2k-1}(t) <\cdots < \ell_2(t)<\ell_1(t)$ for all $t\in[0,b]$ and $\ell_i(b) = y_i$. Define the events
\[
E^j_{\mathrm{corr}} = \{\ell_{2j}(t) \leq \mathcal{Q}_{2j}(t) < \mathcal{Q}_{2j-1}(t) \leq \ell_{2j-1}(t) \mbox{ for all } t\in[0,b]\}, \qquad j\in\llbracket 1,k\rrbracket.
\]
From Definition \ref{Def.PinnedBM}, we have 
\begin{equation}\label{Eq.Floor>cRed}
\begin{split}
\mathbb{P}_{\mathrm{avoid}}^{b,\vec{y},\vec{\mu}}(\mathcal{Q}_{2k}(t) > g(t) \mbox{ for all }t\in[0,b]) &\ge \frac{\bigotimes_{j=1}^k\pabm^{b,(y_{2j-1},y_{2j}),(\mu_{2j-1},\mu_{2j})}(\bigcap_{j=1}^k E^j_{\mathrm{corr}})}{\bigotimes_{j=1}^k\pabm^{b,(y_{2j-1},y_{2j}),(\mu_{2j-1},\mu_{2j})}(E_{\mathrm{avoid}})}\\
&\ge \prod_{j=1}^k \pabm^{b,(y_{2j-1},y_{2j}),(\mu_{2j-1},\mu_{2j})}(E^j_{\mathrm{corr}}).
\end{split}
\end{equation}

Note that the slope of each $\ell_i$ is bounded above by $(M^{\mathrm{bot}}+M + 1)/\Delta^{\mathrm{sep}}$. Consequently, from Lemma \ref{Lem.StayInCorridor}, there exists $\epsilon_0 > 0$, depending on $b,M,M^{\mathrm{bot}},\delta^{\mathrm{sep}},\Delta^{\mathrm{sep}}$, such that 
$$\prod_{j=1}^k \mathbb{P}_{\mathrm{pin}}^{b,(y_{2j-1},y_{2j})}(E^j_{\mathrm{corr}}) \geq \epsilon_0^k.$$
From Lemma \ref{Lem.BrownianEnsembleConv}, applied to each term in the last product, we conclude that we can find $\varpi_1 \in \mathbb{R}$, such that for $\varpi \geq \varpi_1$   
$$\prod_{j=1}^k \pabm^{b,(y_{2j-1},y_{2j}),(\mu_{2j-1},\mu_{2j})}(E^j_{\mathrm{corr}}) \geq (\epsilon_0/2)^k.$$
The latter and (\ref{Eq.Floor>cRed}) imply (\ref{Eq.Floor>c}) with $c = (\epsilon_0/2)^k$.\\

{\bf \raggedleft Proof of (\ref{Eq.ModCtyGood}).} Suppose for the sake of contradiction that no such $\varpi_2$ and $\delta$ exist. Then, for each $\delta_n = 1/n$, we can find $\varpi^n \geq n$, as well as $\vec{y}\,^n \in \weyl_{2k}$ and $\vec{\mu}\,^n \in \mathbb{R}^{2k}$, satisfying the conditions of the lemma with $\varpi = \varpi^n$, such that 
\begin{equation}\label{Eq.ModCtyBad}
\pabm^{b,\vec{y}^n,\vec{\mu}^n}\left( \max_{i \in \llbracket 1, 2k \rrbracket} w(\mathcal{Q}_i,\delta_n)> \eta \right) \geq c \varepsilon.
\end{equation}

Since $|y_i^n|\le M$ and $y^n_i - y^n_{i+1} \ge \delta^{\mathrm{sep}}$ for $i \in \llbracket 1,2k-1\rrbracket$, by possibly passing to a subsequence, we may assume that $\vec{y}\,^n \rightarrow \vec{y} \in \weyl_{2k}$. Using the weak convergence from Lemma \ref{Lem.BrownianEnsembleConv} (it is applicable as $\vec{y}\,^n \rightarrow \vec{y} \in \weyl_{2k}$ and $\varpi^n \rightarrow \infty$ by construction), and the Portmanteau theorem, we conclude for each $\delta > 0$ that 
\begin{equation*}
\limsup_{n \rightarrow \infty}\pabm^{b,\vec{y}^n,\vec{\mu}^n}\left( \max_{i \in \llbracket 1, 2k \rrbracket}w(\mathcal{Q}_i,\delta)\geq \eta\right) \leq \mathbb{P}^{b, \vec{y}}_{\mathrm{pin}} \left(  \max_{i \in \llbracket 1, 2k \rrbracket} w(\mathcal{Q}_i,\delta)\geq \eta \right).
\end{equation*} 

Combining the latter with (\ref{Eq.ModCtyBad}), and the fact that $\delta_n \downarrow 0$, we conclude that for all $\delta > 0$
$$\mathbb{P}^{b, \vec{y}}_{\mathrm{pin}} \left(  \max_{i \in \llbracket 1, 2k \rrbracket} w(\mathcal{Q}_i,\delta)\geq \eta \right) \geq c \varepsilon.$$
The latter gives our desired contradiction, since $\mathbb{P}^{b, \vec{y}}_{\mathrm{pin}}$-a.s. $\lim_{\delta \rightarrow 0+}\max_{i \in \llbracket 1, 2k \rrbracket} w(\mathcal{Q}_i,\delta) = 0$.
\end{proof}

%
%
\section{Uniform convergence over compact sets}\label{Section5} The goal of this section is to prove Theorems \ref{Thm.Convergence}, \ref{Thm.PfaffianStructure}, \ref{Thm.GibbsProperty}, and \ref{Thm.ConvergenceToAiryLE}. The main technical result we establish is Proposition \ref{Prop.Tightness}, which shows that for any sequence $\varpi_n \rightarrow \infty$, the line ensembles $\mathcal{L}^{\mathrm{hs}; \varpi_n}$ from Proposition \ref{Prop.HSAiry} form a tight sequence. This statement is the content of Section \ref{Section5.1}. In Section \ref{Section5.2} we combine Proposition \ref{Prop.Tightness} with the finite-dimensional convergence result in Lemma \ref{Lem.FDConvVarpi} to prove Theorems \ref{Thm.Convergence} and \ref{Thm.GibbsProperty}. The same section also contains the proof of Theorem \ref{Thm.PfaffianStructure}, which is essentially a corollary of Lemma \ref{Lem.PointProcessConvergenceVarpi}. In Section \ref{Section5.3} we establish Theorem \ref{Thm.ConvergenceToAiryLE}, by combining the finite-dimensional convergence from Lemma \ref{Lem.FDConvT} with some results from \cite{CorHamA}. Lastly, Section \ref{SectionSpecial} presents the proof of the ``infinitely many atoms'' lemma, Lemma \ref{Lem.InfinitelyManyAtoms}.

%
%
\subsection{Tightness}\label{Section5.1} We continue with the notation from Sections \ref{Section1} and \ref{Section4}. The goal of this section is to establish the following statement.

\begin{proposition}\label{Prop.Tightness} Assume the same notation as in Proposition \ref{Prop.HSAiry}. Let $\varpi_n \in \mathbb{R}$ be a sequence such that $\lim_{n \rightarrow \infty }\varpi_{n} = \infty$. Then, $\mathcal{L}^{\mathrm{hs}; \varpi_n}$ forms a tight sequence in $C(\mathbb{N} \times [0, \infty))$. In addition, if $\mathcal{L}^\infty$ is any subsequential limit, then $\mathcal{L}^\infty$ satisfies the pinned half-space Brownian Gibbs property from Definition \ref{Def.PinnedBGP}, and the restriction of $\mathcal{L}^\infty$ to $\mathbb{N}\times(0,\infty)$ satisfies the Brownian Gibbs property from Definition \ref{Def.BGPVanilla} with $\Lambda = (0, \infty)$.
\end{proposition}
\begin{proof} We split the proof into four steps for clarity. Tightness is proven in the first three steps, and the Gibbs properties of subsequential limits are established in Step 4.\\

\noindent\textbf{Step 1.} In this step we prove tightness. From \cite[Lemma 2.4]{DEA21}, it suffices to show that for each $k \in \mathbb{N}$ the sequence $\{\mathcal{L}_k^{\mathrm{hs};\varpi_n}(1)\}_{n\ge 1}$ is tight, and for any $\varepsilon,\eta>0$ and $b > 3$ (this will be convenient later), we can find $\delta>0$ and $n_0\in\mathbb{N}$ so that for all $n\ge n_0$, 
\begin{equation}\label{Eq.MOCclaim}
\mathbb{P}\left(w(\mathcal{L}_k^{\mathrm{hs};\varpi_n}[0,b],\delta) > \eta \right) < \varepsilon,    
\end{equation}
where we recall that $w(f,\delta)$ is the modulus of continuity from (\ref{Eq.MOCDef}). In Lemma \ref{Lem.FDConvVarpi} we showed that $\mathcal{L}_k^{\mathrm{hs};\varpi_n}(1)$ converge weakly as $n \rightarrow \infty$, and thus form a tight sequence. Consequently, we only need to prove \eqref{Eq.MOCclaim}. For the remainder of the proof we fix $\varepsilon,\eta>0$, $b > 3$, and all the constants we encounter depend on $b,\varepsilon,\eta$, as well as the sequence $\{\varpi_n\}_{n \geq 1}$ --- we do not mention this further.\\

We claim that there exist $M^{\mathrm{bot}}>0$ and $n_1\in\mathbb{N}$, so that for all $n\ge n_1$,
\begin{equation}\label{Eq.NBMevent}
    \mathbb{P}(A^c) < \varepsilon/4, \mbox{ where } A:= \left\{\sup_{t\in[0,b]} \mathcal{L}_{2k+1}^{\mathrm{hs};\varpi_n}(t) \le M^{\mathrm{bot}}\right\}.
\end{equation}
In addition, we claim that we can find $\delta^{\mathrm{sep}}, M > 0$, $\Delta^{\mathrm{sep}} \in (0, b]$ and $n_2\in\mathbb{N}$, so that for all $n\ge n_2$,
\begin{align}
    \mathbb{P}(B^c) &< \varepsilon/4, \mbox{ where } \label{Eq.SepEvent} \\
    B &:= \left\{ \max_{1\le i\le 2k} |\mathcal{L}_i^{\mathrm{hs};\varpi_n}(b)| \le M\right\} \cap \left\{ \min_{1\le i\le 2k-1} \left(\mathcal{L}_i^{\mathrm{hs};\varpi_n}(b) - \mathcal{L}_{i+1}^{\mathrm{hs};\varpi_n}(b)\right) \ge \delta^{\mathrm{sep}}\right\} \nonumber\\ &\qquad\qquad\qquad\cap \left\{ \mathcal{L}_{2k}^{\mathrm{hs};\varpi_n}(b)-\sup_{t\in[b-\Delta^{\mathrm{sep}},b]} \mathcal{L}_{2k+1}^{\mathrm{hs};\varpi_n}(t) \ge \delta^{\mathrm{sep}}\right\}. \nonumber
\end{align}
We verify these two claims in Steps 2 and 3. Here, we assume their validity, and prove \eqref{Eq.MOCclaim}.\\

Let $C$ denote the event in \eqref{Eq.MOCclaim}. Using that $A,B \in \mathcal{F}_{\mathrm{ext}}(\llbracket 1,2k\rrbracket \times[0,b])$, we have by the half-space Brownian Gibbs property, see Definition \ref{Def.BGP}, that  
\begin{equation}\label{Eq.ABC}
\begin{split}
    \mathbb{P}(A\cap B\cap C) &= \mathbb{E} \left[\mathbf{1}_{A\cap B} \cdot \mathbb{E}\left[\mathbf{1}_C \mid \mathcal{F}_{\mathrm{ext}}(\llbracket 1,2k\rrbracket \times[0,b])\right]\right]\\
    &= \mathbb{E} \left[\mathbf{1}_{A\cap B} \cdot \pabm^{b,\vec{y}^n,\vec{\mu}^n,g^n}\left(w(\mathcal{Q}_k,\delta)>\eta\right)\right],
\end{split}
\end{equation}
where $\vec{y}^n = (\mathcal{L}_1^{\mathrm{hs};\varpi_n}(b),\dots,\mathcal{L}_{2k}^{\mathrm{hs};\varpi_n}(b))$, $\mu^n_i = (-1)^i \sqrt{2}\,\varpi_n$, and $g^n = \mathcal{L}_{2k+1}^{\mathrm{hs};\varpi_n}[0,b]$. 

On the event $A\cap B$, the conditions of Lemma \ref{Lem.MOC} are satisfied. Taking $n_0 \ge \max(n_1,n_2)$ large enough so that $\sqrt{2}\,\varpi_n \ge \varpi_0$ for all $n\ge n_0$, we may find $\delta>0$ so that the probability inside the expectation in the second line of \eqref{Eq.ABC} is at most $\varepsilon/2$ for $n\ge n_0$. It then follows from \eqref{Eq.NBMevent}, \eqref{Eq.SepEvent}, \eqref{Eq.ABC}, and a union bound that $\mathbb{P}(C) < \varepsilon$ for $n\ge n_0$, proving \eqref{Eq.MOCclaim}.\\

\noindent\textbf{Step 2.} In this step, we establish a ``no big max'' estimate for the sequence $\mathcal{L}^{\mathrm{hs};\varpi_n}$, which in particular implies \eqref{Eq.NBMevent}. We will show that for any $b,\varepsilon>0$, there exist $H>0$ and $n_1\in\mathbb{N}$, so that for all $n\ge n_1$,
\begin{equation}\label{Eq.nobigmax}
    \mathbb{P}\left(\sup_{t\in[0,b]} \mathcal{L}_1^{\mathrm{hs};\varpi_n}(t) > H \right) < \varepsilon.
\end{equation}
Since $\mathcal{L}_{2k+1}^{\mathrm{hs};\varpi_n}(t) \le \mathcal{L}_1^{\mathrm{hs};\varpi_n}(t)$ a.s., this implies \eqref{Eq.NBMevent} after replacing $\varepsilon$ with $\varepsilon/4$ and $H$ with $M^{\mathrm{bot}}$. 

By Lemma \ref{Lem.FDConvVarpi}, we can find $M_1,M_2>0$ and $n_1\in\mathbb{N}$, so that for all $n\ge n_1$,
\begin{equation}\label{Eq.NBMevents}
\begin{split}
\mathbb{P}(F^c) &< \varepsilon/4, \mbox{ where } F := \left\{\mathcal{L}_1^{\mathrm{hs};\varpi_n}(2b) \ge -M_1\right\},\\
\mathbb{P}(G) &< \varepsilon/8, \mbox{ where } G:= \left\{\mathcal{L}_1^{\mathrm{hs};\varpi_n}(b) \ge M_2 \right\}.
\end{split}
\end{equation}
We now set $H := 2M_2 + M_1$, and proceed to show (\ref{Eq.nobigmax}).\\

Fix a finite set $\mathsf{S} = \{s_1, \dots, s_m\} \subset (0,b)$ with $s_1 < s_2 < \cdots < s_m$, and define the events
$$E_i^{H; \mathsf{S}} = \{\mathcal{L}_1^{\mathrm{hs};\varpi_n}(s_i) > H \mbox{ and }\mathcal{L}_1^{\mathrm{hs};\varpi_n}(s_j) \leq H \mbox{ for }j \in \llbracket 1, i- 1 \rrbracket \}, \mbox{ and } E^{H; \mathsf{S}} = \bigsqcup_{i = 1}^mE_i^{H; \mathsf{S}}.$$
By \cite[Lemma 2.9]{DY25}, $\mathcal{L}^{\mathrm{hs};\varpi_n}$ satisfies the Brownian Gibbs property of Definition \ref{Def.BGPVanilla}. Using this property and stochastic monotonicity for Brownian bridge ensembles, \cite[Lemmas 2.6, 2.7]{CorHamA}, we get
\begin{equation}\label{Eq.EFG}
\begin{split}
&\mathbb{P}\left(E^{H;\mathsf{S}}_i\cap F\cap G\right) = \mathbb{E}\left[\mathbf{1}_{E^{H;\mathsf{S}}_i \cap F} \cdot \mathbb{E} \left[\mathbf{1}_G \mid \mathcal{F}_{\mathrm{ext}}(\{1\}\times (s_i,2b))\right]\right]\\
&= \mathbb{E}\left[ \mathbf{1}_{E^{H;\mathsf{S}}_i \cap F} \cdot \mathbb{P}^{s_i,2b,\mathcal{L}_1^{\mathrm{hs},\varpi_n}(s_i),\mathcal{L}_1^{\mathrm{hs};\varpi_n}(2b),\infty, \mathcal{L}_2^{\mathrm{hs};\varpi_n}[0,2b]}_{\mathrm{avoid}; \mathrm{Br}}\left(\mathcal{Q}_1(b)\ge M_2\right)\right]\\
    &\ge \mathbb{E}\left[ \mathbf{1}_{E^{H;\mathsf{S}}_i \cap F} \cdot \mathbb{P}_{\mathrm{free}; \mathrm{Br}}^{s_i,2b,H,-M_1}\left(\mathcal{Q}_1(b)\ge M_2\right)\right],
\end{split}
\end{equation}
where $\mathbb{P}_{\mathrm{free}; \mathrm{Br}}^{s_i,2b,H,-M_1}$ is the law of a Brownian bridge from $\mathcal{Q}_1(s_i) = H$ to $\mathcal{Q}_1(2b) = -M_1$.

Notice that since $\mathcal{Q}_1$ has law $\mathbb{P}_{\mathrm{free}; \mathrm{Br}}^{s_i,2b,H,-M_1}$, we know that $\mathcal{Q}_1(b)$ is a normal variable with mean
$$\frac{b}{2b - s_i} \cdot H + \frac{b-s_i}{2b -s_i} \cdot (-M_1) \geq \frac{1}{2} (2M_2 + M_1) - \frac{1}{2} M_1 \geq M_2,$$
where we used that $H = 2M_2 + M_1$ and $s_i \in (0,b)$. The latter implies that the probability on the last line of (\ref{Eq.EFG}) is at least $1/2$, and so
\begin{equation*}
\begin{split}
&\mathbb{P}(E^{H;\mathsf{S}}_i\cap F\cap G) \ge (1/2) \mathbb{P} \left(E^{H;\mathsf{S}}_i \cap F \right).
\end{split}
\end{equation*}
Summing over $i$, we conclude
\begin{equation*}\label{Eq.EFG3}
\begin{split}
&\mathbb{P}(G) \geq \mathbb{P}(E^{H;\mathsf{S}}\cap F\cap G) \ge (1/2) \mathbb{P} \left(E^{H;\mathsf{S}} \cap F \right) \geq (1/2) \mathbb{P} \left(E^{H;\mathsf{S}}\right) - (1/2) \mathbb{P}(F^c).
\end{split}
\end{equation*}

Combining the last inequality with (\ref{Eq.NBMevents}), we conclude for all finite $\mathsf{S} \subset (0,b)$ and $n \geq n_1$, that
\begin{equation*}
\varepsilon/2 > \mathbb{P} \left(E^{H;\mathsf{S}}\right).
\end{equation*}
Fixing finite sets $\mathsf{S}_1 \subset \mathsf{S}_2 \subset \cdots$ with $\mathbb{Q} \cap (0,b) = \cup_{m = 1} \mathsf{S}_m$, we have by continuity of the measure from below that for all $n \geq n_1$,
$$\mathbb{P}\left(\sup_{t\in (0,b) \cap \mathbb{Q}} \mathcal{L}_1^{\mathrm{hs};\varpi_n}(t) > H \right) = \lim_{m \rightarrow \infty}\mathbb{P} \left(E^{H;\mathsf{S}_m}\right) \leq \varepsilon/2.$$
The last equation implies (\ref{Eq.nobigmax}), since $\mathcal{L}_1^{\mathrm{hs};\varpi_n}$ is continuous and $\mathbb{Q} \cap (0,b)$ is dense in $[0,b]$. \\

\noindent\textbf{Step 3.} In this step, we prove \eqref{Eq.SepEvent}. 

From Lemma \ref{Lem.FDConvVarpi}, applied to $m = 1$ and $s_1 = b$, we know that $\{\mathcal{A}_i^{\mathrm{hs};\varpi_n}(b)\}_{i\ge 1}$ converges to a random vector $X = \{X_i\}_{i \geq 1}$, which almost surely satisfies $X_1 > X_2 > \cdots$. Since $\mathcal{L}_i^{\mathrm{hs};\varpi_n}(b) = 2^{-1/2}(\mathcal{A}_i^{\mathrm{hs};\varpi_n}(b) - b^2)$, we conclude that we can find $n_{2,1}, \delta_1^{\mathrm{sep}}, M > 0$, such that for $n \geq n_{2,1}$
\begin{equation}\label{Eq.Prop51Red1}
\mathbb{P}\left(\max_{1\le i\le 2k} |\mathcal{L}_i^{\mathrm{hs};\varpi_n}(b)| > M \right) < \varepsilon/12, \hspace{2mm} \mathbb{P}\left( \min_{1\le i\le 2k-1} \left(\mathcal{L}_i^{\mathrm{hs};\varpi_n}(b) - \mathcal{L}_{i+1}^{\mathrm{hs};\varpi_n}(b)\right) < \delta_1^{\mathrm{sep}} \right) < \varepsilon/12.
\end{equation}

We claim that we can find $\delta^{\mathrm{sep}}_2,\Delta^{\mathrm{sep}}\in (0, b]$ and $n_{2,2}\in\mathbb{N}$, such that for all $n\ge n_{2,2}$,
\begin{equation}\label{Eq.IntervalSep}
 \mathbb{P}\left( \mathcal{L}_{2k}^{\mathrm{hs};\varpi_n}(b)-\sup_{t\in[b-\Delta^{\mathrm{sep}},b]} \mathcal{L}_{2k+1}^{\mathrm{hs};\varpi_n}(t) < \delta_2^{\mathrm{sep}}\right) < \varepsilon/12.
\end{equation}
Note that (\ref{Eq.Prop51Red1}) and (\ref{Eq.IntervalSep}) imply (\ref{Eq.SepEvent}) with $n_2 = \max(n_{2,1}, n_{2,2})$ and $\delta^{\mathrm{sep}} = \min(\delta_1^{\mathrm{sep}}, \delta_2^{\mathrm{sep}})$.\\

We deduce (\ref{Eq.IntervalSep}) from Lemma \ref{Lem.FDConvVarpi} and \cite[Proposition 3.6]{CorHamA}. 

Let $\mathcal{K}^n$ denote the $\llbracket 1,2k+2\rrbracket$-indexed line ensemble on $[-3,3]$, given by $\mathcal{K}^n_i(t) = \mathcal{L}_i^{\mathrm{hs};\varpi_n}(t+b)$ (here is where we finally use that $b > 3$). We claim that $\{\mathcal{K}^n\}_{n\ge 1}$ satisfies the three conditions in Hypothesis $(H)_{2k+1,T+2}$ of \cite[Definition 3.3]{CorHamA} with $T = 1$. Indeed, Hypothesis $(H1)_{2k+1,3}$ says that $\mathcal{K}^n$ is a non-intersecting line ensemble with the Brownian Gibbs property, which is immediate from the Brownian Gibbs property of $\mathcal{L}^{\mathrm{hs};\varpi_n}$ from \cite[Lemma 2.9]{DY25}. Hypothesis $(H2)_{2k+1,3}$ says that for each finite $\mathsf{S}\subseteq[-3,3]$, the joint distribution of $\{\mathcal{K}_i^n(s) : i \in \llbracket 1, 2k + 1 \rrbracket, s\in \mathsf{S}\}$ converges weakly, which follows from Lemma \ref{Lem.FDConvVarpi}. Lastly, $(H3)_{2k+1,3}$ says that for each $\hat{\varepsilon}>0$ and $t\in[-3,3]$ we can find $\hat{\delta}>0$, so that for large $n$, 
$$\mathbb{P}(\min_{1\le i\le 2k} |\mathcal{K}_i^n(t)-\mathcal{K}_{i+1}^n(t)| > \hat{\delta})<\hat{\varepsilon},$$
which follows again from Lemma \ref{Lem.FDConvVarpi}, and especially the strict inequalities in (\ref{Eq.OrdX}). 

It follows from \cite[Proposition 3.6]{CorHamA} that the restriction of $\mathcal{K}^n$ to $\llbracket 1,2k+1\rrbracket \times [-1,1]$ converges weakly (in the uniform topology) as $n\to\infty$ to a {\em non-intersecting} line ensemble. In particular, shifting to the right by $b$, this implies that the restriction of $\mathcal{L}^{\mathrm{hs};\varpi_n}$ to $\llbracket 1,2k+1\rrbracket \times[b-1,b+1]$ converges weakly as $n\to\infty$ to a non-intersecting line ensemble, which we denote by $\mathcal{L}^{\infty}$. Consequently, we can find $\delta^{\mathrm{sep}}_2 > 0$ and $\Delta^{\mathrm{sep}} \in (0, 1)$, such that
\begin{equation*}
 \mathbb{P}\left( \inf_{t\in[b-1,b+1]} \left(\mathcal{L}^{\infty}_{2k}(t)-\mathcal{L}^{\infty}_{2k+1}(t)\right) < 2\delta_2^{\mathrm{sep}}\right) < \frac{\varepsilon}{24}, \quad \mathbb{P}\left(\sup_{|t-b| \le \Delta^{\mathrm{sep}}} |\mathcal{L}^{\infty}_{2k}(t) - \mathcal{L}^{\infty}_{2k}(b)| > \delta^{\mathrm{sep}}_2\right) < \frac{\varepsilon}{24}.
\end{equation*}
In view of the weak convergence, by a union bound this implies \eqref{Eq.IntervalSep} for large $n$.\\

\noindent\textbf{Step 4.} In this final step, we show that any subsequential limit $\mathcal{L}^{\infty}$ of $\mathcal{L}^{\mathrm{hs};\varpi_n}$ satisfies the pinned half-space Brownian Gibbs property and the Brownian Gibbs property on $\Lambda = (0,\infty)$. Let $\mathcal{L}^{\mathrm{hs};\varpi_{n_v}}$ be a subsequence weakly converging to $\mathcal{L}^{\infty}$. By Skorohod's representation theorem \cite[Theorem 6.7]{Billing}, we may assume that $\mathcal{L}^{\mathrm{hs};\varpi_{n_v}}$, $\mathcal{L}^{\infty}$ are defined on the same space and the convergence is uniform on compact sets $\mathbb{P}$-a.s.~as $v\to\infty$. 

Recalling Definitions \ref{Def.PinnedBGP} and \ref{Def.BGPVanilla}, we see that we need to establish the following statements. Firstly, for each $k \in \mathbb{N}$
\begin{equation}\label{Eq.NonIntRed}
\mathbb{P}\left(\mathcal{L}^{\infty}_i(t) >\mathcal{L}^{\infty}_{i+1}(t) \mbox{ for all }  t\in (0,\infty) , i \in \llbracket 1, 2k \rrbracket \right) = 1.
\end{equation}
Secondly, for each $[a,b] \subset (0, \infty)$, $k \in \mathbb{N}$, $S = \llbracket s_1, s_2 \rrbracket \subseteq \llbracket 1, 2k \rrbracket$, bounded Borel-measurable $F: C(S \times [a,b]) \rightarrow \mathbb{R}$ and $U \in \mathcal{F}_{\operatorname{ext}}(S \times (a,b))$, we have
\begin{equation}\label{Eq.BGPRed}
\mathbb{E}\left[ F\left( \mathcal{L}^{\infty} \vert_{S \times [a,b]} \right) \cdot {\bf 1}_U \right] = \mathbb{E}\left[ \mathbb{E}_{\operatorname{avoid}; \mathrm{Br}}^{a,b,\vec{x},\vec{y},f,g} \left[ F(\mathcal{Q}) \right]\cdot {\bf 1}_U \right], 
\end{equation}
where $\vec{x} = (\mathcal{L}^{\infty}_{s_1} (a), \dots, \mathcal{L}^{\infty}_{s_2} (a))$, $\vec{y} = (\mathcal{L}^{\infty}_{s_1} (b), \dots, \mathcal{L}^{\infty}_{s_2} (b))$, $g = \mathcal{L}_{s_2+1}[a,b]$, $f = \mathcal{L}_{s_1-1}^{\infty}[a,b]$ for $s_1 \geq 2$, and $f = \infty$ if $s_1 = 1$. In addition, $\mathcal{Q}$ has distribution $\mathbb{P}_{\operatorname{avoid}; \mathrm{Br}}^{a,b,\vec{x},\vec{y},f,g}$ as in Definition \ref{Def.fgAvoidingBE} with curves indexed by $S$ rather than $\llbracket 1, \dots, s_2 - s_1 + 1\rrbracket$.

Lastly, for each $b > 0$,  $k \in \mathbb{N}$, bounded Borel-measurable $G: C(\llbracket 1, 2k \rrbracket \times [0,b]) \rightarrow \mathbb{R}$ and $V \in \mathcal{F}_{\operatorname{ext}}(\llbracket 1, 2k \rrbracket \times [0,b))$, we have
\begin{equation}\label{Eq.PBGPRed}
\mathbb{E}\left[ G\left( \mathcal{L}^{\infty} \vert_{\llbracket 1, 2k \rrbracket \times [0,b]} \right) \cdot {\bf 1}_V \right] = \mathbb{E}\left[ \mathbb{E}_{\mathrm{pin}}^{b,\vec{y},g} \left[ G(\mathcal{Q}) \right]\cdot {\bf 1}_V \right],
\end{equation}
where $\vec{y} = (\mathcal{L}^{\infty}_{1} (b), \dots, \mathcal{L}^{\infty}_{2k} (b))$, $g = \mathcal{L}_{2k+1}[0,b]$, and $\mathcal{Q}$ has distribution $\mathbb{P}_{\mathrm{pin}}^{b,\vec{y},g} $ as in Definition \ref{Def.PinnedBM}.

By the defining properties of conditional expectation, we have that (\ref{Eq.NonIntRed}) and (\ref{Eq.BGPRed}) verify the conditions of Definition \ref{Def.BGPVanilla}, while (\ref{Eq.NonIntRed}) and (\ref{Eq.PBGPRed}) verify those of Definition \ref{Def.PinnedBGP}. We also mention that by Lemma \ref{Lem.FDConvVarpi}, we have that $\mathcal{L}^{\infty}_i(a) > \mathcal{L}^{\infty}_{i+1}(a)$ and $\mathcal{L}^{\infty}_i(b) > \mathcal{L}^{\infty}_{i+1}(b)$ almost surely for all $i \geq 1$ and $b>a> 0$, so that the expectations on the right sides of (\ref{Eq.BGPRed}) and (\ref{Eq.PBGPRed}) are well-defined.\\

Establishing statements like (\ref{Eq.NonIntRed}), (\ref{Eq.BGPRed}) and (\ref{Eq.PBGPRed}) is by now quite standard, see e.g. the proofs of \cite[Theorem 2.26(ii)]{DEA21} and \cite[Theorem 5.1]{dimitrov2024tightness}, and so we will be brief. Firstly, we can verbatim repeat the argument in Step 1 of the proof of \cite[Theorem 5.1]{dimitrov2024tightness} to conclude that (\ref{Eq.NonIntRed}) follows from (\ref{Eq.BGPRed}). We are thus left with proving (\ref{Eq.BGPRed}) and (\ref{Eq.PBGPRed}).\\

{\bf \raggedleft Proof of (\ref{Eq.BGPRed}).} Fix $m \in \mathbb{N}$, $k_1, \dots, k_m \in \mathbb{N}$, $t_1, \dots, t_m \in (0,\infty)$ and bounded continuous $h_1, \dots, h_m : \mathbb{R} \rightarrow \mathbb{R}$. Define $R = \{i \in \llbracket 1, m \rrbracket: k_i \in S, t_i \in [a,b]\}$. We claim that
\begin{equation}\label{S54E7}
\mathbb{E}\left[ \prod_{i = 1}^m h_i(\mathcal{L}^{\infty}_{k_i}(t_i)) \right] = \mathbb{E}\left[ \prod_{i \not \in R} h_i(\mathcal{L}^{\infty}_{k_i}(t_i))  \cdot \mathbb{E}_{\operatorname{avoid}; \mathrm{Br}}^{a,b,\vec{x},\vec{y},f,g} \left[ \prod_{i  \in R} h_i(\mathcal{Q}_{k_i}(t_i))   \right] \right]. 
\end{equation}

From \cite[Lemma 2.9]{DY25}, we know that $\mathcal{L}^{\mathrm{hs};\varpi_{n_v}}$ satisfies the Brownian Gibbs property, and so
\begin{equation}\label{S54E8}
\mathbb{E}\left[ \prod_{i = 1}^m h_i(\mathcal{L}^{\mathrm{hs};\varpi_{n_v}}_{k_i}(t_i)) \right] = \mathbb{E}\left[ \prod_{i \not \in R} h_i(\mathcal{L}^{\mathrm{hs};\varpi_{n_v}}_{k_i}(t_i))  \cdot \mathbb{E}_{\operatorname{avoid}; \mathrm{Br}}^{a,b,\vec{x}^v,\vec{y}^v,f^v,g^v} \left[ \prod_{i  \in R} h_i(\mathcal{Q}_{k_i}(t_i))   \right] \right], 
\end{equation}
where $\vec{x}^v = (\mathcal{L}^{\mathrm{hs};\varpi_{n_v}}_{s_1} (a), \dots, \mathcal{L}^{\mathrm{hs};\varpi_{n_v}}_{s_2} (a))$, $\vec{y}^v = (\mathcal{L}^{\mathrm{hs};\varpi_{n_v}}_{s_1} (b), \dots, \mathcal{L}^{\mathrm{hs};\varpi_{n_v}}_{s_2} (b))$, $g^v = \mathcal{L}_{s_2+1}[a,b]$, $f^v = \mathcal{L}_{s_1-1}^{\mathrm{hs};\varpi_{n_v}}[a,b]$ for $s_1 \geq 2$, and $f^v = \infty$ if $s_1 = 1$. Using that $\mathcal{L}^{\mathrm{hs};\varpi_{n_v}}$ converges uniformly over compact sets to $\mathcal{L}^{\infty}$, and the continuity of $h_i$, we see that the random variables on the left and right side of (\ref{S54E8}) converge to those in (\ref{S54E7}). We mention that the convergence 
$$\lim_{v \rightarrow \infty} \mathbb{E}_{\operatorname{avoid}; \mathrm{Br}}^{a,b,\vec{x}^v,\vec{y}^v,f^v,g^v} \left[ \prod_{i  \in R} h_i(\mathcal{Q}_{k_i}(t_i))   \right] \overset{a.s.}{=} \mathbb{E}_{\operatorname{avoid}; \mathrm{Br}}^{a,b,\vec{x},\vec{y},f,g} \left[ \prod_{i  \in R} h_i(\mathcal{Q}_{k_i}(t_i))   \right]$$
follows from Lemma \ref{Lem.BridgeEnsemblesCty}.

From the last paragraph, we see that (\ref{S54E7}) follows by taking the $v \rightarrow \infty$ limit in (\ref{S54E8}) and applying the bounded convergence theorem. The fact that (\ref{S54E7}) implies (\ref{Eq.BGPRed}) now follows by the monotone class argument from Step 2 in the proof of \cite[Theorem 5.1]{dimitrov2024tightness}, which carries over verbatim.\\

{\bf \raggedleft Proof of (\ref{Eq.PBGPRed}).} Fix $m \in \mathbb{N}$, $k_1, \dots, k_m \in \mathbb{N}$, $t_1, \dots, t_m \in [0,\infty)$ and bounded continuous $h_1, \dots, h_m : \mathbb{R} \rightarrow \mathbb{R}$. Define $R = \{i \in \llbracket 1, m \rrbracket: k_i \in \llbracket 1, 2k \rrbracket, t_i \in [0,b]\}$. We claim that
\begin{equation}\label{S54E7V2}
\mathbb{E}\left[ \prod_{i = 1}^m h_i(\mathcal{L}^{\infty}_{k_i}(t_i)) \right] = \mathbb{E}\left[ \prod_{i \not \in R} h_i(\mathcal{L}^{\infty}_{k_i}(t_i))  \cdot \mathbb{E}_{\mathrm{pin}}^{b,\vec{y},g} \left[ \prod_{i  \in R} h_i(\mathcal{Q}_{k_i}(t_i))   \right] \right]. 
\end{equation}

From Proposition \ref{Prop.HSAiry}, $\mathcal{L}^{\mathrm{hs};\varpi_{n_v}}$ satisfies the half-space Brownian Gibbs property, and so
\begin{equation}\label{S54E8V2}
\mathbb{E}\left[ \prod_{i = 1}^m h_i(\mathcal{L}^{\mathrm{hs};\varpi_{n_v}}_{k_i}(t_i)) \right] = \mathbb{E}\left[ \prod_{i \not \in R} h_i(\mathcal{L}^{\mathrm{hs};\varpi_{n_v}}_{k_i}(t_i))  \cdot \mathbb{E}_{\operatorname{avoid}}^{b,\vec{y}^v,\vec{\mu}_{2k}^v,g^v} \left[ \prod_{i  \in R} h_i(\mathcal{Q}_{k_i}(t_i))   \right] \right], 
\end{equation}
where $\vec{y}^v = (\mathcal{L}^{\mathrm{hs};\varpi_{n_v}}_{1} (b), \dots, \mathcal{L}^{\mathrm{hs};\varpi_{n_v}}_{2k} (b))$, $g^v = \mathcal{L}_{2k+1}[0,b]$, and $\vec{\mu}_{2k}^v \in \mathbb{R}^{2k}$ has odd coordinates $-\sqrt{2}\varpi_{n_v}$ and even ones $\sqrt{2}\varpi_{n_v}$.

As before, (\ref{S54E7V2}) follows by taking the $v \rightarrow \infty$ limit in (\ref{S54E8V2}) and applying the bounded convergence theorem. Here, we used that
$$\lim_{v \rightarrow \infty} \mathbb{E}_{\operatorname{avoid}}^{b,\vec{y}^v,\vec{\mu}_{2k}^v,g^v} \left[ \prod_{i  \in R} h_i(\mathcal{Q}_{k_i}(t_i))   \right] \overset{a.s.}{=} \mathbb{E}_{\mathrm{pin}}^{b,\vec{y},g} \left[ \prod_{i  \in R} h_i(\mathcal{Q}_{k_i}(t_i))   \right],$$
which follows from Lemma \ref{Lem.BrownianEnsembleConv}. The fact that (\ref{S54E7V2}) implies (\ref{Eq.PBGPRed}) follows by verbatim repeating the monotone class argument from Step 2 in the proof of \cite[Theorem 5.1]{dimitrov2024tightness}.
\end{proof}

%
%
\subsection{Proof of Theorems \ref{Thm.Convergence}, \ref{Thm.PfaffianStructure} and \ref{Thm.GibbsProperty}}\label{Section5.2} 

\begin{proof}[Proof of Theorems \ref{Thm.Convergence} and \ref{Thm.GibbsProperty}]
Let $\{\varpi_n\}_{n\ge 1}$ be any sequence converging to infinity. From Proposition \ref{Prop.Tightness}, we know that $\mathcal{L}^{\mathrm{hs};\varpi_n}$ forms a tight sequence. Suppose that $\mathcal{L}^{\infty,1}$ and $\mathcal{L}^{\infty,2}$ are two subsequential limits. By Lemma \ref{Lem.FDConvVarpi}, and the continuous mapping theorem \cite[Theorem 2.7]{Billing}, we know that for any finite set $\mathsf{S} = \{s_1, \dots, s_m\} \subset (0, \infty)$, we have 
\begin{equation*}
\begin{split}
&\left(\mathcal{L}^{\infty,1}_i(s_j): i \geq 1, j \in \llbracket 1, m \rrbracket \right) \overset{f.d.}{=} \left( 2^{-1/2}(X^{j,\mathsf{S}}_i - s_j^2): i \geq 1, j \in \llbracket 1, m \rrbracket \right), \mbox{ and }\\
&\left(\mathcal{L}^{\infty,2}_i(s_j): i \geq 1, j \in \llbracket 1, m \rrbracket \right) \overset{f.d.}{=} \left( 2^{-1/2}(X^{j,\mathsf{S}}_i - s_j^2): i \geq 1, j \in \llbracket 1, m \rrbracket \right).
\end{split}
\end{equation*}

As finite-dimensional sets form a separating class, see \cite[Lemma 3.1]{DimMat}, we conclude that $\mathcal{L}^{\infty,1} \overset{d}{=}\mathcal{L}^{\infty,2}$. The latter shows that $\mathcal{L}^{\mathrm{hs};\varpi_n}$ has at most one subsequential limit, and by tightness it has exactly one. We call it $\mathcal{L}^{\mathrm{hs};\infty}$ and it satisfies $\mathcal{L}^{\mathrm{hs};\varpi_n} \Rightarrow \mathcal{L}^{\mathrm{hs};\infty}$.

Notice that if we picked a different sequence $\tilde{\varpi}_n \rightarrow \infty$ and $\mathcal{L}^{\mathrm{hs};\tilde{\varpi}_n} \Rightarrow \tilde{\mathcal{L}}^{\mathrm{hs};\infty}$, we have $\tilde{\mathcal{L}}^{\mathrm{hs};\infty} \overset{d}{=} \mathcal{L}^{\mathrm{hs};\infty}$, so that the limit $\mathcal{L}^{\mathrm{hs};\infty}$ does not depend on the particular sequence $\{\varpi_n\}_{n\ge 1}$ we chose. To see the latter, we can take the alternating sequence $\hat{\varpi}_{2i} = \varpi_{i}$ and $\hat{\varpi}_{2i-1} = \tilde{\varpi}_i$, for which we have by the above argument that $\mathcal{L}^{\mathrm{hs};\hat{\varpi}_n}$ is weakly convergent. As $\tilde{\mathcal{L}}^{\mathrm{hs};\infty}, \mathcal{L}^{\mathrm{hs};\infty}$ are both weak subsequential limits of $\mathcal{L}^{\mathrm{hs};\hat{\varpi}_n}$, the uniqueness of the weak limit gives $\tilde{\mathcal{L}}^{\mathrm{hs};\infty} \overset{d}{=} \mathcal{L}^{\mathrm{hs};\infty}$.  \\

Our work so far shows that $\mathcal{L}^{\mathrm{hs};\varpi_n} \Rightarrow \mathcal{L}^{\mathrm{hs};\infty}$ and the limit is independent of the sequence $\{\varpi_n\}_{n \geq 1}$. By (\ref{eq:Parabolic HSA}) and the continuous mapping theorem \cite[Theorem 2.7]{Billing}, we conclude that $\mathcal{A}^{\mathrm{hs};\varpi_n} \Rightarrow \mathcal{A}^{\mathrm{hs};\infty}$, where the latter satisfies (\ref{Eq.ParabolicHSALE}). This proves the weak convergence in Theorem \ref{Thm.Convergence}.

By Proposition \ref{Prop.Tightness}, we know that $\mathcal{L}^{\mathrm{hs};\infty}$ satisfies the Brownian Gibbs property on $(0, \infty)$, and is thus non-intersecting there. By continuity, we conclude $\mathcal{L}^{\mathrm{hs};\infty}$ is ordered on $[0, \infty)$. From (\ref{Eq.ParabolicHSALE}), we conclude the same for $\mathcal{A}^{\mathrm{hs};\infty}$, completing the proof of Theorem \ref{Thm.Convergence}.

The fact that $\mathcal{L}^{\mathrm{hs};\infty}$ satisfies the conditions of Theorem \ref{Thm.GibbsProperty} follows from Proposition \ref{Prop.Tightness}.
\end{proof}

\begin{proof}[Proof of Theorem \ref{Thm.PfaffianStructure} ]
By Theorem \ref{Thm.Convergence} and Lemma \ref{Lem.FDConvVarpi}, we know that 
$$\left( \hsa_i(s_j): i \geq 1, j \in \llbracket 1, m \rrbracket \right) \overset{f.d.}{=} \left( X^{j,\mathsf{S}}_i: i \geq 1, j \in \llbracket 1, m \rrbracket \right).$$
From \cite[Corollary 2.20]{dimitrov2024airy}, we conclude $M^{\mathsf{S};X} \overset{d}{=}M^{\mathsf{S};\mathrm{hs};\infty}$ as in (\ref{eq:HSA point process}). From Lemma \ref{Lem.FDConvVarpi} we conclude $M^{\mathsf{S};\mathrm{hs};\infty} \overset{d}{=} M^{\mathsf{S}; \infty}$ as in Lemma \ref{Lem.PointProcessConvergenceVarpi}(a). The last distributional equality and Lemma \ref{Lem.PointProcessConvergenceVarpi}(a) imply the statement of the theorem.
\end{proof}

%
%
\subsection{Proof of Theorem \ref{Thm.ConvergenceToAiryLE}}\label{Section5.3} 
Define the line ensembles $\mathcal{L}^n$ by $\sqrt{2}\cdot\mathcal{L}_i^n(t) + t^2 = \mathcal{A}_i^n(t)$, and observe that by definition
\begin{equation}\label{Eq.TranslatedLE}
\mathcal{L}_i^n(t) = \hsli_i(t + t_n) + 2^{-1/2} t_n^2 - 2^{1/2}t_n \cdot t\mbox{ for } i\geq 1, t \in [-t_n, t_n].
\end{equation}

We first check that $\mathcal{L}^n$ satisfy Hypothesis $(H)_{k,T}$ from \cite[Definition 3.3]{CorHamA} for every $k\in\mathbb{N}$ and $T>0$. Indeed, from Theorem \ref{Thm.GibbsProperty}(a) we know that $\hsli$ satisfies the Brownian Gibbs property on $(0, \infty)$. As the latter is preserved under translations and affine shifts, we conclude from (\ref{Eq.TranslatedLE}) that $\mathcal{L}^n$ satisfies the Brownian Gibbs property on $(-t_n, t_n)$, verifying $(H1)_{k,T}$. From Lemma \ref{Lem.FDConvT}, and the continuous mapping theorem, we conclude $\mathcal{L}_i^n \overset{f.d.}{\rightarrow} \mathcal{L}^{\mathrm{Airy}}$ from (\ref{PALE}). As $\mathcal{L}^{\mathrm{Airy}}$ is non-intersecting, we see that $(H2)_{k,T}$ and $(H3)_{k,T}$ are also satisfied.

Now that we verified $(H)_{k,T}$, we can apply \cite[Proposition 3.6]{CorHamA} to conclude $\mathcal{L}^n\Rightarrow \mathcal{L}^{\mathrm{Airy}}$. By the continuous mapping theorem, the latter implies $\mathcal{A}^n \Rightarrow \mathcal{A}$. 

%
%
\subsection{Infinitely many atoms}\label{SectionSpecial} The goal of this section is to establish Lemma \ref{Lem.InfinitelyManyAtoms}. We require the following auxiliary result.

\begin{lemma}\label{Lem.CouplingToInfinity}
Suppose that $\{M^N\}_{N\geq 1}$ is a sequence of point processes on $\mathbb{R}$ that converge weakly to a point process $M^{\infty}$. Assume that for any fixed $a \in 2\mathbb{Z}_{\geq 0}$, the sequence $p_N^a:=\mathbb{P}\left(M^N(\mathbb{R})\geq a \right)$ is decreasing in $N$, and $\mathbb{P}(M^{\infty}(\mathbb{R}) = \infty) = 1$. Then, $\mathbb{P}(M^{N}(\mathbb{R}) = \infty) = 1$ for all $N \in \mathbb{N}$.
\end{lemma}
\begin{proof} The proof follows verbatim from that of \cite[Lemma 5.4]{DY25}, except that $\mathbb{Z}_{\geq 0}$ is replaced everywhere by $2 \mathbb{Z}_{\geq 0}$.
\end{proof}

With the above result in place, we turn to the proof of Lemma \ref{Lem.InfinitelyManyAtoms}.
\begin{proof}[Proof of Lemma \ref{Lem.InfinitelyManyAtoms}] 
The proof is very similar to the proof of \cite[Lemma 6.11]{DY25}.
Fix $t \in (0, \infty)$, and let $s_n = t + n$ for $n \in \mathbb{Z}_{\geq 0}$. We seek to apply Lemma \ref{Lem.CouplingToInfinity} to the sequence $\{M^{s_n;\infty}  \}_{ n \geq 0}$ as in Lemma \ref{Lem.PointProcessConvergenceVarpi}, and for clarity we split the proof into three steps.\\

{\bf \raggedleft Step 1.} We claim that for any fixed $a \in 2\mathbb{Z}_{\geq 0}$ and $n \in \mathbb{Z}_{\geq 0}$, we have 
\begin{equation}\label{eq.UOP1}
p_n^a \geq p_{n+1}^a, \mbox{ where } p_n^a :=   \mathbb{P}(M^{s_n; \infty}(\mathbb{R}) \geq a).
\end{equation}
We will establish \eqref{eq.UOP1} in the steps below. Here, we assume its validity and prove the lemma.  

By Lemma \ref{Lem.PointProcessConvergenceT} (b), as $n\rightarrow\infty$,  $M^{s_n; \infty}$ converges weakly to the Airy point process, which has infinitely many atoms almost surely, see e.g. \cite[Equation (7.11)]{dimitrov2024airy}. 
Combining the above facts with \eqref{eq.UOP1}, we know that the sequence $M^{s_n; \infty}$ satisfy the conditions of Lemma \ref{Lem.CouplingToInfinity}, hence $\mathbb{P}(M^{t; \infty}(\mathbb{R}) = \infty) = \mathbb{P}(M^{s_0; \infty} (\mathbb{R}) = \infty) = 1$ as desired.\\

{\bf \raggedleft Step 2.} In the rest of the proof we fix $a \in 2\mathbb{Z}_{\geq 0}$, $n \in \mathbb{Z}_{\geq 0}$ and show \eqref{eq.UOP1}. We can assume $p^a_{n+1} > 0$ and $a \geq 2$, since $p_n^0 = 1$. We will also fix a sequence $\varpi_N\in\mathbb{R}$ satisfying $\varpi_N\rightarrow\infty$ as $N\rightarrow\infty$.

For each $u,v\in\mathbb{Z}_{\geq0}$ and $x\in\mathbb{R}$, we apply \cite[Lemma 5.3]{DY25} to the sequence $M^{s_v;\mathrm{hs};\varpi_N}$ and conclude  
\begin{equation}\label{eq.UOP2}
\lim_{N \rightarrow \infty} \mathbb{P}\left( M^{s_v;\mathrm{hs};\varpi_N}([x, \infty)) \leq u\right) = \mathbb{P}\left( M^{s_v; \infty}([x, \infty)) \leq u\right) .
\end{equation}
Indeed, condition (1) of the lemma is satisfied by Lemma \ref{Lem.TightFromAbove}, condition (2) by Lemma \ref{Lem.PointProcessConvergenceVarpi} (b) and condition (3) by the fact that for each $x \in \mathbb{R}$, we have 
$$\mathbb{E}\left[M^{s_v;\infty}(\{x\})\right] = \int_{\{x\}} K^{\mathrm{hs};\infty}_{12}(s_v,y;s_v,y) dy = 0.$$

Fix $\varepsilon\in(0,1)$. Using $p^a_{n+1} > 0$ and the monotone convergence theorem, we can find $x_0\in\mathbb{R}$ with 
\begin{equation}\label{eq.UOP3}
\mathbb{P}\left(M^{s_{n+1}; \infty}([x_0,\infty)) \geq a\right) \geq (1- \varepsilon)p^a_{n+1} .
\end{equation}
We then use \eqref{eq.UOP2} and conclude that, for all large enough $N$,
\begin{equation}\label{eq.UOP4}
\mathbb{P}\left(M^{s_{n+1};\mathrm{hs};\varpi_N}([x_0,\infty)) \geq a\right) \geq (1- 2\varepsilon)p^a_{n+1} .
\end{equation}
We now claim that there exist $R \geq 0$ and $\hat{\varpi}\in\mathbb{R}$ depending on $x_0, t, n, a, \varepsilon$, such that for all $\varpi\geq\hat{\varpi}$,  
\begin{equation}\label{eq.UOP5}
\mathbb{P}\left(M^{s_{n};\mathrm{hs};\varpi}([x_0 -R ,\infty)) \geq a \big{\vert} M^{s_{n+1};\mathrm{hs};\varpi}([x_0  ,\infty)) \geq a\right) \geq 1- \varepsilon .
\end{equation}
We will establish \eqref{eq.UOP5} in the next step. Here, we assume its validity and prove \eqref{eq.UOP1}.\\

Combining \eqref{eq.UOP4} and \eqref{eq.UOP5} (taking $\varpi=\varpi_N$) we conclude that for all large enough $N$,
\begin{equation*}
\begin{split}
\mathbb{P}\left(M^{s_{n};\mathrm{hs};\varpi_N}([x_0 -R ,\infty)) \geq a\right)  \geq (1- \varepsilon)  (1- 2\varepsilon)p^a_{n+1}.
\end{split}
\end{equation*}
Taking $N \rightarrow \infty$ above, and using \eqref{eq.UOP2} we conclude that
\begin{equation*}
\begin{split}
&p_n^a = \mathbb{P}\left( M^{s_{n}; \infty}(\mathbb{R}) \geq a\right) \geq \mathbb{P}\left( M^{s_{n}; \infty}([x_0 -R ,\infty)) \geq a\right) \\
&= \lim_{N \rightarrow \infty}  \mathbb{P}\left( M^{s_{n};\mathrm{hs}; \varpi_N}([x_0 -R ,\infty)) \geq a\right) \geq (1- \varepsilon) (1- 2\varepsilon)p^a_{n+1}.
\end{split}
\end{equation*}
Taking $\varepsilon \rightarrow 0$ above, we conclude the proof of \eqref{eq.UOP1}.\\

{\bf \raggedleft Step 3.} In this final step we show \eqref{eq.UOP5}. By the definition of $M^{s_{n};\mathrm{hs};\varpi}$, we only need to prove that for any fixed $t\in(0,\infty)$, $a\in2\mathbb{Z}_{\geq1}$, $n\in\mathbb{Z}_{\geq0}$, $x_0\in\mathbb{R}$, and $\varepsilon\in(0,1)$, there exist $R\geq 0$ and $\hat{\varpi}\in\mathbb{R}$, such that for all $\varpi\geq \hat{\varpi}$, we have
\begin{equation}\label{eq.UOP6}
\mathbb{P}\left( \mathcal{A}_a^{\mathrm{hs};\varpi}(s_n)\geq x_0-R \big{\vert }  \mathcal{A}_a^{\mathrm{hs};\varpi}(s_{n+1})\geq x_0 \right)  \geq 1- \varepsilon.
\end{equation}
In view of (\ref{eq:Parabolic HSA}), it suffices to show that for any $x_1\in\mathbb{R}$, there exist $M\geq 0$ and $\hat{\varpi}\in\mathbb{R}$, such that for all $\varpi\geq\hat{\varpi}$, we have
\begin{equation}\label{eq.UOP7}
\mathbb{P}\left( \mathcal{L}_a^{\mathrm{hs};\varpi}(s_n)\geq x_1-M \big{\vert }  \mathcal{L}_a^{\mathrm{hs};\varpi}(s_{n+1})\geq x_1 \right)  \geq 1- \varepsilon.
\end{equation}
Indeed, using \eqref{eq:Parabolic HSA}, \eqref{eq.UOP7} above is equivalent to (\ref{eq.UOP6}) upon setting $x_0 = \sqrt{2}x_1 + s_{n+1}^2$ and $R = \sqrt{2}M + s_{n+1}^2 - s_n^2$.\\

In the rest of the proof we show \eqref{eq.UOP7}. We choose $M>0$ and $\varpi_0\in\mathbb{R}$ according to Lemma \ref{Lem.NoLowMin}, where we let $k=a/2$ and $b=s_{n+1}$, and we prove \eqref{eq.UOP7} for this choice of $M$ and $\hat{\varpi} := \sqrt{2}\varpi_0$.

For simplicity, we will denote $\mathcal{L}_i^{\mathrm{hs};\varpi}$ as $L_i$ for $i\in\mathbb{N}$.
We also denote $\vec{\mu}\in\mathbb{R}^a$ with $\mu_i=(-1)^i\sqrt{2}\varpi$ for $i\in\llbracket1,a\rrbracket$, $\vec{y} = (y_1, \dots, y_a) = (L_1(s_{n+1}), \dots L_a(s_{n+1}))\in\weyl_a$ and $g(s) = L_{a+1}(s)$ for $s \in [0,s_{n+1}]$. 
Let $E=\{ L_a(s_{n+1})\geq x_1\}$ and 
\[
\mathcal{F}=\sigma\left(\left\{ L_i(s_{n+1})\mbox{ for }i\in\llbracket1,a\rrbracket \mbox{ and }  L_{a+1}(s)\mbox{ for } s\in[0,s_{n+1}]\right\}\right).
\]
We have the following tower of inequalities
\begin{equation}\label{eq.UOP9}
\begin{split}
&\mathbb{P} \left( \left\{L_a(s_n)\geq x_1-M \right\}\cap E  \right) =\mathbb{E} \left[ {\bf 1}_{E } \cdot \mathbb{E}  \left[ {\bf 1}\left\{L_a(s_n)\geq x_1-M \right\} \vert \mathcal{F} \right] \right] \\
& = \mathbb{E} \left[ {\bf 1}_{E } \cdot \mathbb{P}_{\operatorname{avoid}}^{b,\vec{y},\vec{\mu},g}  \left(\mathcal{Q}_{a}(s_n)\geq x_1-M  \right) \right] \\
& \geq  \mathbb{E} \left[ {\bf 1}_{E } \cdot \mathbb{P}_{\operatorname{avoid}}^{b,\vec{y},\vec{\mu},g}  \left(\mathcal{Q}_{a}(s_n)\geq L_a(s_{n+1})-M \right) \right] \\
& = \mathbb{E} \left[ {\bf 1}_{E } \cdot \mathbb{P}_{\operatorname{avoid}}^{b,\vec{y},\vec{\mu},g}  \left(\mathcal{Q}_{a}(s_n)\geq y_a-M   \right) \right] \\
&\geq (1-\varepsilon) \cdot \mathbb{P}(E ).
\end{split}
\end{equation} 
We mention that the first equality follows from the tower property for conditional expectation, and $E\in\mathcal{F}$; the second equality follows from the half-space Brownian Gibbs property with parameters $\vec{\mu}$ for the line ensemble $\mathcal{L}^{\mathrm{hs};\varpi}=\{L_i\}_{i\ge1}$ from Proposition \ref{Prop.HSAiry}; the third inequality follows from the fact that on $E$ we have $ L_a(s_{n+1})\geq x_1$; the fourth equality follows from the definition of $\vec{y}$, in particular $y_a=L_a(s_{n+1})$; and the fifth inequality follows from Lemma \ref{Lem.NoLowMin} with $t=s_n$. 

Equation \eqref{eq.UOP9} implies \eqref{eq.UOP7}, and we conclude the proof of the lemma.
\end{proof}

%
%
\section{Distribution at the origin}\label{Section6} 
The goal of this section is to establish Theorem \ref{Thm.OriginDescription}. The key result we require for this purpose is contained in the following lemma.

\begin{lemma}\label{Lem.OriginMeasure} Define the random measure $M$ on $\mathbb{R}$ through
\begin{equation}\label{Eq.OriginMeasure}
M(A) = \sum_{i \geq 1}{\bf 1}\{\hsai_i(0) \in A\}.
\end{equation}
Then, $M \overset{d}{=} 2 M^{\mathrm{GSE}}$, where $M^{\mathrm{GSE}}$ is defined in Lemma \ref{Lem.SAO}.
\end{lemma}

We prove Lemma \ref{Lem.OriginMeasure} in Section \ref{Section6.3}. We now use it to quickly deduce Theorem \ref{Thm.OriginDescription}.
\begin{proof}[Proof of Theorem \ref{Thm.OriginDescription}] By Lemma \ref{Lem.OriginMeasure} we have for any $k_1, \dots, k_n \in \mathbb{N}$ and $a_1, \dots, a_n \in \mathbb{R}$ that
\begin{equation}\label{Eq.CDF1}
\mathbb{P}\left( \cap_{i = 1}^n \{M(a_i, \infty) < k_i\} \right) =  \mathbb{P}\left( \cap_{i = 1}^n \{2M^{\mathrm{GSE}}(a_i, \infty) < k_i\} \right).
\end{equation}
If we now set $(X_1, X_2, \dots) = (-2^{2/3}\Lambda_0, -2^{2/3}\Lambda_0, -2^{2/3}\Lambda_1, -2^{2/3}\Lambda_1, \dots)$ as in Lemma \ref{Lem.SAO}, then we have $X_1 \geq X_2 \geq \cdots$, and 
$$2M^{\mathrm{GSE}}(A) = \sum_{i \geq 1}{\bf 1}\{X_i \in A\}.$$
Combining the latter with $\hsai_1(0) \geq \hsai_2(0) \geq \cdots$ (by Theorem \ref{Thm.Convergence}), we obtain that (\ref{Eq.CDF1}) is equivalent to
\begin{equation*}
\mathbb{P}\left( \cap_{i = 1}^n \{ \hsai_{k_i}(0) \leq a_i \} \right) =  \mathbb{P}\left( \cap_{i = 1}^n \{X_{k_i} \leq a_i\} \right),
\end{equation*}
which implies the statement of the theorem.
\end{proof}

Let $t_N \rightarrow 0+$ be any sequence, and define the random measures $M^N$ on $\mathbb{R}$ by
\begin{equation}\label{Eq.FormedHSA}
M^N(A) = \sum_{i \geq 1} {\bf 1}\{\hsai_i(t_N) \in A \}.
\end{equation}
The way we establish Lemma \ref{Lem.OriginMeasure} is by showing that 
\begin{enumerate}
  \item $M^N \Rightarrow 2M^{\mathrm{GSE}}$,
  \item $M^N \rightarrow M$ almost surely.
\end{enumerate}
As we will see in Section \ref{Section6.3}, the second point follows directly from the first and the continuity of $\hsai$ already established in Theorem \ref{Thm.Convergence}. Hence, the first point requires most of the effort.

From Theorem \ref{Thm.PfaffianStructure} and \cite[Lemma 5.13]{DY25}, we know that the random measures $M^N$ are Pfaffian point processes on $\mathbb{R}$ with reference measure $\mathrm{Leb}$ and correlation kernel given by $\hat{K}^N(x,y) = \hski(t_N,x; t_N,y)$, where $\hski$ is as in (\ref{Eq.DefK}). When $t_N \in (0,1/2]$, we can write 
\begin{equation}\label{Eq.SplitKernel}
\hat{K}^N(x,y) = K^N(x,y) + \Delta^N(x,y) ,
\end{equation}
where
\begin{equation}\label{Eq.DefKN}
\begin{split}
K^N_{11}(x,y) = &\frac{1}{(2\pi \im)^2} \int_{\mathcal{C}_{1}^{\pi/3}}dz \int_{\mathcal{C}_{1}^{\pi/3}} dw \frac{(z - w)H(z,x;w,y)}{4(z + w + 2t_N)(z+t_N)(w+ t_N)}, \\
K^N_{12}(x,y) = -K^N_{21}(y,x) = &\frac{1}{(2\pi \im)^2} \int_{\mathcal{C}_{1}^{\pi/3}}dz \int_{\mathcal{C}_{1}^{\pi/3}} dw \frac{(z - w) H(z,x;w,y)}{2(z+ t_N)(z + w)}, \\
K^N_{22}(x,y) = &\frac{1}{(2\pi \im)^2} \int_{\mathcal{C}_{1}^{\pi/3}}dz \int_{\mathcal{C}_{1}^{\pi/3}} dw \frac{(z - w )H(z,x;w,y)}{z + w - 2t_N},
\end{split}
\end{equation} 
and 
\begin{equation}\label{Eq.DefDeltaN}
\Delta^N(x,y) = \begin{bmatrix} 0 & 0 \\ 0 & \delta_N(x,y) \end{bmatrix} \mbox{ with } \delta_N(x,y) = e^{2t_N^3/3 - t_N(x+y)} \cdot \frac{ (y - x) }{2\pi^{1/2} (2t_N)^{3/2} } \cdot \exp \left(-\frac{(y - x)^2}{8t_N}\right).
\end{equation}
We mention that (\ref{Eq.SplitKernel}) is obtained directly from (\ref{Eq.DefK}) upon setting $s = t = t_N$ and deforming the contours in the definition of $\hsii_{ij}$ to $\mathcal{C}_1$. The reason we do not cross any poles is due to our assumption $t_N \in (0,1/2]$.\\
 
The standard approach to proving weak convergence of Pfaffian point processes is to establish pointwise convergence of their kernels. This approach cannot work here, as it would imply that the limiting point process is itself Pfaffian, and hence simple, which is not the case for $2M^{\mathrm{GSE}}$. At the level of formulas, this approach fails due to the singular nature of $\delta_N$, and so we need to modify it. 

The way we prove $M^N \Rightarrow 2M^{\mathrm{GSE}}$ is by showing that the {\em joint factorial moments} of $M^N$ converge to those of $2M^{\mathrm{GSE}}$. These can be expressed as finite sums of integrals involving products of $K_{ij}^N$ and $\delta_N$. By integrating $\delta_N$ against smooth functions rather than considering $\delta_N$ itself, we obtain smoother expressions that do converge.\\

The rest of the section is organized as follows. In Section \ref{Section6.1} we compute the limit of a sequence of integrals involving the product of several $\delta_N$'s with smooth functions, see Lemma \ref{Lem.LimitDerivatives}. Section \ref{Section6.2} uses this result to show that the factorial moments of $M^N$ converge to those of $2M^{\mathrm{GSE}}$. Finally, Section \ref{Section6.3} contains the proof of Lemma \ref{Lem.OriginMeasure}.

%
%
\subsection{The asymptotic action of $\delta_N$}\label{Section6.1} We begin by introducing some useful definitions and notation. Within this section we use the following {\em multi-index notation}:
\begin{itemize}
\item A multi-index is $\alpha = (\alpha_1, \dots, \alpha_n) \in \mathbb{Z}_{\geq 0}$,
\item $|\alpha| = \alpha_1 + \cdots + \alpha_n$,
\item $\alpha! = \alpha_1! \cdots \alpha_n!$,
\item for $x = (x_1, \dots, x_n)$ and $a = (a_1, \dots, a_n)$, we write $(x-a)^{\alpha} = (x_1-a_1)^{\alpha_1} \cdots (x_n - a_n)^{\alpha_n}$,
\item $D^{\alpha}f(a) = \frac{\partial^{|\alpha|}f}{\partial_{x_1}^{\alpha_1} \cdots \partial_{x_n}^{\alpha_n}} (a)$,
\item $C^{k}(\mathbb{R}^n;\mathbb{C})$ is the set of $k$-times continuously differentiable functions $f: \mathbb{R}^n \rightarrow \mathbb{C}$,
\item for a Borel set $A \subseteq \mathbb{R}^n$, we write $\int_Ady$ to mean $\int_A dy_1dy_2 \cdots dy_n$,
\item if $J_1, \dots, J_m$ are pairwise disjoint finite index sets with $J_k = \{j^k_1, \dots, j^k_{r_k}\}$ and $A_1, \dots, A_m \subseteq \mathbb{R}$ are Borel, we write $\int_{A_1^{J_1} \times \cdots \times A_m^{J_m}}dy$ to mean 
$$\int_{A_1} dy_{j^1_1} \int_{A_1} dy_{j^1_2} \cdots \int_{A_1} dy_{j^1_{r_1}} \int_{A_2} dy_{j^2_1} \int_{A_2} dy_{j^2_2} \cdots \int_{A_2} dy_{j^2_{r_2}} \cdots \int_{A_m} dy_{j^m_1} \int_{A_m} dy_{j^m_2} \cdots \int_{A_m} dy_{j^m_{r_m}} .$$  
\end{itemize}
We also summarize some notation in the following definition.
\begin{definition}\label{Def.LimitDerivatives} Fix $m \in \mathbb{N}$, and integers $n_1, \dots, n_m \in \mathbb{N}$. Define the sets
$$I_1 = \{1, \dots, n_1\} \mbox{ and } I_{i} = \{n_1 + \cdots + n_{i-1} + 1, \dots, n_{1} + \cdots + n_i\} \mbox{ for $i = 2, \dots, m$}.$$
We further set $\vec{n} = (n_1, \dots, n_m)$ and $n = n_1 + \cdots + n_m$. Suppose $r \in \mathbb{N}$ with $n \geq 2r$. For any pairwise distinct integers $u_1, \dots, u_r, v_1, \dots, v_r \in \{1, \dots, n\}$, we let $\#\mathrm{tra}$ denote the number of pairs of indices $1\leq i < j \leq 2r$, such that $w_i > w_j$ where $w$ is the word
\begin{equation}\label{Eq.Word}
 w = (w_1, \dots, w_{2r}) = (u_1,v_1, u_2, v_2, \dots, u_r,v_r).
 \end{equation}
We also define the {\em sign} of $w$, written $(-1)^w$, by setting it to $(-1)^{\#\mathrm{tra}}$. 

We say that $u_1, \dots, u_r, v_1, \dots, v_r$ are {\em $\vec{n}$-compatible} if we can find pairwise disjoint index sets $J_1, \dots, J_{m} \subseteq \{1, \dots, n\}$, such that
$$|J_i| = r_i, \hspace{2mm}  u_j, v_j \in I_i \mbox{ for } i = 1, \dots, m \mbox{ and } j \in J_i, \mbox{ and } r_1 + \cdots + r_m = r.$$
Notice that if $u_1, \dots, u_r, v_1, \dots, v_r$ are $\vec{n}$-compatible, then $n_i \geq 2r_i$ for all $i = 1, \dots, m$. In addition, we note that $\vec{r} = (r_1, \dots, r_m)$ is uniquely determined by $u_1, \dots, u_r, v_1, \dots, v_r$, and we refer to it as their {\em profile}. We also define $I^u_k = \{u_j: j \in J_k\}$, $I^{v}_k = \{v_j: j \in J_k\}$ and $I^{vc}_k = I_k \setminus I^{v}_k$ for $k = 1, \dots, m$.  
\end{definition}

With the above notation in place, we now turn to the main result of this section.
\begin{lemma}\label{Lem.LimitDerivatives} Assume the same notation as in Definition \ref{Def.LimitDerivatives}. Suppose $t_N \in (0,1/2]$ is a sequence such that $\lim_{N \rightarrow \infty} t_N = 0$ and let  $\delta_N$ be as in (\ref{Eq.DefDeltaN}).
Suppose that $f_N \in C^{r+1}(\mathbb{R}^{n}; \mathbb{C})$ is a sequence such that for some $f_{\infty} \in C^{r+1}(\mathbb{R}^{n}; \mathbb{C})$
\begin{equation}\label{Eq.LimitSmoothFunc}
\lim_{N \rightarrow \infty} D^{\alpha}f_N(x) = D^{\alpha}f_{\infty}(x),
\end{equation}
for all $x \in \mathbb{R}^{n}$ and $|\alpha| \leq r+1$. Suppose in addition that we have $m$ pairwise disjoint intervals $[a_i, b_i] \subset \mathbb{R}$ and set $A = \min_{i = 1, \dots, m} a_i$, $B = \max_{i = 1, \dots, m} b_i$. We assume that there is a constant $C(n,r, A, B) > 0$, such that for all $x\in [A,B]^{n}$, $|\alpha| \leq r+1$ and $N \in \mathbb{N}$
\begin{equation}\label{Eq.SmoothFuncBound}
\left| D^{\alpha} f_N(x) \right| \leq C(n,r,A,B).
\end{equation}
If $u_1, \dots, u_r, v_1, \dots, v_r$ are not $\vec{n}$-compatible, then     
\begin{equation}\label{Eq.SmoothFuncVanish}
 \int_{[a_1, b_1]^{n_1} \times \cdots \times [a_m, b_m]^{n_m}}dx \left| \delta_N(x_{u_1}, x_{v_1}) \cdots \delta_N(x_{u_r}, x_{v_r}) f_N(x) \right| \leq C_1e^{-c_1/t_N},
\end{equation}
for some positive $C_1, c_1 > 0$, depending on $n, r, m, \{a_i\}_{i = 1}^m, \{b_i\}_{i = 1}^m$ and $C(n,r,A,B)$.

If $u_1, \dots, u_r, v_1, \dots, v_r$ are $\vec{n}$-compatible, then     
\begin{equation}\label{Eq.IntLim}
\begin{split}
&\lim_{N \rightarrow \infty} \int_{[a_1, b_1]^{n_1} \times \cdots \times [a_m, b_m]^{n_m}}dx \delta_N(x_{u_1}, x_{v_1}) \cdots \delta_N(x_{u_r}, x_{v_r}) f_N(x)  \\
& = \int_{[a_{1}, b_{1}]^{I_1^{vc}} \times \cdots  \times [a_m, b_m]^{I_m^{vc}}} dy \left[(\partial_{x_{v_1}} - \partial_{x_{u_1}}) \cdots (\partial_{x_{v_r}} - \partial_{x_{u_r}})  f_\infty \right](\bar{y}) ,
\end{split}
\end{equation}
where $\bar{y} = (\bar{y}_1, \dots, \bar{y}_{n}) \in \mathbb{R}^{n}$ is obtained from $y$ by setting $\bar{y}_{j} = y_j$ for $j \not \in \{v_1, \dots, v_r\}$ and $\bar{y}_{v_i} = y_{u_i}$ for $i = 1, \dots, r$.
\end{lemma}
\begin{proof} Throughout the proof we denote
$$A_N:=  \int_{[a_1, b_1]^{n_1} \times \cdots \times [a_m, b_m]^{n_m}}dx \delta_N(x_{u_1}, x_{v_1}) \cdots \delta_N(x_{u_r}, x_{v_r}) f_N(x).$$
For clarity, we split the proof into four steps. In the first step we establish (\ref{Eq.SmoothFuncVanish}). In the second step we partially symmetrize $A_N$ by swapping $u_i$'s with $v_i$'s for $i = 1, \dots, r$ and express $A_N = B_N + \mathrm{Err}_N$, see (\ref{Eq.TaylorExpandedFunc}). We further reduce the statement to showing that $B_N$ converges, see (\ref{Eq.BNConv}), and that $\mathrm{Err}_N$ decays with $t_N$, see (\ref{Eq.ErrNDecay}). Equation (\ref{Eq.BNConv}) is established in the third step, and equation (\ref{Eq.ErrNDecay}) is established in the fourth step. Throughout the proof all constants depend on $n, r, m, \{a_i\}_{i = 1}^m, \{b_i\}_{i = 1}^m$ and $C(n,r,A,B)$ -- we do not mention this further.\\

{\bf \raggedleft Step 1.} Suppose that $u_1, \dots, u_r, v_1, \dots, v_r$ are not $\vec{n}$-compatible. Then, for some $1 \leq i\neq j \leq m$ and $k \in \{1, \dots, r\}$ we have $u_k \in I_i$ and $v_k \in I_j$. As $\{[a_i,b_i]\}_{i = 1}^m$ are pairwise disjoint, we can find $\epsilon > 0$, such that 
\begin{equation*}
|x-y| \geq \epsilon \mbox{ for } x \in [a_i,b_i] \mbox{ and } y \in [a_j, b_j].
\end{equation*}
From the definition of $\delta_N$ in (\ref{Eq.DefDeltaN}) we can find $C_0 > 0$, such that for $N \geq 1$
\begin{equation}\label{Eq.DeltaBound}
\begin{split}
&|\delta_N(x,y)| \leq C_0 t_N^{-3/2} \mbox{ if $x,y \in [A,B]$, and } \\
&|\delta_N(x,y)| \leq C_0t_N^{-3/2} e^{- \epsilon^2/ 8 t_N} \mbox{ if } x \in [a_i,b_i], y \in [a_j,b_j].  
\end{split}
\end{equation}
The latter and (\ref{Eq.SmoothFuncBound}) show that for all $x \in [a_1,b_1]^{n_1} \times \cdots \times [a_m, b_m]^{n_m}$ and $N \geq 1$ we have 
$$\left|\delta_N(x_{u_1}, x_{v_1}) \cdots \delta_N(x_{u_r}, x_{v_r}) f_N(x)\right| \leq C_0^r t_N^{-3r/2} \cdot C(n,r,A,B) \cdot e^{- \epsilon^2/ 8t_N},$$
from which (\ref{Eq.SmoothFuncVanish}) immediately follows.\\

{\bf \raggedleft Step 2.} In the remainder we assume that $u_1, \dots, u_r, v_1, \dots, v_r$ are $\vec{n}$-compatible. For a permutation $\sigma \in S_{n}$ and a function $f$ on $\mathbb{R}^{n}$, we denote 
$$f^{\sigma}(x) = f^{\sigma}(x_1, \dots, x_{n}) = f(x_{\sigma(1)}, \dots, x_{\sigma(n)}).$$
Using that $\delta_N(x,y) = - \delta_N(y,x)$, we get upon splitting the integral over sets where $\{x_{u_i} \leq x_{v_i}\}$ or $\{x_{u_i} \geq x_{v_i}\}$ and changing variables that 
\begin{equation}\label{Eq.Symmetrize}
\begin{split}
&A_N= \sum_{\sigma \in S_{2}^r}(-1)^{\sigma} \int_{[a_1, b_1]^{n_1} \times \cdots \times [a_m, b_m]^{n_m}}dx  \prod_{i = 1}^r {\bf 1}\{x_{u_i} \leq x_{v_i}\} \delta_N(x_{u_i},x_{v_i}) \cdot f^{\sigma}_N(x),
\end{split}
\end{equation}
where $S_2^r$ is the subgroup of the permutation group $S_{n}$, generated by the transpositions $\{(u_i,v_i): i = 1, \dots ,r\}$.

From Taylor's formula, see e.g. \cite[Theorem 3.18]{callahan2010advanced}, we have for all $f \in C^{r+1}(\mathbb{R}^n; \mathbb{C})$ and $x,a \in \mathbb{R}^n$
\begin{equation}\label{Eq.TaylorExpansion} 
f(x) = \sum_{ |\alpha| \leq r} \frac{D^{\alpha}f(a) (x-a)^{\alpha}}{\alpha!} + R_{r}(f;a,x),
\end{equation}
where the remainder term $R_{r}(f;a,x)$ is given by
\begin{equation}\label{Eq.TaylorRemainder}
R_{r}(f;a,x)= \sum_{|\alpha| = r+1} \int_0^1 dt\frac{D^{\alpha}f(a + t (x-a)) (x-a)^{\alpha}}{\alpha!} (1-t)^{r}.
\end{equation}
Substituting (\ref{Eq.TaylorExpansion}) for $f = f_N^{\sigma}$ and $a = (a_1, \dots, a_n)$, such that
$$a_{v_i} = x_{u_i} \mbox{ for } i = 1, \dots, r \mbox{ and } a_{j} = x_j \mbox{ for } j \not \in \{v_1, \dots, v_r\},$$
into (\ref{Eq.Symmetrize}), we obtain 
\begin{equation}\label{Eq.TaylorExpandedFunc}
\begin{split}
&A_N =B_N + \mathrm{Err}_N,
\end{split}
\end{equation}
where 
$$B_N = \sum_{ |\alpha| \leq r} \int_{[a_1, b_1]^{n_1} \times \cdots \times [a_m, b_m]^{n_m}}dx  \prod_{i = 1}^r {\bf 1}\{x_{u_i} \leq x_{v_i}\} \delta_N(x_{u_i},x_{v_i}) \sum_{\sigma \in S_{2}^r}(-1)^{\sigma}  \frac{[D^{\alpha}f^{\sigma}_N](a) (x-a)^{\alpha}}{\alpha!},$$
$$\mathrm{Err}_N = \int_{[a_1, b_1]^{n_1} \times \cdots \times [a_m, b_m]^{n_m}}dx  \prod_{i = 1}^r {\bf 1}\{x_{u_i} \leq x_{v_i}\} \delta_N(x_{u_i},x_{v_i}) \sum_{\sigma \in S_{2}^r}(-1)^{\sigma} R_{r}(f^{\sigma}_N;a,x).$$

We claim that 
\begin{equation}\label{Eq.BNConv}
\begin{split}
&\lim_{N \rightarrow \infty} B_N = \int_{[a_{1}, b_{1}]^{I_1^{vc}} \times \cdots  \times [a_m, b_m]^{I_m^{vc}}} dy \sum_{\sigma \in S_{2}^r}(-1)^{\sigma} [\partial_{x_{v_1}} \cdots \partial_{x_{v_r}} f_{\infty}^{\sigma}](\bar{y}),
\end{split}
\end{equation}
where $\bar{y}$ is as below (\ref{Eq.IntLim}).
In addition, we claim that for some $\hat{C}_1 > 0$ and all $N \geq 1$
\begin{equation}\label{Eq.ErrNDecay}
\begin{split}
&|\mathrm{Err}_N| \leq \hat{C}_1 t_N^{1/2}.
\end{split}
\end{equation}
Using that for any $f\in C^{r}(\mathbb{R}^{n};\mathbb{C})$
\begin{equation}\label{Eq.SymmDeriv}
 \sum_{\sigma \in S_{2}^r}(-1)^{\sigma} [\partial_{x_{v_1}} \cdots \partial_{x_{v_r}} f^{\sigma}](\bar{y}) =  \left[(\partial_{x_{v_1}} - \partial_{x_{u_1}}) \cdots (\partial_{x_{v_r}} - \partial_{x_{u_r}})  f\right](\bar{y}),
 \end{equation}
we see that (\ref{Eq.TaylorExpandedFunc}), (\ref{Eq.BNConv}) and (\ref{Eq.ErrNDecay}) imply (\ref{Eq.IntLim}). Consequently, we have reduced the proof to showing (\ref{Eq.BNConv}) and (\ref{Eq.ErrNDecay}).\\

{\bf \raggedleft Step 3.} In this step we prove (\ref{Eq.BNConv}). We observe that for any $f \in C^r(\mathbb{R}^{n};\mathbb{C})$
\begin{equation}\label{Eq.SymmSum}
\begin{split}
& \sum_{ |\alpha| \leq r} \sum_{\sigma \in S_{2}^r}(-1)^{\sigma} \frac{[D^{\alpha}f^{\sigma}](a) (x-a)^{\alpha}}{\alpha!} =  \\
& = \sum_{\sigma \in S_{2}^r}(-1)^{\sigma} [\partial_{x_{v_1}} \partial_{x_{v_2}} \cdots \partial_{x_{v_{r}}} f^{\sigma}](a)(x_{v_1} - x_{u_1}) \cdots (x_{v_r} - x_{u_r}).
\end{split}
\end{equation}
Indeed, since $a_{j} = x_j$ for $j \not \in\{v_1, \dots, v_r\}$, we see that only terms with $\alpha_{j} = 0$ for $j \not \in\{v_1, \dots, v_r\}$ contribute to the first line in (\ref{Eq.SymmSum}). In addition, if $\alpha_{u_i} = \alpha_{v_i} = 0$ for some $i$, and $\tau_i = (u_i, v_i)$, we have $[D^{\alpha}f](a) = [D^{\alpha}f^{\tau_i}](a)$ since $a_{u_i} = a_{v_i}$. Combining this with $(f^{\sigma})^{\tau_i} = f^{\tau_i \sigma}$ shows that each summand 
$$\frac{[D^{\alpha}f_N^{\sigma}](y) (x-a)^{\alpha}}{\alpha!}$$
with $\alpha_{u_i} = \alpha_{v_i} = 0$ can be paired up with the analogous one with $\sigma$ replaced with $\tau_i \sigma$ and they cancel as they come with opposite signs. The above discussion shows that the only non-zero contribution comes from $\alpha$ such that $\alpha_{j} = 0$ for $j \not \in\{v_1, \dots, v_r\}$, and $\alpha_{v_i} \geq 1$ for $i = 1, \dots, r$. Since we have $|\alpha| \leq r$, we see that the only term that contributes is $\alpha_{j} = 0$ for $j \not \in\{v_1, \dots, v_r\}$ and $\alpha_{v_i} = 1$ for $i = 1, \dots, r$, implying (\ref{Eq.SymmSum}).

We substitute (\ref{Eq.SymmSum}) into the definition of $B_N$, as well as $\delta_N(x_{u_i},x_{v_i})$ from (\ref{Eq.DefDeltaN}) and change variables $x_{j} = y_j$ for $j \not \in\{v_1, \dots, v_r\}$, and  $x_{v_i} = y_{u_i} + t_N^{1/2}z_i$ to get
\begin{equation}\label{Eq.NewBN}
\begin{split}
&B_N =  \int_{[a_{1}, b_{1}]^{I_1^{vc}} \times \cdots  \times [a_m, b_m]^{I_m^{vc}}} dy \int_{[0,\infty)^r}dz \prod_{i = 1}^r{\bf 1}\{z_i \leq t_N^{-1/2} (b_{s(i)} - y_{u_i}) \}  \\
&\times \prod_{i = 1}^r e^{2t_N^3/3 - t_N(2y_{u_i}+t_N^{1/2}z_i)} \frac{ z_i^2e^{-z_i^2/8} }{4(2\pi)^{1/2}  }  \cdot \sum_{\sigma \in S_{2}^r}(-1)^{\sigma} [\partial_{x_{v_1}} \cdots \partial_{x_{v_r}} f_N^{\sigma}](\bar{y}),
\end{split}
\end{equation}
where $s(i)$ is such that $u_i \in I_{s(i)}$. Denoting the integrand in (\ref{Eq.NewBN}) by $g_N(y,z)$, we see from (\ref{Eq.LimitSmoothFunc}) that for a.e. $(y,z) \in [a_{1}, b_{1}]^{I_1^{vc}} \times \cdots  \times [a_m, b_m]^{I_m^{vc}} \times [0,\infty)^r$ 
\begin{equation}\label{Eq.PointwiseConv}
\lim_{N \rightarrow \infty}g_N(y,z) = \prod_{i =1}^r \frac{z_i^2e^{-z_i^2/8} }{4(2\pi)^{1/2}  }  \cdot \sum_{\sigma \in S_{2}^r}(-1)^{\sigma} [\partial_{x_{v_1}} \cdots \partial_{x_{v_r}} f_{\infty}^{\sigma}](\bar{y}).
\end{equation}
In addition, from (\ref{Eq.SmoothFuncBound}) and the fact that $t_N \in (0,1/2]$, $|y_i| \leq |A|+ |B|$, we get for all $N \geq 1$
\begin{equation}\label{Eq.PointwiseBound}
|g_N(y,z)| \leq \prod_{i =1}^r \frac{z_i^2e^{1+ |A|+|B| + |z_i| -z_i^2/8} }{4(2\pi)^{1/2}  }  \cdot 2^r C(n,r,A,B).
\end{equation}
By the dominated convergence theorem we conclude 
\begin{equation*}
\begin{split}
&\lim_{N \rightarrow \infty} B_N = \int_{[a_{1}, b_{1}]^{I_1^{vc}} \times \cdots  \times [a_m, b_m]^{I_m^{vc}}}dy \int_{[0,\infty)^r}dz \prod_{i = 1}^{r} \frac{z_i^2e^{-z_i^2/8} }{4(2\pi)^{1/2}  } \sum_{\sigma \in S_{2}^r}(-1)^{\sigma} [\partial_{x_{v_1}} \cdots \partial_{x_{v_r}} f_{\infty}^{\sigma}](\bar{y}).
\end{split}
\end{equation*}
Combining the latter with
$$\int_{0}^{\infty} \frac{z^2 e^{-z^2/8}}{4(2\pi)^{1/2}} dz =\frac{1}{4} \int_{-\infty}^{\infty} \frac{z^2 e^{-z^2/8}}{2(2\pi)^{1/2}}dz  = 1,$$
gives (\ref{Eq.BNConv}). \\

{\bf \raggedleft Step 4.} In this step we prove (\ref{Eq.ErrNDecay}). From (\ref{Eq.SmoothFuncBound}) and (\ref{Eq.TaylorRemainder}), we have for some $C_3 > 0$ and all $x \in [A,B]^{n}$, $\sigma \in S_2^r$, and $N \geq 1$ that
$$\left|R_{r}(f_N^{\sigma}; a,x)\right| \leq C_3 \left( \sum_{i = 1}^r|x_{u_i} - x_{v_i}| \right)^{r+1}.$$
We mention that in deriving the last inequality we used that the non-zero terms in (\ref{Eq.TaylorRemainder}) must have $\alpha_j = 0$ for $j \not \in \{v_1, \dots, v_r\}$ as $a_j = x_j$ by definition.

Applying the same change of variables as above (\ref{Eq.NewBN}) to $\mathrm{Err}_N$, substituting $\delta_N(x_{u_i}, x_{v_i})$ from (\ref{Eq.DefDeltaN}), and using the last remainder upper-bound gives
\begin{equation*}
\begin{split}
&|\mathrm{Err}_N| \leq t_N^{1/2} \int_{[a_{1}, b_{1}]^{I_1^{vc}} \times \cdots  \times [a_m, b_m]^{I_m^{vc}}}\hspace{-1mm}dy \int_{[0,\infty)^r} \hspace{-2mm} dz \prod_{i = 1}^r  e^{2t_N^3/3 - t_N(2y_i+t_N^{1/2}z_i)} \frac{z_ie^{-z_i^2/8} }{4(2\pi)^{1/2}  }  \cdot 2^r C_3 \left( \sum_{i = 1}^rz_i \right)^{r+1}.
 \end{split}
\end{equation*}
The last displayed equation implies (\ref{Eq.ErrNDecay}) once we use that $|y_i| \leq |A|+|B|$, $t_N \in (0,1/2]$.
\end{proof}

%
%
\subsection{Convergence of factorial moments}\label{Section6.2} We start by recalling some basic definitions and notation regarding Pfaffians, following \cite{S90}.

For $n \in \mathbb{N}$, we let $\mathrm{Mat}_n(\mathbb{C})$ denote the set of $n \times n$ matrices with complex entries. We also let $\mathrm{Skew}_{2n}(\mathbb{C}) \subset \mathrm{Mat}_{2n}(\mathbb{C})$ denote the set of skew-symmetric $2n \times 2n$ matrices. If $J \subseteq \{1, \dots, n\}$ and $M \in \mathrm{Mat}_{n}(\mathbb{C})$, we let $M_J$ denote the restriction of $M$ to rows and columns indexed by $J$.

If $A \in \mathrm{Skew}_{2n}(\mathbb{C})$, we define its {\em Pfaffian} by
\begin{equation}\label{Eq.DefPfaffian}
\mathrm{Pf}[A] =\frac{1}{2^{n}n!}\sum_{\sigma \in S_{2n}}\mathrm{sgn}(\sigma)A_{\sigma(1)\sigma(2)}A_{\sigma(3)\sigma(4)}\cdots A_{\sigma(2n-1)\sigma(2n)}.
\end{equation}
One also has the following alternative formula for the Pfaffian
\begin{equation}\label{Eq.DefPfaffian2}
\mathrm{Pf}[A] = \sum_{w = (i_1, j_1, \dots, i_n, j_n)} (-1)^{w}A_{i_1j_1}A_{i_2j_2}\cdots A_{i_nj_n},
\end{equation}
where the sum is over words $w = (i_1,j_1, i_2, j_2, \dots, i_n, j_n)$, such that $w$ is a permutation of $\{1, \dots, 2n\}$, $i_1 < i_2 < \cdots < i_n$, and $i_k < j_k$ for $k = 1, \dots, n$. The sign $(-1)^w$ of $w$ is the usual one for permutations, or equivalently the one in Definition \ref{Def.LimitDerivatives}.

If $K:\mathbb{R}^2 \rightarrow \mathrm{Skew}_{2}(\mathbb{C})$, we write 
$$K(x,y) = \begin{bmatrix} K_{11}(x,y) & K_{12}(x,y) \\ K_{21}(x,y) & K_{22}(x,y) \end{bmatrix},$$
and for $(x_1, \dots, x_n) \in \mathbb{R}^n$, we write $\mathrm{Pf}[K(x_i,x_j)]_{i,j = 1}^n$ for the Pfaffian of the $2n\times 2n$ skew-symmetric matrix formed by the $2\times 2$ blocks $K(x_i,x_j)$ for $1 \leq i,j \leq n$.

The next lemma gives a formula for the Pfaffian of the sum of two matrices.
\begin{lemma}\label{Lem.SumPfaffian} \cite[Lemma 4.2]{S90} For $A, B \in \mathrm{Skew}_{2n}(\mathbb{C})$ we have 
\begin{equation}\label{Eq.PfaffianSum}
\mathrm{Pf}[A+B] = \sum_{I \subseteq \{1, \dots, 2n\}: |I| \text{\hspace{0.5mm}\scriptsize even}} (-1)^{\Sigma(I)- |I|/2} \mathrm{Pf}[A_I] \mathrm{Pf}[B_{I^c}],
\end{equation}
where $\Sigma(I) = \sum_{i \in I}i$, $I^c = \{1, \dots, 2n\} \setminus I$, and $\mathrm{Pf}[A_{\emptyset}] = \mathrm{Pf}[B_{\emptyset}] = 1$.
\end{lemma}

With the above notation in place, we can turn to the main result of this section.

\begin{lemma}\label{Lem.FactMomentConvergence} Fix $m \in \mathbb{N}$, $n_1, \dots, n_m \in \mathbb{N}$ and $m$ pairwise disjoint intervals $[a_i, b_i] \subset \mathbb{R}$. If $M^{\mathrm{GSE}}, K^{\mathrm{GSE}}$ are as in Lemma \ref{Lem.SAO}, then
\begin{equation}\label{Eq.FactorialMomentsDoubledFormula}
\begin{split}
&\mathbb{E} \left[ \prod_{i = 1}^m\frac{\left(2M^{\mathrm{GSE}}([a_i,b_i])\right)!}{\left(2M^{\mathrm{GSE}}([a_i,b_i]) - n_i\right)!}\right] = \sum_{k_1 = \lceil n_1/2\rceil}^{n_1} \cdots \sum_{k_m = \lceil n_m/2\rceil}^{n_m} B(k_1, \dots, k_m), \mbox{ where } \\
&B(k_1, \dots, k_m) = \prod_{i = 1}^m \frac{n_i! 2^{2k_i - n_i}}{(n_i-k_i)!(2k_i-n_i)!} \cdot \int_{[a_1,b_1]^{k_1} \times \cdots \times [a_m,b_m]^{k_m}} \mathrm{Pf}\left[\kgse(x_i,x_j) \right]_{i,j = 1}^kdx,
\end{split}
\end{equation}
and we have set $k = k_1 + \cdots + k_m$. If $t_N \rightarrow 0+$ and $M^N$ are as in (\ref{Eq.FormedHSA}), then
\begin{equation}\label{Eq.FactorialMomentsConvergence}
\begin{split}
&\lim_{N \rightarrow \infty}\mathbb{E} \left[ \prod_{i = 1}^m\frac{\left(M^{N}([a_i,b_i])\right)!}{\left(M^{N}([a_i,b_i]) - n_i\right)!}\right] = \mathbb{E} \left[ \prod_{i = 1}^m\frac{\left(2M^{\mathrm{GSE}}([a_i,b_i])\right)!}{\left(2M^{\mathrm{GSE}}([a_i,b_i]) - n_i\right)!}\right].
\end{split}
\end{equation}
\end{lemma}
\begin{proof} For clarity, we split the proof into five steps. In the first step we prove (\ref{Eq.FactorialMomentsDoubledFormula}). In the second step we analyze the pointwise limit of $K^N$ from (\ref{Eq.DefKN}) and its derivatives. In the same step we relate the limiting kernel $K^{\infty}$ to $\kgse$, see (\ref{Eq.LimitKernelToGSE}). In the third step we express the left side of (\ref{Eq.FactorialMomentsConvergence}) for finite $N$ as a sum of integrals, see (\ref{Eq.OriginMasterFactDecomposition}). The expansion in (\ref{Eq.OriginMasterFactDecomposition}) is based on the joint factorial moment formula for Pfaffian processes from \cite[(B.2)]{DY25} and Lemma \ref{Lem.SumPfaffian}. In addition, each summand in (\ref{Eq.OriginMasterFactDecomposition}) is of a form suitable for the application of Lemma \ref{Lem.LimitDerivatives}. In the same step we apply Lemma \ref{Lem.LimitDerivatives} and find the asymptotic contribution of each summand in (\ref{Eq.OriginMasterFactDecomposition}), see (\ref{Eq.LimitPNZero}) and (\ref{Eq.LimitPNNonZero}). The limiting expression in (\ref{Eq.LimitPNNonZero}) involves a Pfaffian that depends on the limiting kernel $K^{\infty}$ from Step 2, and in Step 4 we find a different formula for it in terms of $\kgse$, see (\ref{Eq.KeyPfaffianEquality}). In the fifth and final step we combine our expansion from (\ref{Eq.OriginMasterFactDecomposition}) in Step 3 and the alternative formula for the summand limits from (\ref{Eq.KeyPfaffianEquality}) in Step 4 to conclude the proof of (\ref{Eq.FactorialMomentsConvergence}).\\ 

{\bf \raggedleft Step 1.} For $n \in \mathbb{N}$ we set $P_n(x) = \binom{x}{n} =  \frac{1}{n!} \prod_{i = 1}^{n}(x-i + 1)$. From multiple applications of the binomial theorem, we have for $ x \in \mathbb{N}$ 
$$(1+t)^{2x} = \sum_{k = 0}^{\infty} \binom{2x}{k}t^k \mbox{ and } (1+t)^{2x} = (1+2t+t^2)^x = \sum_{k = 0}^\infty\binom{x}{k} (2t+t^2)^k = \sum_{k = 0}^\infty\binom{x}{k} t^k \sum_{j = 0}^k \binom{k}{j} 2^{k-j} t^j. $$
Comparing the coefficients of $t^n$ on both sides gives
\begin{equation*}
P_n(2x) = \sum_{k = \lceil n/2 \rceil}^n  \binom{k}{n-k} 2^{2k - n} \cdot P_k(x),
\end{equation*}
which as both sides are degree $n$ polynomials extends to all $x \in \mathbb{C}$. For $Q_n(x) = n! P_n(x)$ we get
\begin{equation}\label{Eq.FactorialMomentsRelation}
Q_n(2x) = n!\sum_{k = \lceil n/2 \rceil}^n  \binom{k}{n-k} 2^{2k - n} \cdot P_k(x) = \sum_{k = \lceil n/2 \rceil}^n  \frac{n! 2^{2k - n}}{(n-k)!(2k-n)!} \cdot Q_k(x).
\end{equation}

From \cite[(B.2)]{DY25} we have for any $k_1, \dots, k_m \in \mathbb{N}$
$$\mathbb{E} \left[ \prod_{i = 1}^m Q_{k_i}(M^{\mathrm{GSE}}([a_i,b_i]))\right] = \int_{[a_1,b_1]^{k_1} \times \cdots \times [a_m,b_m]^{k_m}} \mathrm{Pf}[K^{\mathrm{GSE}}(x_i,x_j)]_{i,j = 1}^k dx.$$
In addition, from (\ref{Eq.FactorialMomentsRelation}) we have
\begin{equation*}
\begin{split}
&\mathbb{E} \left[ \prod_{i = 1}^m\frac{\left(2M^{\mathrm{GSE}}([a_i,b_i])\right)!}{\left(2M^{\mathrm{GSE}}([a_i,b_i]) - n_i\right)!}\right] = \mathbb{E} \left[ \prod_{i = 1}^m Q_{n_i}(2M^{\mathrm{GSE}}([a_i,b_i]))\right] \\
& = \sum_{k_1 = \lceil n_1/2\rceil}^{n_1} \cdots \sum_{k_m = \lceil n_m/2\rceil}^{n_m}  \prod_{i = 1}^m \frac{n_i! 2^{2k_i - n_i}}{(n_i-k_i)!(2k_i-n_i)!} \cdot \mathbb{E} \left[ \prod_{i = 1}^m Q_{k_i}(M^{\mathrm{GSE}}([a_i,b_i]))\right].
\end{split}
\end{equation*}
The last two displayed equations give (\ref{Eq.FactorialMomentsDoubledFormula}).\\

{\bf \raggedleft Step 2.} In the remainder of the proof we focus on establishing (\ref{Eq.FactorialMomentsConvergence}), and by possibly passing to a subsequence, assume throughout that $t_N \in (0,1/2]$. In this step we analyze the kernels $K^N$ from (\ref{Eq.DefKN}). Note that for any $L > 0$, we can find a constant $c_1>0$, depending on $L$ alone, such that for $z \in \mathcal{C}_1^{\pi/3}$ and $x \in [-L,L]$
\begin{equation}\label{Eq.EstimateExpon}
|\exp(z^3/3 - zx)| \leq \exp(-|z|^3/3 + c_1 + c_1|z|^2).
\end{equation}
In addition, using that $t_N \in (0,1/2]$, we have for $z,w \in \mathcal{C}_1^{\pi/3}$
\begin{equation}\label{Eq.EstimateRat}
\left|\frac{z-w}{(z+w+2t_N)(z+t_N)(w+t_N)}\right|, \left|\frac{z-w}{(z+t_N)(z+w)}\right|,  \left| \frac{z-w}{z+w -t_N} \right| \leq |z| + |w|.
\end{equation}
A straightforward application of the dominated convergence theorem, using the bounds in (\ref{Eq.EstimateExpon}) and (\ref{Eq.EstimateRat}), shows that $K_{ij}^N$ are smooth functions on $\mathbb{R}^2$ for each $N$, and their derivatives are given by differentiating under the integral. The same statement also holds for $K_{ij}^{\mathrm{GSE}}$. From (\ref{Def.Kgse}) and (\ref{Eq.DefKN}) we explicitly have
\begin{equation}\label{Eq.DerivativesKN}
\begin{split}
&\partial_x^n \partial_y^m K^N_{11}(x,y) = \frac{1}{(2\pi \im)^2} \int_{\mathcal{C}_{1}^{\pi/3}}dz \int_{\mathcal{C}_{1}^{\pi/3}} dw \frac{(z - w)(-z)^n(-w)^mH(z,x;w,y)}{4(z + w + 2t_N)(z+t_N)(w+ t_N)}, \\
&\partial_x^n \partial_y^m K^N_{12}(x,y) = - \partial_x^m \partial_y^n K^N_{21}(y,x)\\
& = \frac{1}{(2\pi \im)^2} \int_{\mathcal{C}_{1}^{\pi/3}}dz \int_{\mathcal{C}_{1}^{\pi/3}} dw \frac{(z - w) (-z)^n (-w)^mH(z,x;w,y)}{2(z+ t_N)(z + w)}, \\
&\partial_x^n \partial_y^m K^N_{22}(x,y) = \frac{1}{(2\pi \im)^2} \int_{\mathcal{C}_{1}^{\pi/3}}dz \int_{\mathcal{C}_{1}^{\pi/3}} dw \frac{(z - w ) (-z)^n (-w)^mH(z,x;w,y)}{z + w - 2t_N},
\end{split}
\end{equation}
\begin{equation}\label{Eq.DerivativesKGUE}
\begin{split}
&\partial_x^n \partial_y^m K^{\mathrm{GSE}}_{11}(x,y) = \frac{1}{(2\pi \im)^2} \int_{\mathcal{C}_{1}^{\pi/3}}dz \int_{\mathcal{C}_{1}^{\pi/3}} dw \frac{(z - w)(-z)^n(-w)^mH(z,x;w,y)}{4(z + w)zw}, \\
&\partial_x^n \partial_y^m K^{\mathrm{GSE}}_{12}(x,y) = - \partial_x^m \partial_y^n K^{\mathrm{GSE}}_{21}(y,x)\\
& = \frac{1}{(2\pi \im)^2} \int_{\mathcal{C}_{1}^{\pi/3}}dz \int_{\mathcal{C}_{1}^{\pi/3}} dw \frac{(z - w) (-z)^n (-w)^mH(z,x;w,y)}{4z(z + w)}, \\
&\partial_x^n \partial_y^m K^{\mathrm{GSE}}_{22}(x,y) = \frac{1}{(2\pi \im)^2} \int_{\mathcal{C}_{1}^{\pi/3}}dz \int_{\mathcal{C}_{1}^{\pi/3}} dw \frac{(z - w ) (-z)^n (-w)^mH(z,x;w,y)}{z + w}.
\end{split}
\end{equation}

Using the bounds in (\ref{Eq.EstimateExpon}) and (\ref{Eq.EstimateRat}), and the dominated convergence theorem we conclude that 
\begin{equation}\label{Eq.KernelConvergenceOrigin}
\begin{split}
&\lim_{N \rightarrow \infty} \partial_x^n \partial_y^m \begin{bmatrix} K_{11}^N(x,y) & K_{12}^N(x,y) \\ K_{21}^N(x,y) & K_{22}^N(x,y) \end{bmatrix} = \partial_x^n \partial_y^m \begin{bmatrix} K_{11}^{\infty}(x,y) & K_{12}^{\infty}(x,y) \\ K_{21}^{\infty}(x,y) & K_{22}^{\infty}(x,y) \end{bmatrix},
\end{split}
\end{equation}
where
\begin{equation}\label{Eq.LimitKernelToGSE}
\begin{split}
& \begin{bmatrix} K_{11}^{\infty}(x,y) & K_{12}^{\infty}(x,y) \\ K_{21}^{\infty}(x,y) & K_{22}^{\infty}(x,y) \end{bmatrix} =  \begin{bmatrix} K_{11}^{\mathrm{GSE}}(x,y) & 2K_{12}^{\mathrm{GSE}}(x,y) \\ K_{21}^{\mathrm{GSE}}(x,y) & 4K_{22}^{\mathrm{GSE}}(x,y) \end{bmatrix}. 
\end{split}
\end{equation}
In addition, from (\ref{Eq.EstimateExpon}), (\ref{Eq.EstimateRat}) and (\ref{Eq.DerivativesKN}), we conclude that for any $L > 0$, we can find $c_2 > 0$, depending on $L$ alone, such that for $x, y \in [-L, L]$, $m+n \leq L$, $N \geq 1$, and $i,j \in \{1,2\}$
\begin{equation}\label{Eq.KernelBoundsOrigin}
\begin{split}
\left| \partial_x^n \partial_y^m K_{ij}^{N}(x,y)\right| \leq c_2.
\end{split}
\end{equation}

We end this step by observing from (\ref{Eq.DerivativesKGUE}) and (\ref{Eq.LimitKernelToGSE}) that 
\begin{equation}\label{Eq.DerivativeRlations}
\begin{split}
&\partial_{y} K^{\infty}_{11}(x,y) = - \kgse_{12}(x,y) \mbox{, } \partial_{x} K^{\infty}_{11}(x,y) = -\kgse_{21} (x,y), \\
&\partial_{xy} K^{\infty}_{11}(x,y) = \kgse_{22}(x,y), \hspace{2mm} \partial_x K^{\infty}_{12}(x,y) = -2 \kgse_{22}(x,y).
\end{split}
\end{equation}

{\bf \raggedleft Step 3.} In this step we seek to apply Lemma \ref{Lem.LimitDerivatives}. From the kernel decomposition formula (\ref{Eq.SplitKernel}) and Lemma \ref{Lem.SumPfaffian}, we get 
$$\mathrm{Pf}[\hat{K}^N(x_i,x_j)]_{i,j = 1}^n = \sum_{r = 0}^{\lfloor n/2 \rfloor }\sum_{I \subseteq \{1,2, \dots, 2n\}: |I| = 2r} (-1)^{\Sigma(I) - r } \cdot \mathrm{Pf}[\Delta^N_I(x)] \cdot \mathrm{Pf}[K^N_{I^c}(x)].$$
Notice that if $2s+1 \in I$, then $\Delta^N_I(x)$ has a row/column of all zeros, and so $\mathrm{Pf}[\Delta^N_I(x)] = 0$. Consequently, we may restrict the above sum to $I \subseteq \{2,4, \dots, 2n\}$ in which case $\Sigma(I)$ is even. This simplifies the formula to
$$\mathrm{Pf}[\hat{K}^N(x_i,x_j)]_{i,j = 1}^n = \sum_{r = 0}^{\lfloor n/2 \rfloor}\sum_{I \subseteq \{2,4, \dots, 2n\}: |I| = 2r} (-1)^{r} \cdot \mathrm{Pf}[\Delta^N_I(x)] \cdot \mathrm{Pf}[K^N_{I^c}(x)].$$
Combining the last equation with the joint factorial moment formula \cite[(B.2)]{DY25} and the Pfaffian expansion formula (\ref{Eq.DefPfaffian2}) applied to $\Delta^N_I(x)$, we obtain 
\begin{equation}\label{Eq.OriginMasterFactDecomposition}
\begin{split}
\mathbb{E} \left[ \prod_{i = 1}^m\frac{\left(M^{N}([a_i,b_i])\right)!}{\left(M^{N}([a_i,b_i]) - n_i\right)!}\right] = \sum_{r = 0}^{\lfloor n/2 \rfloor}\sum_{I \subseteq \{2,4, \dots, 2n\}: |I| = 2r} (-1)^{r} \sum_{w } (-1)^wP^N(w),
\end{split}
\end{equation}
where the rightmost sum is over permutations $w = (u_1,v_1, \dots, u_r,v_r)$ of $I$ with $u_1 < \cdots < u_r$, $u_i < v_i$ for $i = 1, \dots, r$, the sign $(-1)^w$ is as in Definition \ref{Def.LimitDerivatives} and 
\begin{equation}\label{Eq.OriginBasicTerm}
\begin{split}
P^N(w) = \int_{[a_1, b_1]^{n_1} \times \cdots \times [a_m, b_m]^{n_m}}dx \delta_N(x_{u_1}, x_{v_1}) \cdots \delta_N(x_{u_r}, x_{v_r}) \mathrm{Pf}[K^N_{I^c}(x)].
\end{split}
\end{equation}

We now seek to apply Lemma \ref{Lem.LimitDerivatives} with $f_N(x) = \mathrm{Pf}[K^N_{I^c}(x)]$ and take the limit of $P^N(w)$. Using either of the Pfaffian expansions (\ref{Eq.DefPfaffian}) or (\ref{Eq.DefPfaffian2}), and (\ref{Eq.DerivativesKN}), we see that $\mathrm{Pf}[K^N_{I^c}(x)]$ are smooth functions. In addition, from (\ref{Eq.KernelConvergenceOrigin}) all its derivatives converge to those of $\mathrm{Pf}[K^{\infty}_{I^c}(x)]$, which verifies (\ref{Eq.LimitSmoothFunc}). Lastly, (\ref{Eq.KernelBoundsOrigin}) ensures the derivative bounds in (\ref{Eq.SmoothFuncBound}). Overall, we conclude that all the conditions of Lemma \ref{Lem.LimitDerivatives} are satisfied. In what follows we adopt the notation of Definition \ref{Def.LimitDerivatives}.

If $u_1,\dots, u_r, v_1, \dots, v_r$ are not $\vec{n}$-compatible, we have from Lemma \ref{Lem.LimitDerivatives} that 
\begin{equation}\label{Eq.LimitPNZero}
\lim_{N \rightarrow \infty} P^N(w) = 0.
\end{equation}
On the other hand, if they are $\vec{n}$-compatible, then
\begin{equation*}\label{Eq.CompatibleMasterLimit}
\lim_{N \rightarrow \infty} P^N(w) = \int_{[a_{1}, b_{1}]^{I_1^{vc}} \times \cdots  \times [a_m, b_m]^{I_m^{vc}}} dy \left[(\partial_{x_{v_1}} - \partial_{x_{u_1}}) \cdots (\partial_{x_{v_r}} - \partial_{x_{u_r}})  \mathrm{Pf}[K_{I^c}^{\infty}(x)] \right]\vert_{x = \bar{y}}.
\end{equation*}
Using the symmetrization identity (\ref{Eq.SymmDeriv}) and the fact that if we permute the rows/columns of a skew-symmetric matrix by a permutation $\sigma$, its Pfaffian changes by $(-1)^{\sigma}$, we see that 
$$\left[(\partial_{x_{v_1}} - \partial_{x_{u_1}}) \cdots (\partial_{x_{v_r}} - \partial_{x_{u_r}})  \mathrm{Pf}[K_{I^c}^{\infty}(x)] \right]\vert_{x = \bar{y}} = 2^r \left[\partial_{x_{v_1}} \cdots \partial_{x_{v_r}} \mathrm{Pf}[K_{I^c}^{\infty}(x)] \right]\vert_{x = \bar{y}},$$
and consequently
\begin{equation}\label{Eq.LimitPNNonZero}
\lim_{N \rightarrow \infty} P^N(w) =  2^r\int_{[a_{1}, b_{1}]^{I_1^{vc}} \times \cdots  \times [a_m, b_m]^{I_m^{vc}}} dy \left[\partial_{x_{v_1}} \cdots \partial_{x_{v_r}} \mathrm{Pf}[K_{I^c}^{\infty}(x)] \right]\vert_{x = \bar{y}}.
\end{equation}

{\bf \raggedleft Step 4.} In this step we assume that $u_1,\dots, u_r, v_1, \dots, v_r$ are $\vec{n}$-compatible, and we seek to show that for each $y \in [a_{1}, b_{1}]^{I_1^{vc}} \times \cdots  \times [a_m, b_m]^{I_m^{vc}}$
\begin{equation}\label{Eq.KeyPfaffianEquality}
\left[\partial_{x_{v_1}} \cdots \partial_{x_{v_r}} \mathrm{Pf}[K_{I^c}^{\infty}(x)] \right]\vert_{x = \bar{y}} = (-1)^{w} \cdot (-1)^r 2^{n-2r} \cdot  \mathrm{Pf}[K^{\mathrm{GSE}}(y_i,y_j)]_{i,j \in I_1^{vc} \sqcup \cdots \sqcup I_m^{vc}}.
\end{equation}

We first observe that 
\begin{equation}\label{Eq.TransformKOnce}
\partial_{x_{v_1}} \cdots \partial_{x_{v_r}} \mathrm{Pf}[K_{I^c}^{\infty}(x)] = (-1)^w\mathrm{Pf}[K'(x)],
\end{equation}
where $K'(x)$ is the $(2n-2r) \times (2n-2r)$ matrix, given by
\begin{equation}\label{Eq.BlockOnce}
\begin{split}
&K'(x) = \begin{bmatrix} A' & B'\\ -(B')^T & D' \end{bmatrix}, \mbox{ with } A' = \left( \begin{matrix} K_{11}^{\infty}(x_{u_i}, x_{u_j}) & [\partial_y K^{\infty}_{11}](x_{u_i}, x_{v_j}) \\ [\partial_x K^{\infty}_{11}](x_{v_i}, x_{u_j}) & [\partial_{x} \partial_y K^{\infty}_{11}](x_{v_i}, x_{v_j})  \end{matrix} \right)_{i,j = 1, \dots, r}, \\
& B' = \left( \begin{matrix} K_{11}^{\infty}(x_{u_i}, x_{p_j}) & K^{\infty}_{12}(x_{u_i}, x_{p_j}) \\ [\partial_x K^{\infty}_{11}](x_{v_i}, x_{p_j}) & [\partial_{x}K^{\infty}_{12}](x_{v_i}, x_{p_j})  \end{matrix} \right)_{i = 1, \dots, r; j = 2r+1, \dots, n}, \\
&D' = \left( \begin{matrix} K_{11}^{\infty}(x_{p_i}, x_{p_j}) & K^{\infty}_{12}(x_{p_i}, x_{p_j}) \\ K^{\infty}_{21}(x_{p_i}, x_{p_j}) & K^{\infty}_{22}(x_{p_i}, x_{p_j})  \end{matrix} \right)_{i,j = 2r+1, \dots, n},
\end{split}
\end{equation}
where $p_{2r+1}, \dots, p_{n}$ are the elements of $\{1,\dots, n\} \setminus \{u_1,\dots, u_r, v_1, \dots, v_r\}$ in ascending order. Equation (\ref{Eq.TransformKOnce}) is directly verified by expanding both Pfaffians using (\ref{Eq.DefPfaffian}) and comparing the resulting terms. We mention that the sign $(-1)^w$ can be traced to the fact that the rows/columns $K_{I^c}^{\infty}(x)$ of index $u_1,v_1, u_2, v_2,\dots, u_r,v_r$ have been moved to the top/left $2r$ rows/columns within $K'(x)$.

Setting $x = \bar{y}$ and applying (\ref{Eq.LimitKernelToGSE}) and (\ref{Eq.DerivativeRlations}) gives
\begin{equation}\label{Eq.BlockTwice}
\begin{split}
&K'(\bar{y}) = \begin{bmatrix} A'' & B''\\ -(B'')^T & D'' \end{bmatrix}, \mbox{ where } A'' = \left( \begin{matrix}\kgse_{11}(y_{u_i}, y_{u_j}) &  - \kgse_{12}(y_{u_i}, y_{u_j}) \\ -\kgse_{21}(y_{u_i}, y_{u_j}) & \kgse_{22}(y_{u_i}, y_{u_j})  \end{matrix} \right)_{i,j = 1, \dots, r}, \\
& B'' = \left( \begin{matrix} \kgse_{11}(y_{u_i}, y_{p_j}) & \kgse_{12}(y_{u_i}, y_{p_j}) \\ -\kgse_{21}(y_{u_i}, y_{p_j}) & -2\kgse_{22}(y_{u_i}, y_{p_j})  \end{matrix} \right)_{i = 1, \dots, r; j = 2r+1, \dots, n}, \\
& D'' = \left( \begin{matrix} \kgse_{11}(y_{p_i}, y_{p_j}) & 2\kgse_{12}(y_{p_i}, y_{p_j}) \\ 2\kgse_{21}(y_{p_i}, y_{p_j}) & 4 \kgse_{22}(y_{p_i}, y_{p_j})  \end{matrix} \right)_{i,j = 2r+1, \dots, n}.
\end{split}
\end{equation}

We finally observe that 
\begin{equation}\label{Eq.TransformKTwice}
\begin{split}
\mathrm{Pf}[K'(\bar{y})] = (-1)^r 2^{n-2r} \cdot \mathrm{Pf}[K^{\mathrm{GSE}}(y_i,y_j)]_{i,j \in I_1^{vc} \sqcup \cdots \sqcup I_m^{vc}}.
\end{split}
\end{equation} 
As before, one verifies (\ref{Eq.TransformKTwice}) by expanding both sides using (\ref{Eq.DefPfaffian}) and comparing the resulting terms. From the block form in (\ref{Eq.BlockTwice}) each term in the expansion of $\mathrm{Pf}[K'(\bar{y})]$ appears in the expansion of $\mathrm{Pf}[K^{\mathrm{GSE}}(y_i,y_j)]_{i,j \in I_1^{vc} \sqcup \cdots \sqcup I_m^{vc}}$ but is multiplied by a constant. Specifically, each term is multiplied by an additional factor of $2$ for each row/column $2r+2,2r+4, \dots, 2n-2r$ as well as a factor of $-1$ for each row/column $1,2,\dots, r$, overall producing the $(-1)^r 2^{n-2r}$.

Equation (\ref{Eq.KeyPfaffianEquality}) now follows from (\ref{Eq.TransformKOnce}) and (\ref{Eq.TransformKTwice}).\\

{\bf \raggedleft Step 5.} In this step we finish the proof of (\ref{Eq.FactorialMomentsConvergence}). From (\ref{Eq.LimitPNZero}) we know that if $u_1, \dots, u_r, v_1, \dots, v_r$ are not $\vec{n}$-compatible, then they do not asymptotically contribute to the right side of (\ref{Eq.OriginMasterFactDecomposition}). On the other hand, if $u_1, \dots, u_r, v_1, \dots, v_r$ are $\vec{n}$-compatible and have profile $\vec{r} = (r_1, \dots, r_m)$ (as in Definition \ref{Def.LimitDerivatives}), then from (\ref{Eq.LimitPNNonZero}) and (\ref{Eq.KeyPfaffianEquality}) we get
\begin{equation}\label{Eq.LimitPNNonZero2}
\lim_{N \rightarrow \infty} P^N(w) = (-1)^w \cdot (-1)^r2^{n-r} \int_{[a_1,b_1]^{n_1-r_1} \times \cdots \times [a_m,b_m]^{n_m-r_m}} \mathrm{Pf}\left[\kgse(x_i,x_j) \right]_{i,j = 1}^kdx.
\end{equation}
We mention that in deriving (\ref{Eq.LimitPNNonZero2}) we also applied a simple change of variables to (\ref{Eq.LimitPNNonZero}), which does not affect the sign of the Pfaffian. 

Note that the sign of (\ref{Eq.LimitPNNonZero2}) exactly cancels with the one in front of $P_N(w)$ in (\ref{Eq.OriginMasterFactDecomposition}). Consequently, each $\vec{n}$-compatible $u_1, \dots, u_r, v_1, \dots, v_r$ of profile $\vec{r}$ contributes the same to the right side of (\ref{Eq.OriginMasterFactDecomposition}). As there are precisely $\prod_{i = 1}^m \binom{n_i}{2r_i} \cdot \frac{2r_i!}{2^{r_i}r_i!} = \prod_{i = 1}^m \frac{n_i!}{(n_i-2r_i)! r_i! 2^{r_i}}$ such terms, we conclude from (\ref{Eq.OriginMasterFactDecomposition}) and (\ref{Eq.LimitPNNonZero2}) that 
\begin{equation*}
\begin{split}
&\lim_{N \rightarrow \infty}\mathbb{E} \left[ \prod_{i = 1}^m\frac{\left(M^{N}([a_i,b_i])\right)!}{\left(M^{N}([a_i,b_i]) - n_i\right)!}\right] = \sum_{r_1 = 0}^{\lfloor n_1/2 \rfloor} \cdots \sum_{r_m = 0}^{\lfloor n_m/2 \rfloor} \prod_{i = 1}^m \frac{n_i!2^{n_i}}{(n-2r_i)! r_i! 2^{2r_i}}\\
&\times \int_{[a_1,b_1]^{n_1-r_1} \times \cdots \times [a_m,b_m]^{n_m-r_m}} \mathrm{Pf}\left[\kgse(x_i,x_j) \right]_{i,j = 1}^kdx.
\end{split}
\end{equation*}
The last displayed equation and (\ref{Eq.FactorialMomentsDoubledFormula}) imply (\ref{Eq.FactorialMomentsConvergence}) once we set $k_i = n_i - r_i$ for $i = 1, \dots, m$.
\end{proof}

%
%
\subsection{Proof of Lemma \ref{Lem.OriginMeasure}}\label{Section6.3} In this section we present the proof of Lemma \ref{Lem.OriginMeasure}. We continue with the same notation as above Section \ref{Section6.1}, and split the proof into three steps. In the first step we assume that $M^N \Rightarrow 2M^{\mathrm{GSE}}$ and conclude the statement of the lemma. In the second step we reduce the proof of $M^N \Rightarrow 2M^{\mathrm{GSE}}$ to establishing three claims. These are required for the application of \cite[Corollary 2.4]{dimitrov2024airy}. In the same step we also verify the third of these claims, which shows that the factorial moments of $2M^{\mathrm{GSE}}$ do not grow too quickly, see (\ref{Eq.SumPowerSeriesMoment}). In the third step we verify the first two claims from Step 2, by adapting the arguments in the proof of \cite[Proposition 2.2]{dimitrov2024airy}.\\

{\bf \raggedleft Step 1.} We claim that for any sequence $t_N \rightarrow 0+$
\begin{equation}\label{Eq.MeasConv}
M^N \Rightarrow 2M^{\mathrm{GSE}}.
\end{equation}
We prove (\ref{Eq.MeasConv}) in the second step. Here, we assume its validity and conclude the proof of the lemma.\\

From Theorem \ref{Thm.Convergence} we know that $\hsai$ is ordered, and so a.s. $\hsai_i(t) \geq \hsai_{i+1}(t)$ for $i \geq 1$ and $t \in [0, \infty)$. In particular, $\lim_{k \rightarrow \infty} \hsai_k(t)$ exists in $[-\infty, \infty)$. We define for $L \in \mathbb{N}$ the events
$$E_L = \left\{\lim_{k \rightarrow \infty} \hsai_k(0) \geq - L \right\}.$$
Our first task is to show for each $L \in \mathbb{N}$
\begin{equation}\label{Eq.LocFinite0}
\mathbb{P}(E_L) = 0.
\end{equation}

We argue by contradiction and assume $\mathbb{P}(E_L) \geq \varepsilon$ for some $\varepsilon > 0$. Pick a large $L_1 \geq L$ so that 
$$\mathbb{P}(\hsai_1(0) > L_1) \leq \varepsilon/2.$$
Let $f$ be a non-negative continuous function that is supported on $[-L_1-2, L_1+2]$ and $f(x) = 1$ for $x \in [-L_1 - 1, L_1+ 1]$. Fixing any $n \in \mathbb{N}$, we observe by the continuity of $\hsai$ that 
\begin{equation}\label{Eq.LimitPairing}
\lim_{N \rightarrow \infty}\sum_{i = 1}^n f(\hsai_i(t_N)) =\sum_{i = 1}^n f(\hsai_i(0)).
\end{equation} 
Fix $\omega \in E_L \cap \{\hsai_1(0) \leq L_1\}$, and note by the continuity of $\hsai_1$ and $\hsai_n$ that we can find $N_0 = N_0(\omega)$ such that for $N \geq N_0$
$$-L_1 - 1 \leq \hsai_n(0) - 1 \leq \hsai_n(t_N) \leq \hsai_{n-1}(t_N) \leq \cdots \leq \hsai_1(t_N) \leq \hsai_1(0) + 1 \leq L_1 + 1.$$  
The above set of inequalities and the definition of $f$ show that for all $N \geq N_0$
$$M^{N}f = \sum_{i \geq 1} f(\hsai_i(t_N)) \geq \sum_{i = 1}^n f(\hsai_i(t_N)) = n.$$ 
As $n$ was arbitrary, we conclude that a.s. on the event $E_L \cap \{\hsai_1(0) \leq L_1\}$, which has probability at least $\varepsilon/2$ by construction, we have 
$M^N f \rightarrow \infty$. But from (\ref{Eq.MeasConv}), we have $M^N f \Rightarrow 2M^{\mathrm{GSE}}f$ and the latter is a finite real-valued random variable. This gives our desired contradiction, which proves (\ref{Eq.LocFinite0}).\\

From (\ref{Eq.LocFinite0}) we conclude $M$ is a.s. locally finite. Fix $L > 0$ and a continuous function $g$ on $\mathbb{R}$, supported on $[-L,L]$. Our second task is to show that a.s.
\begin{equation}\label{Eq.ConvToMInfinity}
\lim_{N \rightarrow \infty} M^Ng = Mg.
\end{equation}
From (\ref{Eq.LocFinite0}) we can find $K = K(\omega)$, such that $\hsai_{k}(0) \leq - L - 1$ for $k \geq K+1$. In addition, by the continuity of $\hsai_{K+1}$ we can find $N_0 = N_0(\omega)$, such that for $N \geq N_0$ and $k \geq K+1$
$$ \hsai_k(t_N) \leq \hsai_{K+1}(t_N) \leq \hsai_{K+1}(0) + 1 \leq -L.$$
Using the above inequalities we conclude for all $N \geq N_0$ 
$$M^Ng = \sum_{i = 1}^{\infty} g(\hsai_i(t_N)) = \sum_{i = 1}^K g(\hsai_i(t_N)) \mbox{ and } M^{\infty}g = \sum_{i = 1}^{\infty} g(\hsai_i(0)) = \sum_{i = 1}^K g(\hsai_i(0)).$$
The last displayed equation and the continuity of $g$ and $\hsai$ gives (\ref{Eq.ConvToMInfinity}).\\
  
From (\ref{Eq.ConvToMInfinity}) we conclude that $M^N$ converge in the vague topology to $M$ almost surely, and hence also weakly. As weak limits are unique, we conclude from (\ref{Eq.MeasConv}) that $M \overset{d}{=}2M^{\mathrm{GSE}}$. \\

{\bf \raggedleft Step 2.} In this step we prove (\ref{Eq.MeasConv}). Fix $m \in \mathbb{N}$, $n_1, \dots, n_m \in \mathbb{N}$, $m$ pairwise disjoint intervals $[a_i,b_i] \subset \mathbb{R}$ with $b_i > a_i$, and set $\vec{n} = (n_1, \dots, n_m)$, $\vec{a} = (a_1, \dots, a_m)$ and $\vec{b} = (b_1, \dots, b_m)$. From equation (\ref{Eq.FactorialMomentsConvergence}) in Lemma \ref{Lem.FactMomentConvergence} we have
\begin{equation}\label{Eq.XY1}
\begin{split}
&\lim_{N \rightarrow \infty}\mathbb{E} \left[ \prod_{i = 1}^m\frac{\left(M^{N}([a_i,b_i])\right)!}{\left(M^{N}([a_i,b_i]) - n_i\right)!}\right] = \mathbb{E} \left[ \prod_{i = 1}^m\frac{\left(2M^{\mathrm{GSE}}([a_i,b_i])\right)!}{\left(2M^{\mathrm{GSE}}([a_i,b_i]) - n_i\right)!}\right].
\end{split}
\end{equation}
Note that the right side in (\ref{Eq.XY1}) is finite because of (\ref{Eq.FactorialMomentsDoubledFormula}). In addition, from (\ref{Eq.OriginMasterFactDecomposition}) the left side is finite for finite $N$. In particular, we conclude that we can find $C(\vec{n},\vec{a},\vec{b}) > 0$, such that 
\begin{equation}\label{Eq.XY2}
\begin{split}
&\sup_{N \geq 1}\mathbb{E} \left[ \prod_{i = 1}^m\frac{\left(M^{N}([a_i,b_i])\right)!}{\left(M^{N}([a_i,b_i]) - n_i\right)!}\right] \leq C(\vec{n},\vec{a},\vec{b}).
\end{split}
\end{equation}
In the remainder we use (\ref{Eq.XY1}) and (\ref{Eq.XY2}), and adapt the argument from \cite[Proposition 2.2]{dimitrov2024airy} to prove (\ref{Eq.MeasConv}). We mention that \cite[Proposition 2.2]{dimitrov2024airy} is not directly applicable here as the intervals in that result are not assumed to be well-separated, as they are in our setting.\\

We make the following claims:
\begin{enumerate}
\item $\{M^N\}_{N \geq 1}$ is a tight sequence of locally bounded measures.
\item Any subsequential limit $M^{\infty}$ of $\{M^N\}_{N \geq 1}$ is a point process and satisfies
\begin{equation}\label{Eq.XY3A}
\begin{split}
&\mathbb{E} \left[ \prod_{i = 1}^m\frac{\left(2M^{\mathrm{GSE}}((A_i,B_i])\right)!}{\left(2M^{\mathrm{GSE}}((A_i,B_i]) - n_i\right)!}\right] = \mathbb{E} \left[ \prod_{i = 1}^m\frac{\left(M^{\infty}((A_i,B_i])\right)!}{\left(M^{\infty}((A_i, B_i]) - n_i\right)!}\right],
\end{split}
\end{equation}
where $(A_i,B_i]$ are pairwise disjoint intervals that are allowed to be empty.
\item For any $b > a$, we can find $\epsilon > 0$, such that 
\begin{equation}\label{Eq.SumPowerSeriesMoment}
\sum_{n = 1}^{\infty} \frac{\epsilon^n}{n!} \cdot \mathbb{E} \left[ \frac{\left(2M^{\mathrm{GSE}}((a,b])\right)!}{\left(2M^{\mathrm{GSE}}((a,b]) - n\right)!} \right] < \infty.
\end{equation}
\end{enumerate}
From (\ref{Eq.XY3A}) and (\ref{Eq.SumPowerSeriesMoment}) we see that the conditions of \cite[Corollary 2.4]{dimitrov2024airy} are satisfied, which proves that $M^{\infty} \overset{d}{=} 2M^{\mathrm{GSE}}$. As $M^{\infty}$ was an arbitrary subsequential limit of $M^N$, which is tight by the first claim above, we conclude $M^N \Rightarrow 2M^{\mathrm{GSE}}$ as desired.

Our work so far reduces the proof of the lemma to establishing the above three claims. The first two are shown in the next step. In the remainder of this step we verify (\ref{Eq.SumPowerSeriesMoment}) using the factorial moment formula from (\ref{Eq.FactorialMomentsDoubledFormula}). \\

Recall from Step 2 of the proof of Lemma \ref{Lem.FactMomentConvergence} that $\kgse_{ij}(x,y)$ are smooth functions and hence bounded for $x,y \in [a,b]$. From (\ref{Eq.FactorialMomentsDoubledFormula}), and the Pfaffian expansion formula (\ref{Eq.DefPfaffian}) we can find a constant $C>0$ (depending on $a,b$), such that 
$$\mathbb{E} \left[ \frac{\left(2M^{\mathrm{GSE}}((a,b])\right)!}{\left(2M^{\mathrm{GSE}}((a,b]) - n\right)!} \right] \leq \sum_{k = \lceil n/2 \rceil}^n \frac{n!2^{2k-n}}{(n-k)!(2k-n)!} \cdot C^k \cdot \frac{(2k)!}{k!2^k}.$$
If $\epsilon > 0$ is small enough so that $4C\epsilon (1 + \epsilon/2) < 1$, we obtain
\begin{equation*}
\begin{split}
&\sum_{n = 1}^{\infty} \frac{\epsilon^n}{n!} \cdot \mathbb{E} \left[ \frac{\left(2M^{\mathrm{GSE}}((a,b])\right)!}{\left(2M^{\mathrm{GSE}}((a,b]) - n\right)!} \right] \leq \sum_{n = 1}^{\infty} \frac{\epsilon^n}{n!}\sum_{k = \lceil n/2 \rceil}^n \frac{n!2^{2k-n}}{(n-k)!(2k-n)!} \cdot C^k \cdot \frac{(2k)!}{k!2^k} \\
& = \sum_{k = 1}^{\infty} \frac{(2k)!}{k!} [2C]^k \sum_{n = k}^{2k} \frac{[\epsilon/2]^{n}}{(n-k)!(2k-n)!} = \sum_{k = 1}^{\infty} \frac{(2k)!}{k!k!} [C\epsilon]^k  \sum_{m = 0}^{k} \frac{k![\epsilon/2]^{m}}{m!(k-m)!} \\
&= \sum_{k = 1}^{\infty} \frac{(2k)!}{k!k!} [C\epsilon]^k (1+\epsilon/2)^k = \sum_{k = 1}^{\infty} \binom{2k}{k} [C\epsilon]^k (1+\epsilon/2)^k \leq \sum_{k \geq 1} [4 C \epsilon (1 + \epsilon/2)]^k < \infty.
\end{split}
\end{equation*}
We mention that in going from the second to the third line we used the binomial theorem, and in the next to last inequality we used that $\binom{2k}{k} \leq 2^{2k}$. The last inequality implies (\ref{Eq.SumPowerSeriesMoment}).\\

{\bf \raggedleft Step 3.} In this final step we establish the first two claims from Step 2. Let $F$ be a bounded Borel set, and let $[a,b]$ be such that $F \subseteq [a,b]$. From Chebyshev's inequality we conclude
\begin{equation*}
\begin{split}
&\lim_{r \rightarrow \infty} \sup_{N \geq 1} \mathbb{P}(M^N(F) \geq r) \leq \lim_{r \rightarrow \infty} \sup_{N \geq 1} \mathbb{P}(M^N([a,b]) \geq r) \leq \lim_{r \rightarrow \infty} r^{-1} C(1,a,b) = 0,  
\end{split}
\end{equation*}
where in the last inequality we used (\ref{Eq.XY2}). From \cite[Theorem 4.10]{kallenberg2017random} we conclude that $\{M^N\}_{N \geq 1}$ is a tight sequence of locally finite measures. By Skorohod's representation theorem \cite[Theorem 6.7]{Billing} we may assume that $\{M^{N_v}\}_{v \geq 1}, M^{\infty}$ are all defined on the same probability space $(\Omega, \mathcal{F}, \mathbb{P})$ and, for each $\omega \in \Omega$,
$$M^{N_v}(\omega) \overset{v}{\rightarrow} M^{\infty}(\omega).$$
One consequence of the last equation is that since $M^{N_v}$ are integer-valued measures, the same is true for $M^{\infty}$. In other words, $M^{\infty}$ is a point process. \\

What remains is to prove (\ref{Eq.XY3A}). Fix $k_1, \dots, k_m \in \mathbb{N}$ and $m$ pairwise disjoint intervals $[c_i,d_i] \subset \mathbb{R}$ with $d_i > c_i$. From \cite[Lemma 4.1]{kallenberg2017random} we have $\liminf_{v \rightarrow \infty} M^{N_v}([c_i,d_i]) \geq M^{\infty}((c_i,d_i)) $, which implies 
$$\liminf_{v \rightarrow \infty} \prod_{i = 1}^m\frac{\left(M^{N_v}([c_i,d_i])\right)!}{\left(M^{N_v}([c_i,d_i]) - k_i\right)!} \geq  \prod_{i = 1}^m\frac{\left(M^{\infty}((c_i,d_i))\right)!}{\left(M^{\infty}((c_i, d_i)) - k_i\right)!}.$$
Combining the latter with Fatou's lemma and (\ref{Eq.XY1}), gives
\begin{equation}\label{Eq.MomentUpperBound}
\mathbb{E} \left[ \prod_{i = 1}^m\frac{\left(2M^{\mathrm{GSE}}([c_i,d_i])\right)!}{\left(2M^{\mathrm{GSE}}([c_i,d_i]) - n_i\right)!}\right] \geq \mathbb{E} \left[ \prod_{i = 1}^m\frac{\left(M^{\infty}((c_i,d_i))\right)!}{\left(M^{\infty}((c_i, d_i)) - k_i\right)!}\right].
\end{equation}

From \cite[Lemma 4.1]{kallenberg2017random} we have $\limsup_{v \rightarrow \infty} M^{N_v}([c_i,d_i]) \leq M^{\infty}([c_i,d_i])$, which implies
$$\limsup_{v \rightarrow \infty} \prod_{i = 1}^m\frac{\left(M^{N_v}([c_i,d_i])\right)!}{\left(M^{N_v}([c_i,d_i]) - k_i\right)!} \leq  \prod_{i = 1}^m\frac{\left(M^{\infty}([c_i,d_i])\right)!}{\left(M^{\infty}([c_i, d_i]) - k_i\right)!}.$$
Combining the latter with the reverse Fatou's lemma, see \cite[page 10]{cairoli2011}, and (\ref{Eq.XY1}) gives 
\begin{equation}\label{Eq.MomentLowerBound}
\mathbb{E} \left[ \prod_{i = 1}^m\frac{\left(2M^{\mathrm{GSE}}([c_i,d_i])\right)!}{\left(2M^{\mathrm{GSE}}([c_i,d_i]) - k_i\right)!}\right] \leq \mathbb{E} \left[ \prod_{i = 1}^m\frac{\left(M^{\infty}([c_i,d_i])\right)!}{\left(M^{\infty}([c_i, d_i]) - k_i\right)!}\right].
\end{equation}
We mention that in using the reverse Fatou's lemma from \cite{cairoli2011}, we implicitly used that the random variables 
$$X_{N} = \prod_{i = 1}^m\frac{\left(M^{N}([c_i,d_i])\right)!}{\left(M^{N}([c_i,d_i]) - k_i\right)!}$$
are uniformly integrable. The latter can be deduced from the fact that $\sup_{N \geq 1} \mathbb{E}[X_N^2] < \infty$ (see \cite[(3.18)]{Billing}), which in turn follows from (\ref{Eq.XY2}).

Suppose now that $a_i,b_i,n_i$ are as in the beginning of Step 2. Let $\varepsilon > 0$ be small enough so that $[a_i - \varepsilon, b_i + \varepsilon]$ are still pairwise disjoint. Combining (\ref{Eq.MomentLowerBound}) with $c_i = a_i$, $d_i = b_i$, $k_i = n_i$ and (\ref{Eq.MomentUpperBound}) with $c_i = a_i - \varepsilon$, $d_i = b_i + \varepsilon$, $k_i = n_i$, we obtain
\begin{equation*}
\begin{split}
&\mathbb{E} \left[ \prod_{i = 1}^m\frac{\left(2M^{\mathrm{GSE}}([a_i,b_i])\right)!}{\left(2M^{\mathrm{GSE}}([a_i,b_i]) - n_i\right)!}\right] \leq \mathbb{E} \left[ \prod_{i = 1}^m\frac{\left(M^{\infty}([a_i,b_i])\right)!}{\left(M^{\infty}([a_i, b_i]) - n_i\right)!}\right] \\
&\leq  \mathbb{E} \left[ \prod_{i = 1}^m\frac{\left(2M^{\mathrm{GSE}}([a_i - \varepsilon,b_i + \varepsilon])\right)!}{\left(2M^{\mathrm{GSE}}([a_i-\varepsilon,b_i + \varepsilon]) - n_i\right)!}\right].
\end{split}
\end{equation*}
From (\ref{Eq.FactorialMomentsDoubledFormula}) and the bounded convergence theorem we have that the right side above converges to the left as $\varepsilon \rightarrow 0+$. Consequently, the last equation implies
\begin{equation}\label{Eq.XY3}
\begin{split}
&\mathbb{E} \left[ \prod_{i = 1}^m\frac{\left(2M^{\mathrm{GSE}}([a_i,b_i])\right)!}{\left(2M^{\mathrm{GSE}}([a_i,b_i]) - n_i\right)!}\right] = \mathbb{E} \left[ \prod_{i = 1}^m\frac{\left(M^{\infty}([a_i,b_i])\right)!}{\left(M^{\infty}([a_i, b_i]) - n_i\right)!}\right].
\end{split}
\end{equation}
If $(A_i,B_i]$ are pairwise disjoint and non-empty, then we can find a decreasing sequence $a_i^s \downarrow A_i$ with $a_i^s \in (A_i,B_i)$. Putting $a_i = a_i^s$ and $b_i = B_i$ in (\ref{Eq.XY3}) and taking the $s \rightarrow \infty$ limit gives by the monotone convergence theorem (\ref{Eq.XY3A}). If $A_i = B_i$ for some $i$, then (\ref{Eq.XY3A}) holds trivially as both sides are equal to zero. This concludes the proof of (\ref{Eq.XY3A}), and hence the lemma.

\begin{appendix}
%
%
\section{Edge asymptotics of the GSE}\label{SectionA}
The goal of this section is to prove Lemma \ref{Lem.SAO}. In Section \ref{SectionA.1} we show that the point processes formed by the eigenvalues of the Gaussian Symplectic Ensemble (GSE), converge under edge scaling to a Pfaffian point process with correlation kernel $\kgse$ as in (\ref{Def.Kgse}) and Lebesgue reference measure, see Lemma \ref{Lem.EdgeWeakConvergenceGSE}. The edge asymptotics of the GSE have been extensively studied, see, e.g. \cite{TW05} and \cite[Section 7.6.5]{F10}. As we could not locate the precise statement of Lemma \ref{Lem.EdgeWeakConvergenceGSE} in the literature, we include a short proof of it. In Section \ref{SectionA.2} we give the proof of Lemma \ref{Lem.SAO} by combining Lemma \ref{Lem.EdgeWeakConvergenceGSE} with \cite[Theorem 1.1]{RRV11}.

%
%
\subsection{Point process convergence}\label{SectionA.1} Throughout this section we use freely the definitions and notation regarding point processes from \cite[Section 5]{DY25}.

Suppose that $X^N = (X^N_1, \dots, X^N_N)$ is a random vector with density
\begin{equation}\label{Eq.SAEigenvalueDensity}
f_N(x_1,\dots, x_N) = \frac{{\bf 1}\{x_1 > x_2 > \cdots > x_N \}}{Z_N} \prod_{1 \leq i < j \leq N} (x_i - x_j)^4 \prod_{i = 1}^N e^{-2x_i^2},
\end{equation}
where $Z_N$ is a normalization constant. It is known that $f_N$ is the joint eigenvalue density of the Gaussian Symplectic Ensemble, see \cite[Proposition 1.3.4]{F10}. We also introduce the point processes $M^{\mathrm{GSE};N}$ through
\begin{equation}\label{Eq.GSEPointProcess}
M^{\mathrm{GSE};N}(A) = \sum_{i = 1}^N {\bf 1}\{\tilde{X}_i^N \in A\}, \mbox{ where } \tilde{X}_i^N = 2^{7/6}N^{1/6}\left(X_i^N - (2N)^{1/2}\right).
\end{equation}

With the above notation in place, we can state the main result of the section.
\begin{lemma}\label{Lem.EdgeWeakConvergenceGSE} The measures $M^{\mathrm{GSE};N}$ converge weakly to a Pfaffian point process $M^{\mathrm{GSE};\infty}$, which has correlation kernel $\kgse$ as in (\ref{Def.Kgse}) and Lebesgue reference measure. 
\end{lemma}
\begin{proof}Let $M^N$ be the point process formed by $X_1^N, \dots, X_N^N$. From \cite[(6.55)]{F10} we have that $M^N$ is a Pfaffian point process on $\mathbb{R}$ with Lebesgue reference measure and correlation kernel
\begin{equation}\label{Eq.SAKernelGSEN}
K^N(x,y) = \begin{bmatrix} \int_{x}^{y} S^N_4(x,u) du & S_4^N(x,y) \\ -S_4^N(y,x) & \partial_x S^N_4(x,y) \end{bmatrix}. 
\end{equation}
We mention that in \cite[(6.55)]{F10} the correlation functions are written as quaternion determinants, and to get the above form one needs to use \cite[Proposition 6.1.5]{F10}.

Using \cite[(7.91) and (6.59)]{F10}, we have
\begin{equation}\label{Eq.SASInTermsOfHermites}
\begin{split}
&S_4^N(x,y) = A_4^N(x,y) + B_4^{N}(x,y) \mbox{, where }\\
&A_4^N(x,y) = \frac{e^{-x^2} e^{-y^2}}{\pi^{1/2} 2^{2N+2} (2N)! }  \cdot \frac{H_{2N}(\sqrt{2} x)H_{2N}'( \sqrt{2} y) - H_{2N}'(\sqrt{2}x)H_{2N}(\sqrt{2} y)}{x- y}, \\
&B_4^N(x,y) =  - \frac{e^{-y^2} H_{2N}(\sqrt{2}y) }{\pi^{1/2} 2^{2N+1/2}  (2N-1)!} \int_{\sqrt{2}x}^{\infty} e^{-t^2/2} H_{2N-1}(t)dt.
\end{split}
\end{equation}
In the last equation $H_n(x)$ are the classical Hermite polynomials $H_n(x) = (-1)^n e^{x^{2}}\frac{d^n}{dx^n} e^{-x^2}$. We mention that to obtain (\ref{Eq.SASInTermsOfHermites}) from \cite[(6.59)]{F10} one needs to set $V(x) = \tilde{V}_4(x) = x^2/2$, $p_n(x) = 2^{-n}H_n(x)$ (see \cite[(5.46)]{F10}), $(p_n,p_n)_2 = \pi^{1/2}2^{-n}n!$ (see \cite[(5.48)]{F10}), $c_{2N-1} = (p_{2N}, p_{2N})_2$ (see \cite[(5.65)]{F10}), use the Christoffel-Darboux formula \cite[(5.10)]{F10} 
$$K_{2N}(x,y) = \frac{e^{-x^2/2} e^{-y^2/2}}{(p_{2N-1}, p_{2N-1})_2} \cdot \frac{p_{2N}(x) p_{2N-1}(y) - p_{2N-1}(x) p_{2N}(y)}{x-y},$$
and that $H_n'(x) = 2n H_{n-1}(x)$ (see \cite[(8.952)]{Gradshteyn}).\\

Let $\phi_N(x) = 2^{7/6}N^{1/6}[x - (2N)^{1/2}]$, and observe $M^{\mathrm{GSE};N} = M^N \phi_N^{-1}$. Combining the latter with \cite[Proposition 5.8(5)]{DY25} (for $\phi = \phi_N$) and \cite[Proposition 5.8(6)]{DY25} (for $c_1 = c_2 = 2^{-7/12}N^{-1/12}$), we conclude, that $M^{\mathrm{GSE};N}$ is a Pfaffian point process on $\mathbb{R}$ with Lebesgue reference measure and correlation kernel
$$K^{\mathrm{GSE};N}(x,y) = 2^{-7/6} N^{-1/6} \cdot K^N\left( (2N)^{1/2} + \frac{x}{2^{7/6} N^{1/6}} , (2N)^{1/2} + \frac{y}{2^{7/6}N^{1/6}} \right).$$ 
From \cite[(7.99)]{F10} we get uniformly over bounded $x,y \in \mathbb{R}$ that
\begin{equation}\label{Eq.KernelLimitGSE1}
\lim_{N \rightarrow \infty} K^{\mathrm{GSE};N}(x,y) = K^{\mathrm{GSE};\infty}(x,y) := \begin{bmatrix} \int_x^y S_4(x,u)du & S_4(x,y) \\ -S_4(y,x) & \partial_x S_4(x,y)\end{bmatrix},
\end{equation}
where 
\begin{equation}\label{Eq.LimitS}
S_4(x,y) = \frac{1}{2} \cdot K_{\mathrm{Ai}}(x,y) - \frac{1}{4} \cdot \mathrm{Ai}(y) \int_x^{\infty} \mathrm{Ai}(v)dv.
\end{equation}
Here, $\mathrm{Ai}$ denotes the Airy function
$$\mathrm{Ai}(x) = \frac{1}{2\pi \im}\int_{\mathcal{C}_1^{\pi/3}} e^{z^3/3 - xz},$$
where the contour $\mathcal{C}_1^{\pi/3}$ is as in Definition \ref{Def.S1Contours}, and $K_{\mathrm{Ai}}$ is the Airy kernel
$$K_{\mathrm{Ai}}(x,y) = \frac{\mathrm{Ai}(x)\mathrm{Ai}'(y) - \mathrm{Ai}'(x) \mathrm{Ai}(y)}{x-y}.$$
Equation (\ref{Eq.KernelLimitGSE1}) verifies the conditions of \cite[Proposition 5.10]{DY25} with $\lambda_N = \mathrm{Leb}$ for all $N$, which implies that there exists a Pfaffian point process $M^{\mathrm{GSE};{\infty}}$ that has Lebesgue reference measure and correlation kernel $K^{\mathrm{GSE};\infty}$, and moreover $M^{\mathrm{GSE};N}\Rightarrow M^{\mathrm{GSE};\infty}$. To conclude the statement of the lemma, it remains to show for all $n \geq 1$ and $x_1, \dots, x_n \in \mathbb{R}$ that
\begin{equation}\label{Eq.PfaffianEqalityGSE}
\mathrm{Pf}[K^{\mathrm{GSE};\infty}(x_i,x_j)]_{i,j = 1}^n = \mathrm{Pf}[K^{\mathrm{GSE}}(x_i,x_j)]_{i,j = 1}^n.
\end{equation}

Using the contour integral formula for $\mathrm{Ai}(x)$ from above and the double-contour integral formula for the Airy kernel, see e.g. \cite[(2.30)]{AFM} or just set $s = t$ in (\ref{Eq.S1AiryKer}), 
$$K_{\mathrm{Ai}}(x,y) = \frac{1}{(2\pi\im)^2} \int_{\mathcal{C}_1^{\pi/3}}dw \int_{\mathcal{C}_1^{\pi/3}}dz \frac{e^{z^3/3 + w^3/3 - xz - yw}}{z+w},$$
we obtain that $S_4(x,y)$ equals
\begin{equation*}
\begin{split}
&\frac{1}{2(2\pi\im)^2} \int_{\mathcal{C}_1^{\pi/3}}dw \int_{\mathcal{C}_1^{\pi/3}}dz \frac{e^{z^3/3 + w^3/3 - xz - yw}}{z+w} - \frac{1}{4 (2\pi \im)^2} \int_{\mathcal{C}_1^{\pi/3}}dw e^{w^3/3 - yw} \int_{\mathcal{C}_1^{\pi/3}}dz \int_{x}^{\infty}dv e^{z^3/3 - vz} \\
&=  \frac{1}{2(2\pi\im)^2} \int_{\mathcal{C}_1^{\pi/3}}dw \int_{\mathcal{C}_1^{\pi/3}}dz e^{z^3/3 + w^3/3 - xz - yw} \cdot \left[\frac{1}{z+w} - \frac{1}{2z} \right].
\end{split}
\end{equation*}
We mention that in writing the first line we have implicitly used Fubini's theorem to rearrange the order of the integrals.

The last equation and (\ref{Def.Kgse}) give
\begin{equation}\label{Eq.S4EqualsKgse1}
S_4(x,y) = \kgse_{12}(x,y).
\end{equation}
Using (\ref{Def.Kgse}) and (\ref{Eq.S4EqualsKgse1}) we get
\begin{equation}\label{Eq.S4EqualsKgse2}
\int_x^y S_4(x,u)du = \frac{1}{(2\pi \im)^2}\int_{\mathcal{C}_{1}^{\pi/3}} \hspace{-1.5mm} dz \int_{\mathcal{C}_{1}^{\pi/3}} \hspace{-1.5mm} dw \frac{(z - w ) e^{z^3/3 - xz - w^3/3} (e^{-xw} - e^{-yw})}{4zw(z + w)} = - \kgse_{11}(x,y).
\end{equation}
In the last equality we used that
$$\frac{1}{(2\pi \im)^2} \int_{\mathcal{C}_{1}^{\pi/3}}dz \int_{\mathcal{C}_{1}^{\pi/3}} dw \frac{(z - w ) e^{z^3/3 - xz - w^3/3-xw} }{4zw(z + w)} = 0,$$
which one observes by swapping the roles of $z$ and $w$. Similarly, from (\ref{Def.Kgse}) and (\ref{Eq.S4EqualsKgse1}) we get
\begin{equation}\label{Eq.S4EqualsKgse3}
\partial_x S_4(x,y) = \frac{1}{(2\pi\im)^2} \int_{\mathcal{C}_1^{\pi/3}}dw \int_{\mathcal{C}_1^{\pi/3}}dz \frac{-ze^{z^3/3 + w^3/3 - xz - yw}}{4z(z+w)} = - \kgse_{22}(x,y).
\end{equation}

From (\ref{Eq.S4EqualsKgse1}), (\ref{Eq.S4EqualsKgse2}) and (\ref{Eq.S4EqualsKgse3}) we conclude that 
$$K^{\mathrm{GSE};\infty}(x,y) = \begin{bmatrix} -\kgse_{11}(x,y) & \kgse_{12}(x,y) \\ \kgse_{21}(x,y) & -\kgse_{22}(x,y) \end{bmatrix},$$
which by the definition of the Pfaffian (\ref{Eq.DefPfaffian}) implies (\ref{Eq.PfaffianEqalityGSE}).
\end{proof}

We end this section with the following simple lemma, which is an immediate consequence of \cite[Proposition 2.19]{dimitrov2024airy}. 
\begin{lemma}\label{Lem.FDConvFromPPConvergence} Let $X^N = \{X^N_i\}_{i \geq 1}$ be a sequence of random vectors in $\mathbb{R}^{\infty}$, such that a.s. 
\begin{equation}\label{Eq.OrderedX}
X_1^N \geq X_2^N \geq \cdots.
\end{equation} 
Suppose that $\{X_i^N\}_{N \geq 1}$ is a tight sequence of random variables for each $i \in \mathbb{N}$. In addition, suppose that the random measures $M^N$ on $\mathbb{R}$, defined by
\begin{equation}\label{Eq.PointProcessFromX}
M^N(A)= \sum_{i \geq 1} {\bf 1}\{X_i^N \in A\},
\end{equation}
are point processes, and weakly converge to a point process $M$. Then, $X^N$ converge weakly to some $X^{\infty}$ that satisfies (\ref{Eq.OrderedX}). Moreover, the random measure $M^{\infty}$ as in (\ref{Eq.PointProcessFromX}) with $N = \infty$ is a point process that has the same law as $M$.
\end{lemma}
\begin{proof} Define the random point processes $\tilde{M}^N$ on $\mathbb{R}^2$ by
\begin{equation}\label{Eq.TildeM}
\tilde{M}^N(A) = \sum_{i \geq 1} {\bf 1}\{(0,X_i^N) \in A\}.
\end{equation}
Since $M^N \Rightarrow M$, we conclude $\tilde{M}^N \Rightarrow \tilde{M}$, where $\tilde{M}(A) = M(A_0)$ with $A_0 = \{y: (0,y) \in A\}$. From \cite[Proposition 2.19]{dimitrov2024airy} (applied to $r = 1$ and $t_1 = 0$), we conclude that we can find $X^{\infty} $ satisfying (\ref{Eq.OrderedX}), such that $X^N \Rightarrow X^{\infty}$, and moreover $\tilde{M}^{\infty}$ as in (\ref{Eq.TildeM}) with $N = \infty$ has the same law as $\tilde{M}$. The latter implies $M^{\infty} \overset{d}{=} M$ as desired.
\end{proof}

%
%
\subsection{Proof of Lemma \ref{Lem.SAO}}\label{SectionA.2} Let $\{\tilde{X}_i^N\}_{i = 1}^N$ and $M^{\mathrm{GSE};N}$ be as in (\ref{Eq.GSEPointProcess}). Define $\hat{X}^N = \{\hat{X}_i^N\}_{i \geq 1}$ by setting 
$$\hat{X}_i^N = \tilde{X}_i^N \mbox{ for } i =1,\dots,N, \mbox{ and } \hat{X}_i^N = \min(\tilde{X}_N^N, - i) \mbox{ for } i \geq N+1.$$
We also define the point processes $\hat{M}^N$ on $\mathbb{R}$ through 
$$\hat{M}^N(A) = \sum_{i \geq 1} {\bf 1}\{\hat{X}_i^N \in A\} =M^{\mathrm{GSE};N}(A) + M^{N}_{>N}(A), \mbox{ where } M^N_{>N}(A) = \sum_{ i \geq N+1 }{\bf 1}\{ \hat{X}_i^N \in A \}.$$
From \cite[Theorem 1.1]{RRV11}, we have
\begin{equation}\label{Eq.FDConvSAO}
\hat{X}^N \Rightarrow (-2^{2/3} \Lambda_0, -2^{2/3} \Lambda_1, -2^{2/3} \Lambda_2, \dots ),
\end{equation}
as random vectors in $\mathbb{R}^{\infty}$, or equivalently in the finite-dimensional sense. We mention that the constant $2^{2/3}$ comes from our particular scaling for $\tilde{X}^N_i$ in (\ref{Eq.GSEPointProcess}) as well as the fact that in \cite{RRV11} the authors work with a density that replaces $2x^2$ with $x^2$ in (\ref{Eq.SAEigenvalueDensity}).

From Lemma \ref{Lem.EdgeWeakConvergenceGSE}, we know $M^{\mathrm{GSE};N}\Rightarrow M^{\mathrm{GSE};\infty}$, where the latter is a Pfaffian point process on $\mathbb{R}$ with correlation kernel $\kgse$ as in (\ref{Def.Kgse}) and Lebesgue reference measure. Since $M^{N}_{>N} \Rightarrow 0$ by construction, we conclude $\hat{M}^N \Rightarrow M^{\mathrm{GSE};\infty}$. \\

We are now in a position to apply Lemma \ref{Lem.FDConvFromPPConvergence} to $\hat{X}^N$. Indeed, (\ref{Eq.OrderedX}) is satisfied by the way we defined $\hat{X}^N$, and the previous paragraph shows $\hat{M}^N \Rightarrow M^{\mathrm{GSE};\infty}$. Lastly, the finite-dimensional convergence in (\ref{Eq.FDConvSAO}) in particular implies that $\{\hat{X}^N_i \}_{N \geq 1}$ is a tight sequence for each $i \geq 1$.

From Lemma \ref{Lem.FDConvFromPPConvergence} we conclude that $\hat{X}^N \Rightarrow \hat{X}^{\infty}$ for some $\hat{X}^{\infty}$ and 
$\hat{M}^{\infty} \overset{d}{=} M^{\mathrm{GSE};\infty}$, where
$$\hat{M}^{\infty}(A) = \sum_{i \geq 1} {\bf 1}\{\hat{X}^{\infty}_i \in A\}.$$
As weak limits are unique, we have from (\ref{Eq.FDConvSAO}) that 
$$\hat{X}^{\infty} \overset{d}{=} (-2^{2/3} \Lambda_0, -2^{2/3} \Lambda_1, -2^{2/3} \Lambda_2, \dots ),$$
and so $\hat{M}^{\infty} \overset{d}{=} M^{\mathrm{GSE}}$. Combining the last few statements, we conclude that $M^{\mathrm{GSE}} \overset{d}{=}M^{\mathrm{GSE};\infty}$. This completes the proof.\end{appendix}

\bibliographystyle{amsplain} 
\bibliography{PD}

\providecommand{\bysame}{\leavevmode\hbox to3em{\hrulefill}\thinspace}
\providecommand{\MR}{\relax\ifhmode\unskip\space\fi MR }
\providecommand{\MRhref}[2]{%
  \href{http://www.ams.org/mathscinet-getitem?mr=#1}{#2}
}
\providecommand{\href}[2]{#2}
\begin{thebibliography}{10}

\bibitem{Stegun64}
M.~Abramowitz and I.~Stegun, \emph{Handbook of mathematical functions with
  formulas, graphs, and mathematical tables}, 10th printing ed., National
  Bureau of Standards, Washington, D.C., 1964, Applied Mathematics Series 55.

\bibitem{AFM}
M.~Adler, P.~Ferrari, and P.~Van~Moerbeke, \emph{Airy processes with wanderers
  and new universality classes}, Ann. Probab. \textbf{38} (2010), 714--769.

\bibitem{AH21}
A.~Aggarwal and J.~Huang, \emph{Edge statistics for lozenge tilings of
  polygons, {I}{I}: {A}iry line ensemble}, arXiv:2021:12874v3 (2021).

\bibitem{AH23}
\bysame, \emph{Strong characterization for the {A}iry line ensemble}, Invent.
  Math. (2025), 1--313, https://doi.org/10.1007/s00222-025-01381-6.

\bibitem{BBCSArxiv}
J.~Baik, G.~Barraquand, I.~Corwin, and T.~Suidan, \emph{Pfaffian {S}chur
  processes and last passage percolation in a half-quadrant}, arXiv:1606.00525
  (2016).

\bibitem{BBCS2Arxiv}
\bysame, \emph{Facilitated exclusion process}, arXiv:1707.01923 (2017).

\bibitem{BBCS2}
\bysame, \emph{Facilitated exclusion process}, In: The Abel Symposium (2018),
  1--35.

\bibitem{BBCS}
\bysame, \emph{Pfaffian {S}chur processes and last passage percolation in a
  half-quadrant}, Ann. Probab. \textbf{46} (2018), no.~6, 3015--3089.

\bibitem{BR01a}
J.~Baik and E.M. Rains, \emph{Algebraic aspects of increasing subsequences},
  Duke Math. J. \textbf{109} (2001), no.~1, 1--65.

\bibitem{BR01b}
\bysame, \emph{The asymptotics of monotone subsequences of involutions}, Duke
  Math. J. \textbf{109} (2001), no.~2, 205--281.

\bibitem{BR01c}
\bysame, \emph{Symmetrized random permutations}, Random matrix models and their
  applications, Math. Sci. Res. Inst. Publ. \textbf{40} (2001), 1--19.

\bibitem{BBC20}
G.~Barraquand, A.~Borodin, and I.~Corwin, \emph{Half-space {M}acdonald
  processes}, Forum Math. Pi \textbf{8} (2020), e11.

\bibitem{BBCW18}
G.~Barraquand, A.~Borodin, I.~Corwin, and M.~Wheeler, \emph{Stochastic
  six-vertex model in a half-quadrant and half-line open asymmetric simple
  exclusion process}, Duke Math. J. \textbf{167} (2018), no.~13, 2457--2529.

\bibitem{BC23}
G.~Barraquand and I.~Corwin, \emph{Stationary measures for the log-gamma
  polymer and {K}{P}{Z} equation in half-space}, Ann. Probab. \textbf{51}
  (2023), no.~5, 1830--1869.

\bibitem{BCD24}
G.~Barraquand, I.~Corwin, and S.~Das, \emph{{K}{P}{Z} exponents for the
  half-space log-gamma polymer}, Probab. Theory Relat. Fields (2024), 1--131.

\bibitem{BBNV}
D.~Betea, J.~Bouttier, P.~Nejjar, and M.~Vuleti{\'c}, \emph{The free boundary
  {S}chur process and applications {I}}, Ann. Inst. Henri Poincar{\' e} Probab.
  Stat. \textbf{19} (2018), 2106--2150.

\bibitem{Billing}
P.~Billingsley, \emph{Convergence of {P}robability {M}easures, 2nd ed}, John
  Wiley and Sons, New York, 1999.

\bibitem{BK08}
A.~Borodin and J.~Kuan, \emph{Asymptotics of {P}lancherel measures for the
  infinite-dimensional unitary group}, Adv. Math. \textbf{219} (2008),
  894--931.

\bibitem{BR05}
A.~Borodin and E.M. Rains, \emph{Eynard--{M}ehta theorem, {S}chur process, and
  their {P}faffian analogs}, J. Stat. Phys. \textbf{121} (2005), 291--317.

\bibitem{borodin2015handbook}
A.~Borodin and P.~Salminen, \emph{Handbook of {B}rownian {M}otion -- {F}acts
  and {F}ormulae}, 2nd (corrected reprint) ed., Probability and Its
  Applications, Birkh{\"a}user Basel, Basel, 2015, Corrected reprint of the
  second edition.

\bibitem{cairoli2011}
R.~Cairoli and R.~C. Dalang, \emph{Sequential stochastic optimization}, John
  Wiley \& Sons, 2011.

\bibitem{callahan2010advanced}
J.J. Callahan, \emph{Advanced calculus: a geometric view}, vol.~1, Springer,
  2010.

\bibitem{CU2}
I.~Corwin, \emph{The {K}ardar-{P}arisi-{Z}hang equation and universality
  class}, Random Matrices: Theory Appl. \textbf{1} (2012).

\bibitem{CorHamA}
I.~Corwin and A.~Hammond, \emph{Brownian {G}ibbs property for {A}iry line
  ensembles}, Invent. Math. \textbf{195} (2014), 441--508.

\bibitem{CS18}
I.~Corwin and H.~Shen, \emph{Open {A}{S}{E}{P} in the weakly asymmetric
  regime}, Comm. Pure Appl. Math. \textbf{71} (2018), no.~10, 2065--2128.

\bibitem{DDY26}
S.~Das, E.~Dimitrov, and Z.~Yang, \emph{Pinnning in non-critical half-space
  geometric last passage percolation}, In preparation (2026+).

\bibitem{DasSerio25}
S.~Das and C.~Serio, \emph{The half-space {K}{P}{Z} line ensemble and its
  scaling limit}, Preprint: arXiv:2506.07939 (2025).

\bibitem{DNV19}
D.~Dauvergne, M.~Nica, and B.~Vir{\' a}g, \emph{Uniform convergence to the
  {A}iry line ensemble}, Ann. Inst. Henri Poincare (B) \textbf{59} (2023),
  no.~4, 2220--2256.

\bibitem{DOV22}
D.~Dauvergne, J.~Ortmann, and B.~Vir{\'a}g, \emph{The directed landscape}, Acta
  Math. \textbf{229} (2022), no.~2, 201--285.

\bibitem{dimitrov2024airy}
E.~Dimitrov, \emph{Airy wanderer line ensembles}, arXiv:2408.08445 (2024).

\bibitem{dimitrov2024tightness}
\bysame, \emph{Tightness for interlacing geometric random walk bridges},
  arXiv:2410.23899 (2024).

\bibitem{DEA21}
E.~Dimitrov, X.~Fang, L.~Fesser, C.~Serio, C.~Teitler, A.~Wang, and W.~Zhu,
  \emph{Tightness of {B}ernoulli {G}ibbsian line ensembles}, Electron. J.
  Probab. \textbf{26} (2021), 1--93, DOI: 10.1214/21-EJP698.

\bibitem{DimMat}
E.~Dimitrov and K.~Matetski, \emph{Characterization of {B}rownian {G}ibbsian
  line ensembles}, Ann. Probab. \textbf{49} (2021), no.~5, 2477--2529.

\bibitem{DY25}
E.~Dimitrov and Z.~Yang, \emph{Half-space {A}iry line ensembles},  (2025),
  Preprint: arXiv:2505.01798.

\bibitem{DY25b}
\bysame, \emph{A note on last passage percolation and {S}chur processes},
  (2025), Preprint: arXiv:2510.04713.

\bibitem{FNH99}
P.~J. Forrester, T.~Nagao, and G.~Honner, \emph{Correlations for the
  orthogonalunitary and symplectic-unitary transitions at the hard and soft
  edges}, Nuclear Phys. B \textbf{53} (1999), no.~3, 601--643.

\bibitem{F10}
P.J. Forrester, \emph{Log-{G}ases and {R}andom {M}atrices}, London Math. Soc.
  Monographs, Princeton University Press, 2010.

\bibitem{Gradshteyn}
I.S. Gradshteyn and I.M. Ryzhik, \emph{Table of integrals, series, and
  products}, Academic Press, 2014.

\bibitem{IO19}
F.~Iafrate and E.~Orsingher, \emph{Some results on the {B}rownian meander with
  drift}, J. Theor. Probab. \textbf{33} (2019), 1034--1060.

\bibitem{IMS22}
T.~Imamura, M.~Mucciconi, and T.~Sasamoto, \emph{Solvable models in the
  {K}{P}{Z} class: approach through periodic and free boundary {S}chur
  measures}, arXiv:2204.08420 (2022).

\bibitem{kallenberg2017random}
O.~Kallenberg, \emph{Random measures, theory and applications}, Springer
  International Publishing Switzerland, 2017, 10.1007/978-3-319-41598-7.

\bibitem{Mac94}
A.M.S. Mac{\^e}do, \emph{Universal parametric correlations at the soft edge of
  the spectrum of random matrix ensembles}, Europhys. Lett. \textbf{26} (1994),
  641.

\bibitem{Munkres}
J.~Munkres, \emph{Topology, 2nd ed}, Prentice Hall, Inc., Upper Saddle River,
  NJ, 2003.

\bibitem{Par19}
S.~Parekh, \emph{The {K}{P}{Z} limit of {A}{S}{E}{P} with boundary}, Comm.
  Math. Phys. \textbf{365} (2019), 569--649.

\bibitem{Spohn}
M.~Pr{\" a}hofer and H.~Spohn, \emph{Scale invariance of the {P}{N}{G}
  {D}roplet and the {A}iry process}, J. Stat. Phys. \textbf{108} (2002),
  1071--1106.

\bibitem{RRV11}
J.~Ramirez, B.~Rider, and B.~Vir{\'a}g, \emph{Beta ensembles, stochastic {A}iry
  spectrum, and a diffusion}, J. Amer. Math. Soc. \textbf{24} (2011), no.~4,
  919--944.

\bibitem{RY}
D.~Revuz and M.~Yor, \emph{Continuous {M}artingales and {B}rownian {M}otion},
  third ed., Springer, 1999.

\bibitem{SI04}
T.~Sasamoto and T.~Imamura, \emph{Fluctuations of the one-dimensional
  polynuclear growth model in half-space}, J. Stat. Phys. \textbf{115} (2004),
  749--803.

\bibitem{Sod15}
S.~Sodin, \emph{A limit theorem at the spectral edge for corners of
  time-dependent {W}igner matrices}, Int. Math. Res. Not. \textbf{2015} (2015),
  7575--7607.

\bibitem{SteinFourier}
E.M. Stein and R.~Shakarchi, \emph{Fourier analysis: an introduction}, vol.~1,
  Princeton University Press, 2011.

\bibitem{S90}
J.R. Stembridge, \emph{Nonintersecting paths, {P}faffians, and plane
  partitions}, Adv. Math. \textbf{83} (1990), 96--131.

\bibitem{TW05}
C.~Tracy and H.~Widom, \emph{Matrix kernels for the {G}aussian orthogonal and
  symplectic ensembles}, Annales de l'institut Fourier, vol.~55, 2005,
  pp.~2197--2207.

\bibitem{Wu20}
X.~Wu, \emph{Intermediate disorder regime for half-space directed polymers}, J.
  Stat. Phys. \textbf{181} (2020), no.~6, 2372--2403.

\end{thebibliography}

\end{document}